\documentclass[11pt,reqno]{amsart}
\usepackage{fullpage}
\usepackage[mathscr]{eucal}
\usepackage{mathrsfs}
\usepackage{amsfonts}
\usepackage{amsmath}
\usepackage{amsthm}
\usepackage{amssymb}
\usepackage{latexsym}
\usepackage[all]{xy}
\usepackage{array}

\usepackage[dvipsnames]{xcolor}
\usepackage{graphicx}

\usepackage{mathtools}
\usepackage[bookmarks=false]{hyperref}
\usepackage{hyperref}

\usepackage{graphicx}

\theoremstyle{plain}
\newtheorem{thm}[equation]{Theorem}
\newtheorem{cor}[equation]{Corollary}
\newtheorem{prop}[equation]{Proposition}
\newtheorem{lem}[equation]{Lemma}
\newtheorem{conj}[equation]{Conjecture}

\theoremstyle{definition}

\newtheorem*{defn}{Definition}

\theoremstyle{remark}
\newtheorem{examp}[equation]{Example}
\newtheorem{examps}[equation]{Examples}
\newtheorem{rem}[equation]{Remark}
\newtheorem{rems}[equation]{Remarks}
\newtheorem{claim}[equation]{Claim}

\makeatletter
\renewcommand{\subsection}{\@startsection{subsection}{2}{0pt}{-3ex
plus -1ex minus -0.2ex}{-2mm plus -0pt minus
-2pt}{\normalfont\bfseries}} \makeatother

\numberwithin{equation}{subsection}
\allowdisplaybreaks

\newlength{\dhatheight}

\newcommand{\hdot}{{\:\raisebox{2pt}{\text{\circle*{1.5}}}}}
\newcommand{\idot}{{\:\raisebox{2pt}{\text{\circle*{1.5}}}}}
\newcommand{\hd}{H^{\:\raisebox{2pt}{\text{\circle*{1.5}}}}}
\newcommand{\rst}[1]{\ensuremath{{\mathbin|}\raise-.5ex\hbox{$#1$}}}  
\DeclareMathOperator{\grmod}{\!\text{-}\mathrm{grmod}}

\DeclareMathOperator{\Wh}{\mathrm{Wh}}

\DeclareMathOperator{\Ext}{\mathrm{Ext}}
\DeclareMathOperator{\sym}{\mathrm{Sym}}

\DeclareMathOperator{\im}{\mathrm{Im}}
\DeclareMathOperator{\supp}{\mathrm{Supp}}

\DeclareMathOperator{\Ker}{\mathrm{Ker}}

\DeclareMathOperator{\End}{\mathrm{End}}
\DeclareMathOperator{\cend}{{\mathcal E}\!\textit{nd}} 

\DeclareMathOperator{\rk}{{\mathrm{rk}}\,\g}
\DeclareMathOperator{\gr}{\mathrm{gr}}

\DeclareMathOperator{\Lie}{\mathrm{Lie}}

\DeclareMathOperator{\rep}{\mathsf{Rep}}
\DeclareMathOperator{\Ad}{\mathrm{Ad}}
\DeclareMathOperator{\ad}{\mathrm{ad}}

\DeclareMathOperator{\loc}{_{\op{loc}}}

\DeclareMathOperator{\Spec}{\mathrm{Spec}}
\DeclareMathOperator{\pr}{pr}

\newcommand{\iso}{{\;\stackrel{_\sim}{\to}\;}}

\newcommand{\cd}{\!\cdot\!}

\renewcommand{\mod}{{\,\operatorname{\textsl{mod}}\ }}
\newcommand{\erem}{\hfill$\lozenge$\end{rem}}
\newcommand{\eerem}{\hfill$\lozenge$\end{rem}\vskip 3pt }
\newcommand{\dis}{\displaystyle}

\newcommand{\beq}{\begin{equation}\label}
\newcommand{\eeq}{\end{equation}}
\def\ccirc{{{}_{\,{}^{^\circ}}}}
\DeclareMathOperator{\Hom}{\mathrm{Hom}}

\newcommand{\fe}{{\mathfrak e}}

\newcommand{\zh}{{Z_\hb}}

\newcommand{\BA}{{\mathbb A}}

\newcommand{\uph}{{\bu^\psi_\hb}}
\newcommand{\pt}{{\op{pt}}}
\newcommand{\M}{\textrm{M}}
\newcommand{\sy}{{\op{S}}}

\newcommand{\BG}{{\mathbf G}}
\newcommand{\bg}{{\mathbf G}}
\newcommand{\chb}{{\C[\hbar]}}

\newcommand{\aff}{_{\op{aff}}}
\newcommand{\bfv}{{\mathbf v}}
\newcommand{\ce}{{\mathcal E}}

\newcommand{\T}{{\mathcal T}}
\newcommand{\YY}{{\T^*(G/U)}}
\newcommand{\XX}{{\T^\psi(G/\bU)}}
\newcommand{\opp}{^{op}}
\newcommand{\De}{\Delta }

\newcommand{\vol}{{\operatorname{vol}}}
\newcommand{\resp}{ \quad\text{resp.}\quad }
\newcommand{\jj}{{\mathcal J}}

\newcommand{\bU}{{\bar{U}}}
\newcommand{\bo}{\mbox{$\bigotimes$}}
\renewcommand{\o}{\otimes }
\newcommand{\bplus}{\mbox{$\bigoplus$}}
\newcommand{\ccong}{\ \cong \  }

\newcommand{\bfp}{{\mathbf p}}

\newcommand{\ver}{{\mathsf M}}
\newcommand{\verh}{{\mathsf M}_\hb }

\newcommand{\wt}{\widetilde }
\newcommand{\bt}{{\mathsf T}}

\newcommand{\hc}{{\textsl{hc}}}
\newcommand{\dds}{{{}^{\dagger\!}\dd}}

\renewcommand{\dag}{{{}^{\dagger\!}}}
\newcommand{\hbo}{|_{\hb=0}}
\newcommand{\Om}{\Omega }
\newcommand{\oom}{{\Xi}}
\newcommand{\Id}{{\operatorname{Id}}}

\newcommand{\BX}{{{\mathsf Q}}}
\newcommand{\BM}{{\mathbb M}}
\newcommand{\bd}{{\mathsf B}}
\newcommand{\dsl}{/\!\!/}

\newcommand{\fl}{{}^{\flat\!}}
\newcommand{\fZ}{{\mathfrak{Z}}}
\newcommand{\cg}{{\mathfrak{Z}}}
\renewcommand{\th}{\theta }
\newcommand{\oinn}{{\overset{_\text{in}}\o}}

\newcommand{\out}{\text{out}}

\newcommand{\fm}{{\mathfrak m}}
\newcommand{\fn}{{\mathfrak n}}

\newcommand{\up}{{\bar{\u}^\psi}}
\newcommand{\ddp}{{\dd^\psi(G/\bar U)}}

\def\ccirc{{{}_{^{\,^\circ}}}}

\newcommand{\ssf}{{\mathsf s}}

\newcommand{\bbb}{{\b_\bbf}}
\newcommand{\bu}{{\bar{\mathfrak u}}}
\newcommand{\fii}{{\mathfrak i}}
\newcommand{\fjj}{{\mathfrak j}}
\newcommand{\fr}{{\mathfrak r}}
\newcommand{\fT}{{\mathfrak T}}
\newcommand{\sat}{{{\mathcal P}_{_{\!}}\textit{erv}}_{G^\vee(\OO)}(\Gr)}

\newcommand{\alg}{^{\textrm{alg}}}
\newcommand{\geom}{^{\textrm{geom}}}
\newcommand{\fh}{{\mathfrak h}}

\newcommand{\eh}{{\mathfrak r}}
\newcommand{\cz}{{\mathcal Z}}
\newcommand{\rs}{_{rs}}
\renewcommand{\sc}{^{\mathrm{sc}}}
\newcommand{\vth}{\vartheta }

\newcommand{\ut}{{\mathcal{U}\t}}
\newcommand{\la}{\lambda}

\newcommand{\J}{{\mathsf J}}
\newcommand{\be}{\beta }
\newcommand{\si}{{s}}

\newcommand{\cy}{\C[Y]}
\newcommand{\op}{\operatorname}

\newcommand{\nt}{{\mathsf A}}

\newcommand{\fk}{{\mathfrak k}}

\newcommand{\fa}{{\mathfrak a}}
\newcommand{\fd}{{\mathfrak d}}

\newcommand{\nf}{{\mathfrak N}}

\newcommand{\uub}{{\UU\b}}
\newcommand{\mux}{{\mu\in\BX }\ }

\newcommand{\ppm}{{\{\pm 1\}}}
\newcommand{\sz}{ \sminus \{0\} }

\newcommand{\dd}{{\mathscr{D}}}
\newcommand{\oo}{{\mathcal{O}}}
\newcommand{\wh}{\widehat }

\newcommand{\ug}{{\mathcal{U}}\mathfrak{g}}
\newcommand{\UU}{{\mathcal{U}}}
\newcommand{\U}{{U}}

\newcommand{\mum}{\mu_{\mathcal X} }
\newcommand{\back}{\backslash }

\newcommand{\tg}{{\widetilde \g}}
\newcommand{\tgr}{{\widetilde\g_r}}

\renewcommand{\u}{{\mathfrak u}}

\newcommand{\ude}{{\boldsymbol{\u}}_\tb }
\newcommand{\bde}{{\boldsymbol{\b}}_\tb }
\newcommand{\GO}{{G^\vee(\OO)}}

\newcommand{\wb}{\b^w}

\newcommand{\bbe}{{\mathsf e}}
\newcommand{\bbf}{{\mathsf f}}
\newcommand{\bbh}{{\mathsf h}}

\newcommand{\sll}{{\mathfrak{sl}_2}}

\newcommand{\ba}{{\mathsf p}}
\renewcommand{\t}{{\mathfrak t}}

\newcommand{\vpi}{\varpi}
\newcommand{\tb}{{\wt\bb}}
\newcommand{\tbm}{{{\wt\bb}_-}}
\newcommand{\IC}{{\operatorname{IC}}}
\newcommand{\KK}{{\mathbf K}}
\newcommand{\OO}{{\mathbf O}}
\newcommand{\Gv}{{\check G}}

\newcommand{\Uvm}{{\check U}_-}
\newcommand{\BS}{{\mathbb S}}
\newcommand{\hb}{\hbar }
\newcommand{\uh}{{\mathcal{U}_\hb}}

\newcommand{\Gr}{{\mathsf{Gr}}}

\newcommand{\ogr}{{\overline{{\mathsf{Gr}}}}}
\newcommand{\zf}{{\xi(z)}}

\newcommand{\al}{\alpha }

\newcommand{\fz}{{\mathfrak z}}

\newcommand{\fc}{{\mathfrak{c}}}

\newcommand{\vkap}{\varkappa }

\newcommand{\kap}{\kappa }

\newcommand{\kz}{\wt\kap }

\newcommand{\bbo}{{\mathbb O}}
\newcommand{\bco}{{\mathbb O}_- }
\newcommand{\ooo}{{\mathbb O}}

\newcommand{\xx}{{\mathcal X}}
\newcommand{\wtx}{Z }
\newcommand{\eps}{\epsilon }
\newcommand{\oox}{{\mathcal X}}
\newcommand{\oxx}{{{\mathcal X}}}
\newcommand{\bbm}{\Omega }
\newcommand{\tom}{{\wt\Omega}}

\newcommand{\C}{\mathbb{C}}
\newcommand{\g}{\mathfrak{g}}
\renewcommand{\b}{{\mathfrak{b}}}
\newcommand{\n}{\mathfrak{n}}
\newcommand{\m}{\mathfrak{m}}

\newcommand{\bb}{{\mathcal B}}

\newcommand{\mm}{{\mathcal M}}

\newcommand{\Gm}{{{\mathbb G}_m}}

\newcommand{\inv}{^{-1}}
\newcommand{\reg}{_{r}}
\newcommand{\cf}{{\mathcal F}}
\newcommand{\Z}{{\mathbb Z}}
\newcommand{\en}{{\enspace}}

\newcommand{\vi}{${\en\sf {(i)}}\;$}
\newcommand{\vii}{${\;\sf {(ii)}}\;$}
\newcommand{\viii}{${\sf {(iii)}}\;$}
\newcommand{\iv}{${\sf {(iv)}}\;$}
\newcommand{\vv}{{\mathbf v}}

\newcommand{\sset}{\subseteq}
\newcommand{\sminus}{\smallsetminus}
\newcommand{\intoo}{\,\xymatrix{\ar@{^{(}->}[r]&}\,}
\newcommand{\ontoo}{\,\xymatrix{\ar@{->>}[r]&}\,}
\newcommand{\into}{\,\hookrightarrow\,}
\newcommand{\too}{\,\longrightarrow\,}
\newcommand{\mto}{\mapsto}
\newcommand{\onto}{\,\twoheadrightarrow\,}

\newcommand{\Ga}{\Gamma }
\newcommand{\gm}{\Gm }%{\C^\times}}

\newcommand{\om}{\omega }

\begin{document}
\title{Differential operators on $G/U$ and the Gelfand-Graev action}
\author{Victor Ginzburg}
\address{V.G.:
Department of Mathematics, University of Chicago,  Chicago, IL 
60637, USA.}
\email{ginzburg@math.uchicago.edu}
\author{David Kazhdan}
\address{D.K.:
Einstein Institute of Mathematics, Hebrew University,
Givat Ram, Jerusalem 91904, Israel}
\email{kazhdan@math.huji.ac.il}
%\dedicatory{To Volodya Drinfeld on the occasion of his 65th Birthday}
\maketitle

\begin{abstract} Let $G$ be  a complex semisimple group and $U$ its maximal unipotent
subgroup. We study
the algebra $\dd(G/U)$ of algebraic differential operators on $G/U$ and also  its quasi-classical 
counterpart:
the algebra of regular functions on  $\T^*(G/U)$, the cotangent bundle. A long time
ago, S. Gelfand and M. Graev have constructed an action of the Weyl group on  $\dd(G/U)$ by algebra
automorphisms. The Gelfand-Graev construction was not algebraic, it involved analytic methods
in an essential way. We give  a new algebraic construction of the  Gelfand-Graev
action, as well as its quasi-classical counterpart. Our approach is based on Hamiltonian reduction
and involves the ring of Whittaker differential operators on $G/U$, a twisted analogue of  $\dd(G/U)$.

Our main result  has an interpretation, via geometric Satake,  in terms of spherical perverse sheaves
on the affine Grassmanian for the Langlands dual group.
\end{abstract}

%Key words: Hamiltonian reduction, differential operators, semisimple groups.

{\small
\tableofcontents
}

%sss

\section{Introduction}
\subsection{}\label{ss1.1}
Let $G$ be a complex  connected semisimple group and $U$ a maximal unipotent subgroup
of $G$. The ring $\dd(G/U)$  of algebraic differential
operators on  $G/U$  has a rich structure which was analyzed in
\cite{BGG} and studied further more recently in
\cite{BBP}, \cite{LS}, and \cite{GR}. In an  unpublished
paper written in the 1960's,
S. Gelfand and  M. Graev have constructed
 an action of the  Weyl group $W$ on $\dd(G/U)$
by algebra automorphisms. This action is somewhat mysterious 
 due to the fact that it does not come from a 
$W$-action on the variety $G/U$ itself.
 In  rank 1,
the  action  of the nontrivial element 
of $W=\Z/2\Z$ is, essentially,  the Fourier transform
 of polynomial  differential
operators on a 2-dimensional vector space. In the case of higher rank,
the action of  each individual simple reflection is defined by reducing 
to a rank 1 case, but it is not  {\em a priori} clear
 that the resulting  automorphisms
of   $\dd(G/U)$ satisfy the Coxeter relations.
For a   proof  (essentially, due to Gelfand and Graev) that involves analytic
arguments  the reader is referred to
\cite[Proposition 6]{BBP}, cf. also  \cite{BP}, \cite{Ka}, \cite{KL} for closely
related 
results.

One of the goals of the present paper is to provide a different approach
to the  Gelfand-Graev action. Specifically, we will present the algebra
$\dd(G/U)$ as a quantum Hamiltonian reduction (see Theorem \ref{mthm})  in such a way that
the $W$-action on the algebra becomes manifest.

To explain this, recall the general setting of quantum Hamiltonian reduction.
Let $A$ be an associative ring and $I$ a left ideal  of
$A$ (there is also a counterpart of the construction for right ideals). Thus, $A/I$ is a  left $A$-module.
The  quantum Hamiltonian reduction of $A$ with respect to $I$
is defined to be $\big(\End_A\,A/I\big)^{op}$, an opposite of the associative ring of $A$-module
endomorphisms of $A/I$.
More explicitly, let $N(I)=\{a\in A\mid Ia\sset I\}$
be the {\em normalizer} of $I$ in $A$.
By construction, $N(I)$ is a subring of $A$ such that $I$ is a
two-sided
ideal of $N(I)$.  For any  $a\in N(I)$,  the  assignment
 $f_a: x\mto xa$  induces an  endomorphism of $A/I$. Moreover, 
this endomorphism  only depends on $a\mod I$ and one has a ring  isomorphism
\beq{qhr}
\big(\End_A\,A/I\big)^{op}\xleftarrow[\cong]{f_a\leftarrowtail a} N(I)/I\ =\
 \{a\in A\mid (xa-ax)\
\mod  I=0\ \  \forall x\in I\}/I.
\eeq

The following  special cases of quantum Hamiltonian reduction will be especially important for ~us.
\begin{examps}\label{examps}
\vi 
Let $\fk$  be a Lie algebra
and $\imath:\fk\to A$ a Lie algebra map into an associative
algebra $A$, i.e., a linear map such that $\imath([x,y])=\imath(x)\imath(y)-\imath(y)\imath(x),\
\forall x,y\in\fk$.
We let $I=A\fk $ be a left ideal of $A$ generated
by the image of $\imath$.
In this case, we have $N(I)/I=(A/A\fk )^\fk$, the centralizer of
$\imath(\fk)$ in $A/A\fk $.
For any left, resp. right, $A$-module $M$, the right action of 
$(A/A\fk)^\fk$ on $A/A\fk$
induces a left, resp. right,  action of $(A/A\fk)^\fk$ on the space
$\Hom_A(A/A\fk, M)=M^\fk$ of $\fk$-invariants,
 resp. space $M/M\fk=M \o_A (A/A\fk)$ of $\fk$-coinvariants.
Similarly, one can consider a right ideal,
$\fk A$,  generated
by the image of $\imath$ and the corresponding algebra $(A\back\fk A)^\fk$.

\vii Let  $A_1,A_2,Z$, be a triple
of associative rings
and $\iota_i: Z\to A_i,\ i=1,2$,  a pair of ring homomorphisms.
Let $I$ be a right ideal of $A_1^{op}\o A_2$
generated by the elements $\iota_1(z)\o 1-1\o\iota_2(z)$, $z\in Z$.
Then, we have $(A_1\opp\o A_2)/I=A_1\o_Z A_2$.
Therefore, in this case, we obtain
\beq{NI}N(I)/I=\{a_1\o a_2\in A_1\o_Z A_2\mid
\iota_1(z)a_1\o a_2=a_1\o a_2\iota_2(z),\ \forall z\in Z\}=(A_1\o_Z A_2)^Z,
\eeq 
where for any $Z$-bimodule $M$ we put $M^Z:=\{m\in M\mid zm=mz,\ \forall z\in Z\}$.
Multiplication in the ring \eqref{NI} reads as follows:
\[(a_1\o a_2)\cdot (a_1'\o a_2')= (a_1'\cdot a_1)\o (a_2\cdot a_2'),\qquad\forall a_1,a_1'\in A_1,\
a_2,a_2'\in A_2.
\]

By construction, the space $(A_1\opp\o A_2)/I=A_1\o_Z A_2$ comes equipped with
the structure of a left $(A_1\o_Z A_2)^Z$-module given by the `inner' action
$a_1\o a_2:\, (a_1'\o a_2')\mto (a_1'\cdot a_1)\o (a_2\cdot a_2')$,
and the structure of a right $A_1^{op}\o A_2$-module given by the `outer'
action $a_1\o a_2:\, (a_1'\o a_2')\mto (a_1\cdot a_1')\o (a_2'\cdot a_2)$.

Let $E$ be  an $(A_1, A_2)$-bimodule, resp.   $(A_2, A_1)$-bimodule.
There are natural isomorphisms
\begin{align*}
\op{HH}^\hdot(Z,E)&:=
                    \Ext^\hdot_{Z\o Z^{op}}(Z,E)\ccong\op{Ext}^\hdot_{A_1^{op}\o A_2}(A^{op}_1\o_Z A_2,\, E),\en\text{resp.}\\
  \op{HH}_\idot(Z,E)&:=\op{Tor}^{Z\o Z^{op}}_\idot(Z,E)\ccong\op{Tor}^{A_1^{op}\o A_2}_\idot(A^{op}_1\o_Z A_2,\, E).
\end{align*}
Here, $\op{HH}^\hdot(Z,-)$, resp. $\op{HH}_\idot(Z,-)$, stands for Hochschild cohomology,
resp. homology.
Using the above  isomorphisms and the isomorphism in \eqref{qhr}, we see that
each of the groups $\op{HH}^i(Z,E)$, resp. $\op{HH}_i(Z,E)$,
has the natural structure of a right, resp. left, $(A_1\o_Z A_2)^Z$-module.
\end{examps}

\subsection{}\label{ss1.2}
Throughout the paper,  we work over $\C$ and use the notation $\sym \fk$, resp. $\UU\fk$, 
for the symmetric, resp. enveloping, algebra  of a vector space,
resp.   Lie algebra, $\fk$.  For any scheme $X$ put $\C[X]:=\Ga(X,\oo_X)$.
Let  $\T^*X$, resp. $\dd_X$ and $\dd(X)$, denote the cotangent bundle,
resp. the sheaf and ring of algebraic differential operators, on a smooth algebraic variety $X$.

Let $G$ be a complex semisimple group with trivial center, and
 $U, \bar{U}$,  a pair of opposite maximal unipotent subgroups of $G$.
Let $\g$, resp. $\u,\bar{\u}$, denote
the Lie algebra of $G$, resp. $U,\bar{U}$.
We fix
a nondegenerate character $\psi:\bu\to\C$, i.e., such that $\psi(f_i)\neq0$ for every simple root
vector $f_i\in \bu$.

The action of $G$ on itself by right translations gives a Lie
algebra map $\g\to \dd(G)$. Let $\imath$, resp. $\bar\imath$,
denote its restriction to the subalgebra $\u$, resp. $\bu$.
It is well known, cf. 
e.g. \cite[\S3.1]{GR}, that using the notation of
Example \ref{examps}(i), one has $\dd(G/U)\cong (\dd(G)/\dd(G)\u)^{\u}$.
Let
$\up$ be the image of the map $\bu\to\dd(G),\ x\mto \bar\imath(x)-\psi(x)$.
The algebra  of {\em Whittaker differential operators} on $G/\bU$ is defined as
a quantum hamiltonian reduction
$\dd^\psi(G/\bar{U}):=(\dd(G)/\dd(G)\up)^{\up}$, cf. also \S\ref{tors-sec}.
The differential of the action of $G$ on itself by left translations induces an algebra homomorphism
$i:\ug\to \dd(G/U)$, resp. $i^\psi: \ug\to \dd^\psi(G/\bar{U})$.

Let $T$  be the abstract maximal torus of $G$,
and $\t=\Lie T$. We have an imbedding
 $i_T: \ut\into \dd(T)$ as
the subalgebra of translation invariant differential operators.
There is a natural $T$-action on $G/U$ by right translations.
The differential of this
action  induces an algebra homomorphism
$i_r: \ut\to  \dd(G/U)$.

Let $W$ be the (abstract) Weyl group, $Z\g$ the center of the algebra $\ug$,
and $\hc: Z\g\iso(\UU\t)^W$ the Harish-Chandra isomorphism,
where $W$-invariants are taken with respect to the `dot-action' of $W$ on $\UU\t$.
One has the following diagram of algebra homomorphisms:

\beq{diag-q}
\xymatrix{
\dd^\psi(G/\bar{U})\ &\ \ug\ \ar[l]_<>(0.5){i^\psi}
 &\ Z\g\ \ar@{_{(}->}[l]\ar[r]^<>(0.5){\hc}_<>(0.5){\cong} &\ (\ut)^W\
 \ar@{^{(}->}[r]&
\ \ut\ \ar@{^{(}->}[r]^<>(0.5){i_T}& \dd(T).
}
\eeq
Let $\iota_1:  Z\g\to \dd(T)$,  resp.  $\iota_2: Z\g\to \dd^\psi(G/\bar{U})$, 
be the composite homomorphism on the right, resp. left, of   \eqref{diag-q}.
We apply the construction of Hamiltonian reduction in the setting of Example \ref{examps}(ii) for
the triple $A_1=\dd(T),\, A_2=\dd^\psi(G/\bar{U}),\, Z=Z\g$,
and the homomorphisms $\iota_1,\iota_2$.

With the above notation, the main result of the paper reads as follows.

\begin{thm}\label{mthm}
There is a natural  algebra isomorphism
\beq{zz}\dd(G/U)\ \cong\ \big(\dd(T)\ \bo_{Z\g}\ \dd^\psi(G/\bar{U})\big)^{Z\g},
\eeq
such that the map $i_r$, resp. $i$,
corresponds via \eqref{zz} to the map
$u\mto i_T(u)\o1$, resp. ${u\mto 1\o i^\psi(u)}$.
\end{thm}

The  Weyl group acts  on  the RHS of  \eqref{zz}
via its natural action on 
$\dd(T)$, the first tensor factor.
Thanks to  the theorem,
one can transport the $W$-action  on  the RHS of \eqref{zz} to the LHS.
We obtain a $W$-action on  $\dd(G/U)$ by algebra automorphisms.
One can check, although we will not do that in the present paper, that the $W$-action 
thus obtained is the same as the  Gelfand-Graev action (it is sufficient to check this for
simple reflections, which reduces to a rank one computation). 

\begin{rem} In the main body of the paper we explain how  to define the varieties $G/U, G/\bU$, as well as the
ring $\dd^\psi(G/\bU)$,  in a canonical way that involves neither the choice of a pair of opposite unipotent
subgroups $U,\bU$, see \S\ref{inv},
nor a choice of nondegenerate character $\psi$,
see \S\S\ref{tors-sec}-\ref{gtors}.  The  isomorphism in \eqref{zz} then becomes canonical.
Therefore, our construction of the  Gelfand-Graev action yields a canonical $W$-action,
while the original construction of Gelfand and Graev depends on the choice of a pinning on $G$.
\erem

 Theorem \ref{mthm} combined with
the  observation at the end of Example \ref{examps}(ii) and the fact that each of the algebras
$\dd(G/U)$ and $\dd(T)$ is isomorphic to its opposite, gives the following result:
\begin{cor}\label{tor} For each  $i\geq0$,  the assignment
$E\mto  \op{HH}_i(Z\g,\, E)$, resp.
$E\mto \op{HH}^i(Z\g,\, E)$, gives a functor
from the category
of $(\dd^\psi(G/\bar{U}),\,\dd(T))$-bimodules
to  the category
of $\dd(G/U)$-modules.
\end{cor}

\begin{rem} It is not difficult to show that each of the two algebras
 $\dd(T)$ and $\dd^\psi(G/\bar{U})$ is flat as (either left or right)  $Z\g$-module.
It follows that $H^i\big(\dd(T)\, \overset{L}{\o}_{Z\g}\,\dd^\psi(G/\bar{U})\big)=0$ for all $i>0$.
It
might be interesting to find Hochschild cohomology
groups $\op{HH}^i\big(Z\g,\ \dd(T)\, \o_{Z\g}\, \dd^\psi(G/\bar{U})\big)$ for  $i>0$.
\end{rem}
\subsection{}
Theorem \ref{mthm} has a `quasi-classical' counterpart.
In more detail, 
write $V^\perp\sset\g^*$ for the annihilator of a vector subspace $V\sset \g$ and
let
$ \psi+\bu^\perp:=\{\phi\in\g^*\mid \phi|_\bu=\psi\}$,
 an affine linear subspace of $\g^*$.
Let
\beq{yyy}\YY=G\times_U\u^\perp\resp
\XX:=G\times_\bU(\psi+\bu^\perp).\eeq
Here, the equality on the left is a standard
isomorphism of $G$-equivariant vector bundles on $G/U$; 
the equality on the right is our definition
of $\XX$. The variety $\XX$ is a twisted cotangent bundle on $G/\bU$
that may be thought of as a deformation of
$\T^*(G/\bU)$. In particular, $\XX$ comes equipped with a
natural symplectic structure such that 
the map $\XX\to G/\bU$ is a  $G$-equivariant affine bundle on $G/\bU$
with Lagrangian fibers. 
Further, the $G$-action on $\YY$, resp. $\XX$, is Hamiltonian with
moment map $\mu$, resp. $\mu^\psi$. 
The $T$-action on $G/U$ by right translations
induces a Hamiltonian action on
$\T^*(G/U)$.
Let  $\mu_r: \T^*(G/U)\to\t^*$, resp. $\mu_T:\T^*T \to\t^*$,
be the moment map for the $T$-action on $\T^*(G/U)$, resp.
$\T^*T $.
Finally, let $\g^*\to\fc:=\g^*\dsl G=\Spec(\C[\g^*]^G)$
be the (co)adjoint quotient morphism
and identify $\fc$ with $\t^*/W$ via the Chevalley isomorphism.

A  quasi-classical counterpart of diagram \eqref{diag-q}
reads as follows:
\[
\xymatrix{
\T^\psi(G/\bU)\ \ar[rr]^<>(0.5){\mu^\psi}&&\  \g^*\  \ar@{->>}[r] &\
\fc\ =\ 
%&\ar@{=}[r] &\
\t^*/W\ &\  \t^*\ \ar@{->>}[l] &&
\ \T^* T.\ \ar[ll]_<>(0.5){\mu_T}
}
\]

Let $p_1: \T^*T\to \fc$, resp. $p_2: \T^\psi(G/\bU)\to \fc$, be the composite
map on the right, resp. left, of the diagram.
Further, let $\fc_\De\sset \fc\times\fc$ denote the diagonal
and
$\T ^* T\times_{\fc}\XX:=(p_1\times p_2)\inv(\fc_\De)$.

%Thus, we have an algebra homomorphism
%\[ \C[\fc]\o\C[\fc]=
%\C[\fc\times\fc]\ \xrightarrow{(p_1\times p_2)^*}\  \C[\T ^* T\times\XX].
%\]

We equip  $\T ^* T\times\XX$ with the natural symplectic structure of a cartesian product,
where the symplectic form  on the
first  factor is taken
with a negative sign.
It turns out 
that $\T ^* T\times_{\fc}\XX$
is a smooth coisotropic subvariety
of  $\T ^* T\times\XX$; furthermore, the 0-foliation on this subvariety
is generated by Hamiltonian vector fields associated with the
functions $\wt z:=p_1^*(z)-p_2^*(z),\, z\in \C[\fc]$.
It follows that there is a well defined
pairing
\[\C[\fc]\times\C[\T ^* T\times_{\fc}\XX]\to \C[\T ^* T\times_{\fc}\XX],\,
 \en (z,f)\mto \{\wt z,f\},
\]
induced by the Poisson bracket.
Moreover, the Poisson centralizer

\beq{p-var}\C[\T ^* T\times_{\fc}\XX]^{\C[\fc]}\ :=\
\big\{f\in \C[\T ^* T\times_{\fc}\XX]\ \, \big|\ \,  \{\wt z,\,f\}=0
\en\forall z\in\C[\fc]\big\}
\eeq
inherits 
the structure of a Poisson algebra, cf. also Section \ref{fZ-sec}.

The  Poisson algebra in \eqref{p-var} is a quasi-classical counterpart of  \eqref{NI}.
The  quasi-classical   counterpart of Theorem \ref{mthm} reads as follows.
\begin{thm}\label{thm*} There is a natural $G\times T$-equivariant Poisson algebra isomorphism
\[ \C[\YY]\ \cong\ 
\C[\T ^* T\times_{\fc}\XX]^{\C[\fc]},
\]
such that the pull-back $\mu^*_r: \C[\t^*]\to \C[\YY]$, resp.
$\mu^*: \C[\g^*]\to \C[\YY]$, corresponds
via the isomorphism, to the map
$f\mto \mu_T^*(f)\o 1$, resp.
 $f\mto 1\o (\mu^\psi)^*(f)$.
\end{thm}

The natural Weyl group action  on $\T^*T$ induces an action on
$\C[\T^*T \times_{\t^*/ W}\XX]$.
The subalgebra
$\C[\T ^* T\times_{\t^*/ W}\XX]^{\C[\fc]}$
is $W$-stable.
Transporting the $W$-action 
via the isomorphism of Theorem \ref{thm*} yields a  $W$-action on $\C[\YY]$
by Poisson algebra automorphisms. 
Thanks to \cite{GR}, the algebra
 $\C[\T^*(G/U)]$ is   finitely generated. Therefore,  $\YY\aff:=\Spec\C[\T^*(G/U)]$, 
 the affinization of
$\T^*(G/U)$, is an affine variety that  comes equipped with 
a  Poisson structure and  a Hamiltonian $G\times T$-action. From Theorem \ref{thm*},  
we deduce

\begin{cor}\label{Wcor} There exists a   $W$-action on $(\T^*(G/U))\aff$
with the following properties:

\vi The actions of $W$ and $T$ combine together to give a
$W\ltimes T$-action on  $(\T^*(G/U))\aff$ such that 
the moment map $(\T^*(G/U))\aff\to\t^*$, for  the  $T$-action, is a $W$-equivariant morphism. 

\vii The $G$-action commutes with the $W\ltimes T$-action.
\end{cor}

The   $W$-action on $(\T^*(G/U))\aff$
 is a quasi-classical counterpart of
the Gelfand-Graev action. A different construction of
the same  $W$-action on $(\T^*(G/U))\aff$ was given earlier in
\cite[Proposition ~5.5.1]{GR}.

\begin{rems} \vi The $W$-action on $(\T^*(G/U))\aff$
does not commute with the $\C^\times$-action that comes from the dilation action
on the fibers of $\T^*(G/U)\to G/U$.

\vii Analogues of Theorems \ref{mthm}  and \ref{thm*} are likely to hold for any connected semisimple group
$G$, not necessarily of adjoint type. The case of a simply connected group will be  discussed
in \S\ref{sc_sec}.
\end{rems}

The geometry of $\XX$ is, in a way,
much simpler than that of $\YY$. Indeed, the variety
$\XX$ is affine,  the  $G$-action on $\XX$ is free,
and the  map $\mu^\psi$ is a smooth morphism
with image $\g_r$, the set of regular (not necessarily semisimple) elements
of $\g^*$. On the other hand, the variety
  $\T^*(G/U)$ is only quasi-affine, the
$G$-action on  $\T^*(G/U)$ is free only generically, the
 map $\mu: \T^*(G/U)\to \g^*$ is not flat and its image
is the whole of $\g$.

The Poisson variety $\YY\aff$ is, we believe, quite interesting and it
deserves further study. 

In the special case of  type $A$, the variety $\YY\aff$ has a quiver construction,
\cite{DKS}, 
analogous to Nakajima's construction of $T^*(G/B)$, \cite{Na}.
In this case, a construction of the Gelfand-Graev action
which is similar to the construction, due to Lusztig and Nakajima,
of Weyl group actions on quiver varieties was found in \cite{W}.

We propose the following

\begin{conj}\label{Be} For any semisimple group $G$,
  the variety $\YY\aff$ has symplectic singularities, \cite{Be}; in particular,  $\YY\aff$
  is a union of finitely many symplectic leaves,  \cite{Kal}.
\end{conj}

Using the  quiver interpretation 
it was shown
in \cite{J} that Conjecture \ref{Be} holds in type $A$.

In the case $G=SL_3$
(which is not  a group of adjoint type)
 the variety $(\T^*(G/U))\aff$ is well understood, see  \cite{J}.
 Specifically, it
 is isomorphic to the closure, $\bar O$, of the minimal nilpotent
orbit $O$ in $\mathfrak{so}_8$.
The Gelfand-Graev action of the Weyl group $W=S_3$
on $(\T^*(G/U))\aff\cong \bar O$ comes from {\em triality},
and the Poisson structure on $(\T^*(G/U))\aff$ agrees 
 with the Kirillov-Kostant
 symplectic structure on $O$.
 Since $\bar O=O\sqcup \{0\}$,
there are  two symplectic leaves:
$\{0\}$ and $O$.
It is known that  $\bar O$ 
does not have a symplectic resolution, \cite{Fu}.

\subsection{Interpretation via the affine Grassmannian}
\label{satake-intro}
Let $\KK=\C(\!(z)\!)$, resp. $\OO=\C[[z]]$.
Let
$G^\vee$ be the Langlands dual group of $G$ 
and  $\Gr=G^\vee(\KK)/G^\vee(\OO)$  the affine Grassmannian.
Since $G$ is of adjoint type,
the group $G^\vee$ is simply-connected, so $\Gr$ is connected.
The group $G^\vee$ comes equipped with
 the canonical maximal torus $T^\vee$. Let
 $\bt=\Gm\times T^\vee$. %vvv \times\C^\times$.
The group $\bt$ acts on $\Gr$, where the factor $\Gm$ acts by loop
rotation. Let $\Gr^\bt$ be the $\bt$-fixed point set.
There is a canonical bijection  $\Gr^\bt\cong \BX$, where $\BX$ is the root lattice of $G$.
We write $\pt_\la$ for the $\bt$-fixed point corresponding to $\la\in\BX $,  and
$i_\la:\ \{\pt_\la\}\into \Gr$ for the  imbedding. Let $\C_X$ denote a constant sheaf on a space $X$.
The restriction map
$i_\la^*: \hd_\bt(\Gr)=\hd_\bt(\Gr, \C_\Gr)\to\hd_\bt(\pt_\la)$
is a surjective algebra homomorphism. Let $\jj_\la$ be its kernel,
an ideal of $ \hd_\bt(\Gr)$.

The equivariant cohomology $\hd_\bt(\cf)=\hd_\bt(\Gr,\cf)$,
of  a $\bt$-equivariant constructible complex $\cf$  on $\Gr$,
has the natural structure of a $\Z$-graded
 $\hd_\bt(\Gr)$-module.
For such an $\cf$, we put
$\hd_\bt(\cf)^{\jj_\la}:=$ ${\{h\in \hd_\bt(\cf)
\mid jh=0,\ \forall j\in\jj_\la\}}$.

Let $\sat$ be the Satake category.
Any object of that category is known to be a finite direct sum of the
IC-sheaves  $\IC_\la:=\IC(\ogr_\la)$, where $\ogr_\la$ denotes 
 the closure of the  $\GO$-orbit of $\pt_\la$.
Furthermore, it is known that 
objects of $\sat$ come equipped with a canonical
$\gm\ltimes \GO$-equivariant, in particular
$\bt$-equivariant,  structure.

From  Theorem \ref{vthm}
of the present paper combined with some results from
\cite{BF} and \cite{GR}, we will show in section \ref{sec6}
that  the following theorem
is essentially equivalent, via geometric Satake, to 
a combination of Theorems \ref{mthm} and \ref{thm*}.

\begin{thm}\label{grthm} For any $\cf\in\sat$ and $\la\in
\BX $,  the  adjunction  $(i_\la)_!i_\la^!\cf\to\cf$ induces  an isomorphism
\[\hd_\bt(i_\la^!\cf)\ \iso\ \hd_\bt(\cf)^{\jj_\la}.
%:=\{h\in \hd_\bt(\cf) \mid jh=0,\ \forall j\in\jj_\la\}.
\]
\end{thm}

\begin{rems}\label{obstacle}
\vi 
Let $\mu,\la\in\BX$ be such that $\pt_\la\in \ogr_\mu$,
let $i_{\la,\mu}: \{\pt_\la\}\into
\ogr_\mu$ be the imbedding, and  put $\op{C}_{\la,!}:=(i_{\la,\mu})_!\C_{\pt_\la}$.
Write  $\Ext^\hdot_\bt(-,-)$  for  Ext-groups in the $\bt$-equivariant constructible
derived category of $\ogr_\mu$. 
For any $\cf$ in that category, there is a  canonical isomorphism
$\hd_\bt(i_{\la,\mu}^!\cf)=\Ext^\idot_\bt(\op{C}_{\la,!},\, \cf)$.
On the other hand, since $\hd_\bt(\op{C}_{\la,!})$ and
$\hd_\bt(\Gr)/\jj_\la$ are isomorphic $\hd_\bt(\Gr)$-modules, we find
$\hd_\bt(\cf)^{\jj_\la}=\Hom_{H_\bt^\hdot(\Gr)}(\hd_\bt(\Gr)/\jj_\la,\,\hd_\bt(\cf))=
\Hom_{H_\bt^\hdot(\Gr)}(\hd_\bt(\op{C}_{\la,!}),\,\hd_\bt(\cf))$.
Thus,
 the isomorphism of the theorem 
amounts to the claim that, for all pairs $\la,\mu$,
as above, the functor $\hd_\bt(-)$ induces  an isomorphism 
\beq{griso}
\Ext^\idot_\bt(\op{C}_{\la,!},\ \IC_\mu)\ \iso\
\Hom_{H_\bt^\hdot(\Gr)}(\hd_\bt(\op{C}_{\la,!}),\ \hd_\bt(\IC_\mu)).
\eeq

 \vii A non-equivariant analogue of the above isomorphism, hence
 of Theorem \ref{grthm}, is  false, in  general. Indeed,
the groups  $\hd(\op{C}_{\la,!})$ and $\hd(\op{C}_{\nu,!})$ are 
isomorphic $\hd(\Gr)$-modules, for any pair of points
$\pt_\la,\pt_\nu\in\ogr_\mu$. On the other hand,
 the  $\hd(\Gr)$-modules $\hd(i_\la^!\IC_\mu)$ and
$\hd(i_\nu^!\IC_\mu)$ 
are clearly not necessarily
 isomorphic, in  general.

\viii Isomorphism \eqref{griso}  is reminiscent of a result from \cite{Gi1}.
Specifically,  according to \cite{Gi1} there is
 an analogue of  isomorphism \eqref{griso}
where the sheaf $\op{C}_{\la,!}$ is replaced by $\IC_\la$.
However, unlike the main result of  \cite{Gi1} which holds in a much more
general setting,  isomorphism 
\eqref{griso}  seems to be
a  special feature of the Satake category.
It is unlikely that an analogue of \eqref{griso}
holds for IC-sheaves of   Schubert varieties
in an arbitrary (finite dimensional, say)  partial flag variety.

\iv In the setting of Theorem \ref{grthm}, there is a natural map $\hd_\bt(\cf)/\jj_\la\cdot\hd_\bt(\cf)
\to \hd_\bt(i_\la^*\cf)$, induced by the adjunction $\cf\to (i_\la)_*i_\la^*\cf$.
This map is not necessarily injective, in general.
\end{rems}

\subsection{Layout of the paper} 
In section \ref{inv} we review various (well-known) constructions which allow,
in particular, to define the $G$-varieties $G/U$ and $G/\bU$ in a way that  does not involve a
choice of unipotent subgroups $U,\bU$.
In section 3 we introduce a certain torsor of the group scheme of regular centralizers.
The same torsor has appeared in a less explicit way in the work of Donagi and Gaitsgory \cite{DG}.
In section \ref{4} we  study  the {\em Miura variety}, a  smooth  $G$-stable Lagrangian
correspondence in $\T^*(G/U )\times \XX$.
This variety comes equipped with commuting actions of the Weyl group and 
 the  group scheme of regular centralizers. The Miura variety is the main geometric
ingredient of the proof of 
Theorem \ref{thm*} given in section \ref{5}.

The goal of section \ref{6} is to construct a map $\kap$,  cf. \eqref{kap}, between the algebras in
the LHS and RHS
of \eqref{zz}. Theorem \ref{kaph-thm}, which is a more precise version of Theorem
\ref{mthm}, states that the constructed map is an algebra isomorphism.
 A key role in
the construction of $\kap$
 is played by a certain $\dd(G/U)\o\dd(T)\o\dd^\psi(G/\bU)$-module which we call the  {\em Miura bimodule}, see \eqref{miu1}.
The  Miura bimodule is a slightly refined version of the  quotient $M=\dd(G)\big/(\u\dd(G)+\dd(G)\up)$.
The latter quotient has the natural structure of a $(\dd(G/U),\,\dd^\psi(G/\bU))$-bimodule.
The bimodule $M$ may be viewed as a quantization of the  Miura variety.

The proofs of our main results are completed
in section \ref{7}. In \S\ref{reform}, we show that Theorem \ref{kaph-thm}
is equivalent  to a  result (Theorem \ref{vthm}) concerning singular vectors in the universal Verma module.
This result may be of independent interest. 
Theorem \ref{vthm} is proved in \S\ref{vpf} by reduction to the commutative case, i.e.
to Theorem   \ref{thm*}, via a deformation argument.
The proof of Theorem \ref{grthm} is given in \S\ref{sec6}.

\subsection{Acknowledgements} 
We are grateful to Gus Lonergan for useful comments and the referee for an extremely careful reading
of the paper.
The first author is supported in part by the NSF grant  DMS-1602111.
The second  author is partially supported by ~NSERC.
The project has received funding from ERC under grant agreement 669655.

\section{Geometry of $G/U$}

\subsection{}\label{inv}
Let $G$ be a connected  semisimple group with trivial center
and $\g=\Lie G$.
Let $\bb$ be the flag variety of Borel subgroups  $B\sset G$, equivalently, Borel subalgebras $\b\sset\g$.
The tori $B/[B,B]$ associated with various
Borel subgroups $B\in\bb$ are canonically
isomorphic. Let $T=B/[B,B]$
be this universal Cartan torus  for $G$. Write
$\t=\Lie T$,  let $\BX=\Hom(T,\C^\times) \sset \t^*$ be the root lattice of $\g$, and
$W$ the universal Weyl group.
Let $I$ be the set of vertices of the Dynkin diagram of $G$.

Given a Borel subalgebra $\b$, let
  $\u(\b):=[\b,\b]$ be the nilradical of $\b$, and $\fa(\b):=\u(\b)/[\u(\b),\u(\b)]$. The weights of the $T$-action on $\fa(\b)$
are called simple roots. Write  $\al_i$
 for the  simple root
associated with a vertex $i\in I$.
 For $\la,\mu\in\BX $, we write $\la\leq\mu$
iff $\mu-\la=\sum_{i\in I}\, n_i\al_i$ for some nonnegative integers $n_i$.
The group $G$ being  adjoint, the weights
of any finite dimensional representation $V$ of $G$ are contained in
$\BX$.
Associated with $\b\in\bb$, there is a canonical
$\b$-stable filtration $\g^{\geq\mu,\b},\ \mu\in\BX $, on $\g$, such that 
$\gr^{\mu,\b} \g:=\g^{\geq\mu,\b}/\g^{>\mu,\b}$ is the
$\mu$-weight space for the natural  action of 
the universal Cartan algebra $\t=\b/\u(\b)$. In particular, we have
$\g^{\geq0,\b}=\b$, resp.
$\gr^{0,\b}\g=\t$. For each $i\in I$,  the space
$\gr^{\al_i,\b}\g$, resp. $\gr^{-\al_i,\b}\g$, is  1-dimensional  and 
we put $\ooo^i(\b)=(\gr^{\al_i,\b}\g)\sminus\{0\}$, resp. $\ooo^i_-(\b)=(\gr^{-\al_i,\b}\g)\sminus\{0\}$.
Let $\ooo(\b)=\prod_i\ \ooo^i(\b)$, resp. $\ooo_-(\b)=\prod_i\ \ooo^i_-(\b)$. 
The action of $T$ makes $\ooo(\b)$, resp. $\ooo_-(\b)$, a $T$-torsor.

We put $\fd(\b)=\sum_{i\in I}\ \g^{\geq -\al_i,\b}$.
The map $(s_i)_{i\in I}\mto \sum_i\ s_i$, provides a
$T$-equivariant isomorphism $\oplus_i\,\gr^{\al_i,\b}\g\iso \fa(\b)$, resp.
$\oplus_i\, \gr^{-\al_i,\b}\g\iso \fd(\b)/\b$.
This gives a canonical identification of the $T$-torsor $\bbo(\b)=\prod_i\ \bbo^i(\b)$, resp. 
$\bco(\b)=\prod_i\ \bco^i(\b)$,  with a unique open dense  $T$-orbit in $\fa(\b)$, resp. $\fd(\b)/\b$.
We will abuse notation and write  $x+\b$, resp. 
$\bco(\b)+\b$, for the preimage
of $x\in  \fd(\b)/\b$, resp. the open dense  $T$-orbit in $\fd(\b)/\b$,
under the natural projection $\fd(\b)\to \fd(\b)/\b$.

The family of torsors $\ooo(\b)$, resp. $\ooo_-(\b)$,  for varying  $\b\in\bb$ gives a variety
\beq{tbb}
\wt\bb:=\{(\b, s)\ \mid  \b\in \bb, \en s\in\bbo(\b)\},\resp
\tbm:=\{(\b, s)\ \mid  \b\in \bb, \en s\in\bbo_-(\b)\}.
%\en\ \sset\ \en \V_\pm.
\eeq
The group $G$ acts on $\bb$, resp.  $\wt\bb$ and $\tbm$,
 in a natural way. By construction,  the
first projection  $(\b, s) \mto \b$ is a $G$-equivariant $T$-torsor on $\bb$.

%\beq{bco_def}\ooo^i(\b)=(\gr^{\al_i,\b}\g)\sminus\{0\},\quad\ooo(\b)=\prod_i\ \ooo^i(\b),
%\quad\ooo^i_-(\b)=(\gr^{-\al_i,\b}\g)\sminus\{0\},\quad \ooo_-(\b)=\prod_i\ \ooo^i_-(\b).
%\eeq
 
\subsection{}\label{obb}
We say that  a pair  $\b,\bar{\b}\in \bb$,  of Borel subalgebras, is  in `opposite position' if
$\b\cap\bar{\b}$ is a Cartan subalgebra of $\g$. In that case, one has
a triangular decomposition $\g=\u(\b)\oplus(\b\cap\bar{\b})\oplus \u(\bar{\b})$ and
a diagram
\[\t=\bar{\b}/\u(\bar{\b})\  \twoheadleftarrow\  \bar{\b}\  \hookleftarrow \ \bar{\b}\cap\b\
\into\  \b \ \onto\  \b/\u(\b)  =\t.\]
It follows from definitions that the resulting isomorphism between the
leftmost and rightmost
copy of  $\t$, in the diagram, is given by the map
$t\mto w_0(t)$, where $w_0$ is the longest element of $W$.
The assignment $\al_i\mto -w_0(\al_i)$ gives a permutation $i\mto i'$, of
the set $I$.
For every $i\in I$, one has a diagram
\[
 \gr^{-\al_i,\bar{\b}}\g\  \twoheadleftarrow
\ \g^{\geq -\al_i,\bar{\b}}\    \hookleftarrow  \ 
 \g^{\geq -\al_i,\bar{\b}}\cap\g^{\geq \al_{i'},\b}\
\into
\g^{\geq \al_{i'},\b}\ \onto\ \gr^{\al_{i'},\b}\g.
\]
The compositions on the left and on the right
give 
canonical  isomorphisms
\[
\gr^{-\al_i,\bar{\b}}\g\ccong {\g^{\geq \al_{i'},\b}\cap \g^{\geq -\al_i,\bar{\b}}}\ccong\gr^{\al_{i'},\b}\g,
\qquad
\fd(\bar{\b})/\bar{\b}\ \xleftarrow[\cong]{\kap_-}\ \fd(\bar{\b})\cap \u(\b )\
\xrightarrow[\cong]{\kap_+}\  \fa(\b ).
\]
Thus, the map $\kap_+\ccirc\kap_-\inv$ yields an isomorphism
\beq{bbo-bco}
\fd(\bar{\b})/\bar{\b}\ \iso\
\fa(\b ),\quad\text{resp.}\quad
\kap_{\bar{\b},\b}:\ \bco(\bar{\b})\ \iso\
\bbo(\b),
\eeq
to be denoted $\kap_{\bar{\b},\b}$, such that 
\beq{Tkap}\kap_{\bar{\b},\b}(ts)=w_0(t)\inv\kap_{\bar{\b},\b}(s),\quad\forall\ t\in T, s\in \bco(\bar{\b}).
\eeq
%\beq{kap-iso}
%\xymatrix{
 %\bbo(\bar{\b})\ &&
%\ [\bbo(\bar{\b})+\bar{\b}]\,\cap\, \u(\b )\ \ar[ll]_<>(0.5){\cong}
%\ar[rr]^<>(0.5){\cong}&&
%\ \bco(\b ),
%}
%\eeq

Let $G_\De\sset G\times G$, resp. $T_\De\sset T\times T$, be the diagonal.
The set $\Om$ formed by the pairs of Borel subalgebras in opposite position is a unique open dense $G_\De$-orbit 
 in $\bb\times\bb$. 
Let $\wt\Om$ be the preimage of  $\Om$  under the projection 
$\tbm\times \tb\to\bb\times\bb$. 
We also consider  a subvariety of  $\wt\Om$ defined
as follows:
\beq{Xi}
\Xi\ :=\ \{(\bar{\b},\bar{x},\b,x)\in \tbm \times\tb\mid
(\bar{\b}, \b )\in \Om,\ \kap_{\bar{\b},\b }(\bar{x})=x \}.
\eeq
The group $G\times G$, resp. $T\times T$,  acts on $\tbm\times \tb$ on the left, resp. right.
The variety $\wt\Om$  is $T\times T\times G_\De$-stable,
resp. $\Xi$   is $T_\De\times G_\De$-stable.

Fix a pair $B,\bar{B}$ of Borel subgroups  in opposite position and an element
$s\in \bco(\bar{\b})$. The stabilizer of the point $(\b,\kap_{\bar{\b},\b }(s))\in \tb$, resp.
$(\bar{\b}, s)\in \tbm$, equals  $U=[B,B]$,   resp. $\bar{U}=[\bar{B},\bar{B}]$, 
 the maximal
unipotent subgroup of $B$, resp. $\bar B$.
Thus, the $G$-action gives a $G$-equivariant isomorphism
 $G/U\iso\tb$, resp. $G/\bar{U}\iso \tbm$. With this identification, we have
\beq{xi}
\Xi=\{(g\bar{U}/\bar{U} , gU/U)\in G/\bar{U}\times G/U,\ g\in G\}.
\eeq
The action of $T_\Delta$ on
$\Xi$  is given by the formula
$t:\  (g\bar U /\bar U , gU /U)\mto (gt\bar{U}/\bar{U},\,gtU /U)$.
%Let $pr_-:\ \wt\Om\to\tbm$ be the first,  resp. $pr:\ \wt\Om\to\tb$ the second,
%projection. 

One has a natural map
$p_\Om: \Xi\to \Om$, resp.
$p_-: \Xi\to \tbm$ and  $p:\Xi\to\tb$, given by
 $p_\Om(\bar{\b},\bar{x},\b,x)=(\bar{\b},\b)$,  resp.
$p_-(\bar{\b},\bar{x},\b,x)=(\bar{\b},\bar{x})$ and  $p(\bar{\b},\bar{x},\b,x)=(\b,x)$.

\begin{lem}\label{closed} \vi 
The map 
$T\times \Xi\iso \wt\Om,\ \big(t,\,(\bar{\b},\bar{x},\b,x)\big)\mto (\bar{\b},\bar{x},\b,xt)$,
is  a $T\times G_\De$-equivariant isomorphism.
Furthermore, the variety $\Xi$ is 
closed in  $\tbm \times \tb$. 

\vii 
 The map $p_\Om: \Xi\to \Om$ is a $G_\De$-equivariant $T_\De$-torsor;
moreover,  $\Xi$ is a $G_\De$-torsor.

\viii  Each of the two maps 
below is a  $G_\De$-equivariant isomorphism:
\beq{wtx}
\xymatrix{
\tbm\times_{\bb}\bbm\ &&\ 
\oom\ \ar[ll]_<>(0.5){p_-\times p_\Om}
\ar[rr]^<>(0.5){p_\Om\times p} &&\
\bbm\times_{\bb}\tb .}
\eeq
\end{lem}
\begin{proof} All statements except for the second statement in (i)  are immediate from \eqref{Xi}.
To prove  the second statement in (i), recall that the variety $\tbm$ is quasi-affine, thanks to
 the Pl\"ucker imbedding.
The orbits of a unipotent group action on an affine variety are known to be closed, cf. eg. \cite{Di}, \S 11.2.4. 
It follows that any $U$-orbit in
$\tbm$ is a closed subset of the affine closure of $\tbm$,
hence this orbit is closed in $\tbm$.
The fiber of the map $p$ over $1U/U\in\tb$ is a 
single $U$-orbit in $\tbm$.  We deduce that
the $G_\Delta$-orbit of the point $(1\bU/\bU,1U/U)$  is closed
in $G/\bU\times G/U$, and we are done by ~\eqref{xi}.
\end{proof}

\section{A torsor on the set of regular elements}%\label{sec2}

\subsection{} \label{fa}

The quotient $\fe(\b)=\b/[\u(\b),\u(\b)]$ has the natural structure of a Lie algebra such that
$\fa(\b)=\u(\b)/[\u(\b),\u(\b)]$ is an abelian ideal of $\fe(\b)$, and we have $\fe(\b)/\fa(\b)=\t$.
If confusion is unlikely, we will use simplified notation
$\fd=\fd(\b),\fa=\fa(\b)$, etc.
We use the Killing  form  to identify
$\g$ with $\g^*$.
We obtain the following  isomorphisms:
\beq{killing}
\xymatrix{
%\fe\ar@{=}[d]&\fa\ar@{=}[d]\ar@{_{(}->}[l] &
\u\ \ar@{=}[d]\ar@{^{(}->}[r] & \ \b\ \ar@{=}[d]\ar@{^{(}->} [r] 
&\  \fd\ \ar@{=}[d]\ar@{->>}[r]  &\  \fd/\u \ \ar@{=}[d]\ar@{->>}[r]  &\  \fd/\b\ \ar@{^{(}->}[r] \ar@{=}[d]
& \ \g/\b\ar@{=}[d]&\  \g/\u\ \ar@{->>}[l]\ar@{=}[d] \\
%(\fd/\fa)^*&(\fd/\b)^*\ar@{_{(}->}[l] &
\b^{\perp}\ \ar@{^{(}->}[r] 
 &\  \u^{\perp}\ \ar@{^{(}->}[r]  &\  [\u, \u]^{\perp}\ \ar@{->>}[r] &\  \fe^*\ \ar@{->>}[r] 
&\ \fa^*\ \ar@{^{(}->}[r] &\ \u^*\ &\ \b^*.\ \ar@{->>}[l] .
}
\eeq

Recall the notation $\fc=\g\dsl G:=\Spec(\C[\g]^G)$.
We consider the following diagram:
\beq{springer}
\xymatrix{
\g\ &&\ \tg\ :=\ \{(\b, x)\in \bb\times \g \mid
x\in\b\}\ \ar[ll]_<>(0.5){x\ \leftarrowtail\ (\b,x):\ \pi} \ar[rrr]^<>(0.5){\ \nu:\ (\b, x)\ \rightarrowtail\
 x\mod \u(\b)\ }
&&&\ \t.
}
\eeq
The  map $\pi$ is a  
projective morphism,   the Grothendieck-Springer
morphism. The morphism $\nu$ is  smooth and the map $\pi\times\nu$  factors through
$\g\times_\fc\t$. The above maps are
 $G$-equivariant where $G$ acts  diagonally on $\tg$ and trivially on $\t$.
The first projection $(\b,x)\mto \b$ makes $\tg$ a $G$-equivariant
 vector bundle on $\bb$ with fiber $\b$.

The stabilizer of an element
$(\b,s)\in \tb$ is the maximal unipotent subgroup associated with the
Lie algebra $\u(\b)$.  It  follows that one has
\[\T^*\tb\ =\ \{(\b, s,x)\mid \b\in\bb,\,s\in \bbo(\b),\, x\in \u(\b)^\perp\}.
\]
The $G\times T$-action on $\tb$ induces a Hamiltonian 
 $G\times T$-action on $\T^*\tb$ with moment map
${\mu_{\T^*\tb}}\times{\nu_{\T^*\tb}}:\ \T^*\tb\to\g^*\times\t^*$.
A choice of base point $(\b,s_0)\in \tb$
 gives a $G$-equivariant isomorphism
\beq{ttb}  G\times_U\u(\b)^\perp \iso  \T^*\tb,\quad (g,x)\mto \big(\Ad g(\b), \Ad g(s_0), \Ad g(x)\big),
\eeq
where $B$ is the Borel subgroup with Lie algebra $\b$, resp. $U=[B,B]$ and $\Ad g(s_0)$ is an element of $\bbo(\Ad g(\b))$. 
Using the identifications \eqref{killing}, we get the following
$G$-equivariant  isomorphisms:
\beq{id-tg}
\xymatrix{
\T^*\tb\ \ar@/^1.5pc/[rrrr]|-{\ {\mu_{\T^*\tb}}\times{\nu_{\T^*\tb}}\ } 
\ar@{=}[d]\ar@{->>}[rr]&&\  (\T^*\tb)/T \ 
\ar@{=}[d]\ar@{->>}[rr] &&\  
\g^*\times_{\g^*/\!/G}\t^*\ \ar@{=}[d]\\
\tg\times_\bb\tb \ \ar@{->>}[rr] _<>(0.5){\tilde\varrho }^<>(0.5){T\text{-torsor}}&&\  \tg\ \ar@{->>}[rr]_<>(0.5){\pi\times\nu}
 &&\  
\g\times_\fc\t,
}
\eeq
where  $\tilde\varrho :\ \tg\times_\bb\tb \to\tg$ is  a pull-back of the $T$-torsor
$\varrho: \tb\to\bb$ via the vector bundle map $\tg\to\bb$.

From now on, we will identify the first, resp. second and third, object of the top
row of diagram \eqref{id-tg} with the first, resp. second and third, object of the
bottom
row of the diagram. Thus, we view $\tilde\varrho $ as a map
$\T^*\tb\to (\T^*\tb)/T=\tg$ so, we have ${\mu_{\T^*\tb}}=\pi\ccirc \tilde\varrho $, resp. ${\nu_{\T^*\tb}}=\nu\ccirc \tilde\varrho $.

\subsection{}\label{fZ-sec} We write  $G_x$ for 
 the stabilizer  of
an element $x\in\g$  under the   $\Ad G$-action and
let $\g_x=\Lie G_x$. 
We say that $x$ is regular if $\dim\g_x=\rk$. 
 Let $\g\reg$ be 
 the set of regular (not necessarily semisimple) elements of $\g$.
Let  $G$ act on itself by conjugation and on
 $\g_r \times G$ diagonally. We define
\[\cg:=\{(x,g)\in \g_r \times G \mid \Ad  g(x)=x\}.
\]

The fiber of $\cg$ over $x\in \g_r $
 equals $G_x$, which is a connected  abelian group, \cite{Ko1}, Proposition 14.
This makes $\cg\to\g_r$   a  smooth connected $G$-equivariant
abelian subgroup scheme of  the constant group scheme $\g_r \times G \to\g_r$,
cf. eg. \cite{Ngo}, \S 2.1.

%\begin{defn} The {\em universal centralizer} is defined as $\fZ:=\Spec(\C[\cg]^G)$.
%The map $\cg\to\fc$ descends to a morphism $\fZ\to\fc$, making
%$\fZ$ an abelian  group scheme on $\fc$.
%\end{defn}

We introduce the following incidence variety
\beq{xx}
\xx:=\{(\b, x)\in \bb\times \g\mid  x\in
\bbo_-(\b)+\b\}.
\eeq
We let $G$ act on $\xx$ diagonally. We have the following  $G$-equivariant maps
\beq{moment}
\xymatrix{
\tbm\ &&&\ \oox \ 
\ar[lll]_<>(0.5){(\b,\, x\mod \b) \ \leftarrowtail\  (\b,x):\ q}
\ar[rrr]^<>(0.5){\mum :\  (\b,x)\ \rightarrowtail \ 
x}&&&
\ \g.
}
\eeq
The map $q$ is a
fibration on $\tbm$ with affine-linear fibers $q\inv(\b,s)\cong s+\b$,
and  $\mum$ is a  projective morphism.

Let $Y$ be a $G$-variety  and  $f: Y\to \g_r$ a $G$-equivariant map.
For any $x\in\g_r$,  the fiber $f\inv(x)$  is $G_x$-stable.
The family of maps $G_x\times f\inv(x)\to f\inv(x),\ x\in\g_r$, yields
an action $\fZ\times_{\g_r}Y\to Y$; furthermore,
the action map  is $G$-equivariant.

\begin{lem}\label{torsor} \vi We have $\mum(\oox)=\g_r$, resp. $\vth(\g_r)=\fc$,
where $\vth: \g\to\fc=\g\dsl G$ is the adjoint quotient.

\vii The map $\mum: \oox\to \g_r$ is  a $\fZ$-torsor  in the \'etale topology,
in particular, $\mum$ is a smooth morphism.

\viii The composite $\oox\xrightarrow{\mum}\g_r\xrightarrow{\vth}\fc$ is a
trivial $G$-torsor on $\fc$.
\end{lem}

To prove the lemma, it is convenient to choose, once and for all,
a principal $\sll$-triple
$(\bbe,\bbh,\bbf)$. 
Let  $\b_\bbe$, resp. $ \b_\bbf$, be the  unique Borel 
subalgebra of $\g$ that contains the element 
$\bbe$, resp. $\bbf$. Let $\u_\bbe=\u(\b_\bbe)$, resp. 
$\u_\bbf=\u(\b_\bbf)$.
We use similar notation for 
 the corresponding subgroups of $G$. 
 Since $\bbo_-(\bbb)$ is a $T$-torsor, the imbedding  $\bbe+\bbb\into
\bbo_-(\bbb)+\bbb$
induces a 
$B_\bbf$-equivariant isomorphism
\beq{bboe}
B_\bbf\times_{U_\bbf}(\bbe+\b_\bbf) \ \iso \ \bbo_-(\bbb)+\bbb.
\eeq

 Let $\bbe+\g_\bbf$ be the Kostant slice.
We recall the following well-known result, see
\cite{Ko2}.

\begin{prop}\label{munuprop} \vi The map $\vth: \g_r\to\fc$
is a smooth  and surjective morphism; moreover, each fiber of this map
 is a single $G$-orbit in $\g_r$.

 \vii We have $\bbe+\g_\bbf\sset\g_r$. Furthermore, the slice $\bbe+\g_\bbf$
 meets every $G$-orbit in $\g_r$ transversely at a single point,
 in particular, one has 
 $(\bbe+\g_\bbf)\cap \Ad G(\bbe)=\{\bbe\}$.

\viii The composition $\bbe+\g_\bbf\into \g_r\xrightarrow{\vth}\fc$  is an isomorphism.

\iv The action map $U_\bbf\times (\bbe+\g_\bbf) \, \to\,\bbe+\b_\bbf$ is an isomorphism.\qed
\end{prop}

\begin{proof}[Proof of Lemma \ref{torsor}] 
%From  Proposition \ref{munuprop}(iv), we deduce 
%that the $B_\bbf$-action on $\bbo(\bbb)+\bbb$
%yields a $B_\bbf$-equivariant isomorphism
%$B_\bbf\times \s\iso \bbo(\bbb)+\bbb$
We have the following chain of 
$G$-equivariant
isomorphisms:
\beq{longiso}
\oox\ccong G\times_{B_\bbf} (\bbo_-(\b_\bbf)+\b_\bbf)\ccong
G\times_{U_\bbf} (\bbe+\b_\bbf)\cong G\times_{U_\bbf} (U_\bbf\times (\bbe+\g_\bbf))\ccong
G\times(\bbe+\g_\bbf),
\eeq
where the first isomorphism is immediate from \eqref{xx},
the second  isomorphism follows  from \eqref{bboe}, and
the third   isomorphism follows  from  Proposition \ref{munuprop}(iv).
Using \eqref{longiso} and the identification
$\tbm=G/\U_\bbf$,  the maps in \eqref{moment} read as follows:
\[
\xymatrix{
\tbm=G/\U_\bbf\ &&&\ \oox \ \cong\  G\times_{U_\bbf} (\bbe+\b_\bbf)\ccong G\times (\bbe+\g_\bbf) \
\ar[lll]_<>(0.5){gU_\bbf/U_\bbf \ \leftarrowtail\  (g,x):\ q}
\ar[rrr]^<>(0.5){\mum :\ (g,x)\ \rightarrowtail \ 
\Ad g(x)}&&&
\ \g.
}
\]
This yields parts (i) and  (iii) of Lemma \ref{torsor}.
Also, it is immediate from Proposition \ref{munuprop}(i)
and the explicit description above that the
differential of the map $\mum$ is surjective, hence this map is a smooth morphism.
To prove (ii) we first show that
$\oox\to \g_r$ is a $\fZ$-quasi-torsor, i.e. the action
map $a: \fZ\times_{\g_r}\oox\to\oox\times_{\g_r}\oox$ is an isomorphism.
The varieties involved are smooth since $\fZ$ and $\oox$ are smooth schemes over $\g_r$.
Thus, it suffices to show that $a$ is a bijection.
To this end, let  $x\in \g_r$.
It follows from (iii)  that  the map
$\mum: \mum\inv(\Ad G(x))\to \Ad G(x)$ can be identified with the
quotient map $G\to G/G_x$. This identification respects the
$G$-action. Furthermore, the fiber $\mum\inv(x)\sset \mum\inv(\Ad G(x))$
goes, via the identification, to the subgroup $G_x\sset G$.
This yields a $G_x$-equivariant isomorphism $\mum\inv(x)\cong G_x$,
proving that $\oox\to \g_r$ is a quasi-torsor.
%Also, it follows from Proposition \ref{munuprop}(ii) that
%the differential  of  the
%last map in the displayed formula above  is surjective,
%hence, this map is a smooth morphism.

It remains to show that this
quasi-torsor is \'etale locally trivial. By $G$-equivariance, it suffices to
show that for any $x\in \bbe+\g_\bbf$
the quasi-torsor has a section on an \'etale neighborhood of $x$.
To construct such a  section, recall
that the Lie algebra  $\g_\bbe$ is an $\ad\bbh$-stable
subspace of  $\u_\bbe$
and let   ${\mathfrak v}$ be an
arbitrary $\ad\bbh$-stable vector
space complement of $\g_\bbe$ in $\u_\bbe$.
The group $U_\bbe$ being unipotent, the image of
${\mathfrak v}$ under the exponential map $\exp: \u_\bbe\to U_\bbe$
is a closed  algebraic subvariety $V\subset U_\bbe$; moreover, the map
$\exp: {\mathfrak v}\to V$ is an isomorphism  of algebraic varieties.
We define a map
$f: V\times (\bbo_-(\b_\bbf)+\b_\bbf)\to \g_r$ 
by the assignment
$(v,y)\mto \Ad v(y)$.

To complete the proof we observe that   Proposition \ref{munuprop}(ii) implies,
using that
$\u_\bbe={\mathfrak v}\oplus \g_\bbe$,
that one has
a direct sum decomposition
$\g=\fd(\b_\bbf)\oplus [\bbe, {\mathfrak v}]$,
see Section \ref{inv} for the definition of $\fd(\b_\bbf)$.
It follows that for any $x$ in a
Zariski neighborhood of $\bbe$ in $\bbe+\g_\bbf$ 
one has $\g=\fd(\b_\bbf)\oplus [x, {\mathfrak v}]$.
Using the standard ${\mathbb G}_m$-action that contracts the slice
$\bbe+\g_\bbf$ to $\bbe$ and the fact that  the vector
spaces $\fd(\b_\bbf)$ and ${\mathfrak v}$ are
 ${\mathbb G}_m$-stable,  we deduce that a similar direct sum decomposition holds
for all $x\in \bbe+\g_\bbf$.
Now, the set $\bbo_-(\b_\bbf)+\b_\bbf$ is an open
subset of  $\fd(\b_\bbf)$ and for
$x\in \bbe+\g_\bbf$, we  have $(1,x)\in V\times (\bbo_-(\b_\bbf)+\b_\bbf)$.
It follows from
the decomposition  $\g=\fd(\b_\bbf)\oplus [x, {\mathfrak v}]$
 that
the differential of the map $f$ at the point $(1,x)$ 
is a vector space  isomorphism.
Hence, there is a Zariski open neighborhood
$D\sset V\times (\bbo_-(\b_\bbf)+\b_\bbf)$, of $(1,x)$, such that the restriction
of $f$ to $D$ is  an \'etale morphism. Now, view (as we may)
$\oox\times_{\g_r}D$ as a subvariety of
$(\bb\times\g_r)\times_{\g_r}D=\bb\times D$.
Then, it is immediate from definitions
that the
map $D\to \bb\times D,\, (v,y)\mto (\Ad v(\b_\bbf), v,y)$, provides
a section of the $f^*\fZ$-torsor $f^*\oox=\oox\times_{\g_r}D\to D$, as desired.
\end{proof}

\subsection{}\label{can-psi}
 For a Borel subgroup  $\bar  B$, one has a chain of imbeddings
$\bbo_-(\bar\b)\into \fd(\bar\b)/\bar\b=\fa^*(\bar\b)\into \u(\bar\b)^*$,
cf. \eqref{killing}. Let $\Psi: \bbo_-(\bar\b)\to \u(\bar\b)^*$ be the composition.
A character  of the Lie algebra
$\u(\bar\b)$ is 
 called {\em nondegenerate} if it is contained in the image of $\Psi$.
 Thus,  the elements
of $\bco(\bar \b)$ may (and will) be identified  with nondegenerate characters.
It is clear  from \eqref{xx}-\eqref{moment} that
one has $q\inv(\bar \b,s)=\Psi(s)+\u(\bar  \b)^\perp$,
where we have used the identifications from \eqref{killing}.

Let $\u_\bb$ be a vector bundle on $\bb$ with fibers $\u(\bar\b),\ \bar\b\in\bb$, 
and  let $\u_\tbm$ be its pull-back to $\tbm$ via the projection $\tbm\to \bb,\ (\bar \b,s)\mto\bar \b$.
The assignment  $(\bar \b,s)\mto \Psi(s)$ gives a canonical section of the dual vector bundle
$\u_\tbm^*$. We may view $\u_\tbm^\perp$ as a  subbundle of the trivial 
 vector bundle $\tbm\times\g^*\to\tbm$. Then,
it follows from the above discussion that there is a canonical
 $G$-equivariant isomorphism
\beq{psi-can}
\oox\ccong \Psi+\u_\tbm^\perp =: \T^\psi\tb_-,
\eeq
of schemes over $\tbm$, where   
 we have used the identifications from \eqref{killing}.

Fix a point
$(\bar \b,s)\in\tbm$  and let $\bar U=[\bar B,\bar B]$, resp. $\bu=\u(\bar\b)$
and $\psi=\Psi(s)$.
Using the identification $\tbm=G/\bU$, isomorphism \eqref{psi-can}
takes the form
\beq{psix}
\oox\cong G\times_\bU (\psi+\bu^\perp)=:\T^\psi(G/ \bar U).
\eeq

Let $m: \T^*G\to\bu^*$ be
the moment map
associated with the $\bar U$-action on $G$ by right translations.
The variety $G\times_\bU (\psi+\bu^\perp)$
 may be identified with $m\inv(\psi)/\bar U$,
%Motivated by \eqref{ttb}, we put
%\beq{psi1}\T^\psi(G/ U)\ :=\ G\times_{ U} (\psi+\bu^\perp)\cong\mu_{\T^*G}\inv(\psi)/U,
%\eeq
%where $\mu_{\T^*G}: \T^*G\to\g^*$ is the moment map
%associated with the $U$-action on $G$ by right translations.
%Thus, $\T^\psi(G/ U)$ is 
a Hamiltonian reduction of $\T^*G$ with respect to $(\bar U,\psi)$.
Therefore,   $\T^\psi(G/ \bar U)$ comes equipped with a  symplectic structure such that
the  $G$-action on  $\T^\psi(G/ \bar U)$ is  Hamiltonian and the corresponding moment
map  goes, via \eqref{psix}, to the map $\mum$.
The symplectic structure on  $\T^\psi(G/\bar  U)$ gives, via \eqref{psix},  a symplectic structure on ~$\oox$.

It is possible to use isomorphism \eqref{psi-can} to define
the  symplectic structure on $\oox$ in a canonical way that does not involve the choice of 
a point $(\bar \b,s)\in\tbm$. To this end, one has to
generalize the  Hamiltonian reduction construction
to the setting of group scheme actions.
In section \ref{gtors}, we will explain a quantum counterpart of such a construction that produces
a sheaf $\dd_\tbm^\Psi$, of twisted differential operators on $\tbm$,
 that may be viewed as a quantization of $\oox$.

\section{The Miura variety}\label{4}

\subsection{}\label{Ysec} We will freely use the notation of \S\ref{fa}.
Put  $\tgr=\pi\inv(\g_r)$, a $G$-stable Zariski open subset of ~$\tg$.

Part (i)  of the following result is due to Kostant; part (ii) is also known, cf. e.g. \cite{Gi2}, Lemma 5.2.1, for a proof.
\begin{prop}\label{genpos} \vi The restriction of 
the map  $\pi\times\nu$ to $\tg_r$
yields an isomorphism $\tgr\iso \g_r\times_\fc\t$.

\vii Let $\b, \bar{\b}$ be a pair of Borel subalgebras.
Then, the set $\b\cap (\bco(\bar{\b})+\bar{\b})$ is nonempty iff
$\b$ and $\bar \b$ are in opposite position.
\end{prop}

\begin{defn}
The {\em Miura variety} is defined as 
$\wtx:=\xx\times_{\g_r}\tgr$.\ 
Set theoretically, we have
 \begin{align}\wtx&=\{(\bar{\b} ,\b,x)\in\bb\times\bb \times\g\ \mid\ 
\bar{\b}\in\mum\inv(x),\, \b \in\pi\inv(x)\}\label{triples}\\
&=\big\{(\bar{\b}, \b ,x)\in\bb\times\bb \times\g\ \mid\ 
x\in \b \,\cap\,[\bco(\bar{\b})+\bar{\b}]\big\}.\nonumber
\end{align}
\end{defn}

From 
Proposition \ref{genpos}(i), we deduce
\beq{oxx_fib}
Z\ =\ \oxx\times_{\g_r}\tgr
\ \cong\ \oxx\times_{\g_r}(\g_r\times_\fc\t)
\ \cong\ \oxx\times_\fc\t.
\eeq

We let  $G$ act on $\oxx\times_{\g_r}\tgr$ diagonally.  Let $\wt\pi: Z\to\oox$, resp.
$\wt\mu:\ Z\to \tgr$, be the first, resp. second, projection. 
These maps are $G$-equivariant and one has a
 diagram with cartesian squares:
\beq{cartZ}
\xymatrix{
\wtx\  \ar[d]^<>(0.5){\wt\pi} 
\ar[rrr]^<>(0.5){\wt\mu}\ar@{}[drrr]|{\Box}
&&&\ \tg_r\cong \g_r\times_\fc\t\
\ar[d]^<>(0.5){\pi}\ar[rr]^<>(0.5){\nu}\ar@{}[drr]|{\Box}&&
\ \t \ \ar[d]^<>(0.5){\th}\\
 \oox\ 
\ar[rrr]^<>(0.5){\mum }&&&\ \g_r\ \ar[rr]^<>(0.5){\vth} && \ \fc=\t/W,
}
\eeq
where $\th$ is the quotient map. 
It follows from  Lemma \ref{torsor}(iii) that   the map $\nu\ccirc\wt\mu$  is a $G$-torsor,
in particular, $Z$ is smooth.

Recall the notation of \S\ref{obb}.
By Proposition \ref{genpos}, the assignment 
$(\bar{\b}, \b ,x)\mto (\bar{\b}, \b)$ gives a map
$Z\to\Om$. 
Let   $G$ act on $(\tbm\times_{\bb}\Om) \times \t$
via its action on $\tbm\times_{\bb}\Om$, the first factor.

\begin{prop}\label{zsmooth} 
 The  following map  is a $G$-equivariant  isomorphism
\[q_Z:\  Z \too
 (\tbm\times_{\bb}\Om) \times \t,\quad 
(\bar{\b}, \b ,x)\ \mto\ \big(\big((\bar{\b}, x\mod \bar\b),\, (\bar\b,\b)\big),\ x\mod\u(\b )\big).
\]
\end{prop}

\begin{proof} It is immediate from the construction that the map in question is
 $G$-equivariant. Further, we know that $\tbm\times_{\bb}\Om$  is a $G$-torsor,
see Lemma \ref{closed}. Therefore,  the  map $q_Z$ is a morphism of $G$-torsors on $\t$.
Hence, it is an isomorphism.
\end{proof}

 We now give a more explicit
(though less canonical) description of the Miura variety.
To this end,  we identify  
$\bb\times\bb=G/B_\bbf\times G/B_\bbe$.
The isotropy group of the base  point equals  $B_\bbf\cap B_\bbe=G_\bbh$, a maximal  torus
in $G$. Hence, one has a $G$-equivariant isomorphism
$G/G_\bbh\iso \bbm$. Let $\bbo:={\u(\b_\bbe)\cap [\bbo(\bbb)+\bbb]}$.
This is a $G_\bbh$-torsor. The fiber of the projection
$pr_\Om: Z\to\Om$ over  the base  point equals  $\bbo+\g_\bbh$.
For
$s\in {\mathbb O},\ h\in\g_\bbh$ and  $t\in G_\bbh$,
one has
$\Ad t(s+h)= \Ad t(s) +h$, so  the set ${\mathbb O}+
\g_\bbh$ is $\Ad G_\bbh$-stable. Further,  $\bbe\in {\mathbb O}$
and we have
\beq{sets}
\bbe+\g_\bbh\ =\ \b_\bbe\cap(\bbe+\bbb)\en \sset\en
\b_\bbe\,\cap\,[\bbo(\bbb)+\bbb]\ =\ {\mathbb O}+
(\b_\bbe\cap\bbb)\ =\ {\mathbb O}+
\g_\bbh.
\eeq

We deduce $G$-equivariant isomorphisms
\beq{zz2}
Z\ \cong\ G\times_{G_\bbh}({\mathbb O}+\g_\bbh)\ \cong\ 
G\times_{G_\bbh}(G_\bbh\cd\bbe+\g_\bbh)\ \cong\ G\times (\bbe+\g_\bbh).
\eeq
In particular,
for any  $(\bar{\b}, \b ,x)\in Z$, there are uniquely determined
elements $h\in\g_\bbh$
and $g\in G$ such that one has $(\bar{\b}, \b ,x)=\Ad g(\b_\bbe,\bbb, h+\bbe)$.
Further, the map $G\to \tbm\times_\bb\Om,\ g\mto (g U_\bbf/U_\bbf, gG_\bbh/G_\bbh)$,
is a  $G$-equivariant isomorphism. With these identifications, the isomorphism $q_Z$ of
 Proposition \ref{zsmooth} takes the form
$q_Z:\ (g,h+\bbe)\mto (g,h)$.

\begin{rem}
The isomorphism of Proposition \ref{zsmooth} can also be seen as follows.
Given a Borel $\bar \b$, let 
$\wh\bbo_-(\bar \b):=(\bar\b+\bbo_-(\bar\b))/\u(\bar\b)$.
The   natural 
projection $d: \wh\bbo_-(\bar \b)\to\bco(\bar \b),\
y\mto y\mod \bar\b$, is an 
affine bundle, a torsor under the action of $\t=\bar\b/\bu(\bar\b)$
viewed as an additive group.

For any Borel $\b$ which is opposite to $\bar \b$, we have a chain of maps
\[\u(\b)\cap(\bar\b+\bbo_-(\bar \b))\overset{a}\into
\b\cap(\bar\b+\bbo_-(\bar \b))\overset{b}\into
\bar\b+\bbo_-(\bar \b)\overset{c}\onto
(\bar\b+\bbo_-(\bar \b))/\u(\bar\b)\overset{d}\onto
(\bar\b+\bbo_-(\bar \b))/\bar\b.
\]
It is clear that the composite map $c\ccirc b$,
resp. $d\ccirc c\ccirc b\ccirc a$,  is an isomorphism.
Therefore, the map $(c\ccirc b\ccirc a)\ccirc(d\ccirc c\ccirc b\ccirc a)\inv:\ \bbo_-(\bar \b)\to \wh\bbo_-(\bar \b)$
provides a section of the $\t$-torsor $d$.

One can let the Borel subalgebra $\bar \b$ vary.
Specifically, we put
\begin{gather*}
Z_\u:=\{(\bar \b,\b,x)\mid (\bar \b,\b)\in\Om,\ x\in \u(\b)\cap(\bar\b+\bbo_-(\bar \b))\};\\
\wh\bb_-:=\{(\bar\b,y)\mid \bar\b\in\bb,\ y\in\wh\bbo_-(\bar \b)\}.
%,\quad\ \bb^\spadesuit_-:=\{(\bar\b,x)\mid \bar\b\in\bb,\ x\in\b+\bbo_-(\bar \b)\}.
\end{gather*}

Similarly to the above, one has a $\t$-torsor $d: \wh\bb_-\to\tbm$.
A counterpart of the above 
diagram
has the form
\beq{section}
\xymatrix{
Z_\u\ \ar@{^{(}->}[r]^<>(0.5){a}&Z\ \ar@{^{(}->}[r]^<>(0.5){b}&
\ \oox\times_\bb\Om\ \ar@{->>}[r]^<>(0.5){c}&
\ \wh\bb_-\times_\bb\Om\ \ar@{->>}[r]^<>(0.5){\en\ d\times\Id_\Om\en\ }&
\ \tbm\times_\bb\Om.
}
\eeq
This is a diagram of smooth schemes over $\Om$ such that, for
any  $(\bar \b,\b)\in \Om$, the corresponding fibers in \eqref{section} form the previous diagram.
It follows that the composite map $c\ccirc b$,
resp. $(d\times\Id_\Om)\ccirc c\ccirc b\ccirc a$,  in \eqref{section} is an isomorphism.
We deduce that the map $(d\times\Id_\Om)\ccirc c\ccirc b$, in \eqref{section}, is  a $\t$-torsor and
the map $a\ccirc((d\times\Id_\Om)\ccirc c\ccirc b\ccirc a)\inv:\ \tbm\times_\bb\Om
\to Z$
provides a section of that torsor.
This section yields a trivialization $Z\cong (\tbm\times_\bb\Om)\times \t$,
which is the isomorphism $q_Z$ of Proposition ~\ref{zsmooth}.

Using the canonical isomorphism $\tbm\times_{\bb}\Om\cong \Xi$,
see \eqref{wtx}, we will often view  $q_Z$ as an  isomorphism  $Z\iso \Xi \times \t$.
\end{rem}

\subsection{Relation to dynamical Weyl groups}
We let $W$ act  on $\g_r\times_\fc\t$, resp.   $\oox\times_\fc\t$, via the natural
$W$-action on $\t$, the second factor.
Transporting the  $W$-action via
the isomorphism of Proposition \ref{genpos}(i), resp. Proposition \ref{zsmooth}, 
gives a  $W$-action on $\tgr$, resp. $Z$.
In particular, for any $x\in\g_r$, we get a well-defined $W$-action
$w: \b\mto\b^w$, on the fiber $\bb^x =\pi\inv(x)$.
The maps in the top row of diagram \eqref{cartZ} are  $G\times W$-equivariant.

Let $\BG$ be
 the automorphism group  of the  $G$-torsor $\Xi\cong\tbm\times_\bb\Om$.
Thus, $\BG$ is   noncanonically
isomorphic to $G$. %The universal Cartan torus $T$ acts  on $\Xi$
%by $G$-equivariant
%automorphisms. Hence, there is
%a {\em canonical} imbedding $T\into \BG$.
%It turns out that it is the group $\BG$, rather than $G$,
%that appears in the most canonical formulation of
%Langlands duality. Specifically, our observation,
%which seems to be new, is 
%that
%the fiber functor on the Satake category provided
%by equivariant cohomology is a functor to 
% the
%category $\rep$,  of finite dimensional representations
%of the group $\BG$, rather than  to 
% the
%category $\rep(G)$, cf. also \S\ref{grsec} below.
%For  any  $V\in \rep$,
%thanks to the imbedding  $T\sset\BG$
%as a maximal torus,
%one has a weight  decomposition
%$V=
%\oplus_{\la\in \bX}\ V_\la$, 
%where  $\bX\sset\t^*$ is  the  weight lattice
%of $T$.
The action 
of an element $w\in W$ on $Z$ gives an automorphism 
of the variety  $\Xi\times\t$ of the form
$w:\ (x,h)\mto (\phi_w(h)(x),\ w(h))$,
where $\phi_w: \t\to\BG$ is a certain 
regular map. It follows from the construction
that the  maps $\{\phi_w,\ w\in W\}$ satisfy
a cocycle equation:
\beq{phi1}
\phi_{w_1}(w_2(h))\cdot\phi_{w_2}(h)=\phi_{w_1w_2}(h),\qquad \forall w_1,w_2\in W,\ h\in \t.
\eeq

Explicit  formulas for the maps $\phi_w$ are reminiscent
of the formulas
appearing in the theory of classical $r$-matrices and
 dynamical Weyl groups, cf. \cite{GR} for some related
results.

\section{The key construction}\label{5}
\subsection{}\label{xyzsec} 
Let $\g_{rs}$ be the set of regular semisimple elements of $\g$.
Given a scheme $Y$ over $\g$, we put $Y\rs:=Y\times_\g\g\rs$.
For any variety $Y$, let $T_Y$  denote a constant group scheme $T\times Y\to Y$.

The goal of this subsection is  to relate, following
\cite{BK} and \cite{DG}, the group schemes  $\pi^*\fZ\to\tgr$
and 
$T_{\tgr}\to\tgr$.  To this end, we will use the following known result.

\begin{lem}\label{bk-lem} 
Let 
$B$ be a Borel subgroup with Lie algebra $\b$,
and $x\in  \b\cap\g_r$. Then,  $G_x\sset B$.
Furthermore, writing  $x=h+n$ for the Jordan decomposition of $x$,
we have $G_x=Z(G_h)\cdot U_x$, where $U_x$ is
 the unipotent radical of  the group $G_x$ and $Z(G_h)$, the center of $G_h$, is a torus (in particular, it is connected).
\end{lem}
\begin{proof} It is a standard fact that the group $G_h$ is connected and, moreover,
if $G$ is adjoint then the group $Z(G_h)$ is a torus. The proofs of other statements 
in the (much harder) case where $x$ is an element of $G$ rather than $\g$ can be found in
\cite{SS}, \S\S 1.6, 1.14.
\end{proof}

Given an element $(\b,x)\in\tgr$, let  $\vkap_{\b,x}$ be  the following composition
\beq{vkap}
\vkap_{\b,x}:\ G_x \ \intoo\  B\ \ontoo\  B/[B,B]=T,
\eeq
where the first map is well defined thanks to the lemma.
It is straightforward to upgrade the construction of  the map $\vkap_{\b,x}$
for each individual element
$(\b,x)\in\tgr$ and obtain  a morphism
$\dis\vkap:\ \pi^*\fZ\to T_{\tgr}$,
of  group schemes on $\tgr$.

Let the Weyl group $W$ act  on  $T\times\tgr$
diagonally and act  on $\fZ\times_{\g_r}\tgr$
through its action on $\tgr$, the second factor.
Further, we let $G$ act on $T\times\tgr$ and on  $\fZ\times_{\g_r}\tgr$
through its action on $\tgr$.
The actions of $W$ and $G$ on $T_\tgr=T\times\tgr$, resp. $\pi^*\fZ=\fZ\times_{\g_r}\tgr$,
commute.
This makes $T_\tgr$, resp.
$\pi^*\fZ$, a   $G\times W$-equivariant group scheme on $\tgr$.

\begin{lem}\label{imvkap}  The morphism $\vkap: \pi^*\fZ\to T\times\tgr$
is  a  morphism of $G\times W$-equivariant group schemes.
\end{lem}
\begin{proof} It is immediate to check that the morphism
$\pi^*\fZ\rs\to T_{\tg\rs}$ induced from
$\vkap$ by restriction is a $G\times W$-equivariant isomorphism.
The result follows from
this since  $\fZ_{rs}$ is
Zariski dense in $\fZ$.
\end{proof}

%\begin{rem}\label{bk-rem} 
 The  $W$-equivariance of the map
$\vkap$ says that 
$\ \dis \vkap_{\wb,x}(g)=w(\vkap_{\b,x}(g))$, for  any $(\b,x)\in \tgr,\ g\in G_x,\
w\in W$.

It will be convenient in what follows to view $\g_r$ as a subset of $\g^*$ rather than $\g$.
The $G$-equivariant group scheme $\fZ$
on $\g_r$ descends to a smooth group scheme $\fZ_\fc$ on $\fc$, see \cite[\S 2.1]{Ngo}.
Thus, we have $\fZ=\fZ_\fc\times_\fc\g_r$.
From the isomorphism $\tgr\cong \g_r\times_\fc\t^*$ we get
\[\pi^*\fZ=\fZ_\fc\times_\fc\g_r\times_\fc\t^*,\quad\text{resp.}\quad
T_\tgr=T\times (\g_r\times_\fc\t^*).
\]

It is immediate from the  criterion
for faithfully flat descent  of morphisms 
that the composition
$\pi^*\fZ\xrightarrow{\,\vkap\,}T_\tgr=T\times (\g_r\times_\fc\t^*)
\xrightarrow{\,{}_{pr_{1,3}}\,}T\times \t^*$
factors through
a morphism,   \cite{BK}, \cite{DG},
\beq{kapz}
\vkap_\fc:\ \fZ_\fc\times_\fc\t^*\too T_{\t^*}=T\times\t^*,
\eeq
of group schemes on $\t^*$, so that one has
$\vkap=\vkap_\fc\times_{\t^*}\Id_\tgr$.

Let $\fz\to \g_r$, resp. $\fz_\fc\to\fc$, be the Lie algebra of $\fZ$, resp. $\fZ_\fc$.
Thus, $\fz$, resp. $\fz_\fc$,  is a vector bundle on $\g_r$, resp. $\fc$,
and one has a canonical isomorphism $\vth^*\fz_\fc\iso \fz$.
The fiber of $\fz$  over $x\in \g_r$ equals  $\g_x$, hence, the
fiber of  $\fz_\fc$ over $c\in\fc$ 
is canonically identified with $\g_x$, for any  $x\in \g_r$ such that $c=\vth(x)$.
Although the following  result is known, cf. eg. \cite{BF}, p.40, 
we  reproduce the proof since the argument will be used later.

\begin{lem}\label{fz-lem}
  There is a canonical isomorphism
  $\fz\cong\vth^*(\T^*\fc)$, resp. 
$\fz_\fc\cong\T^*\fc$.
\end{lem}
\begin{proof} It is clear that the second isomorphism follows from the first,
  by descent. To prove the first isomorphism,
fix  $x\in \g^*$ and let $N_x$, resp. $N^*_x$, be  the normal, resp. conormal, space at $x$
to  the $G$-orbit of $x$. 
Thus,  $N^*_x$ is a subspace of $(\g^*)^*$ and it is immediate to check that using the identification
 $(\g^*)^*=\g$, one has $N_x=\g_x$.
Assume now that $x$ is regular and let $c=\vartheta(x)$.
Then, Proposition \ref{munuprop}(i)  implies that 
the differential  of the map $\vth$
yields an isomorphism $d\vth: N_x\iso \T_c(\fc)$.
We deduce that  the dual of the map $d\vth$
provides a canonical isomorphism $\T_c^*(\fc)\iso N_x=\g_x$.
This isomorphism   sends $(df)_c$ to $d(\vth^*f)_x$,
where  $\vth^*: \C[\fc]\iso \C[\g^*]^G,\ f\mto\vth^*f$ is
the tautological isomorphism.
\end{proof}

Let $Y$ be a smooth symplectic manifold equipped with a morphism $\mu_Y: Y\to\fc$ and 
an action $\fZ_\fc\times_\fc Y\to Y$.  Let $\alpha: \fz_\fc\times_\fc Y\to \T Y$ be the differential
of that action and write $\xi_F$ for the Hamiltonian vector field on $Y$  (a section of $\T Y\to Y$)
associated with a regular function $F\in \C[Y]$.
 
Observe that the differential of a regular function
$f\in \C[\fc]$ may be viewed, thanks to Lemma \ref{fz-lem},
 as  a section $df\in \Ga(\fc, \fz_\fc)$. 
We say that the $\fZ_\fc$-action on $Y$
 is {\em Hamiltonian with moment map} $\mu_Y$ if the following  holds:
\beq{df}\alpha(df)=\xi_{\mu_Y^*f},\qquad\forall f\in \C[\fc].
\eeq

To unburden  notation we will write $\C[Y]^{\C[\fc]}$ for the Poisson centralizer
of the algebra $\mu_Y^*\C[\fc]$ in the Poisson algebra of regular functions on $Y$.

We have  a $\fZ_\fc$-action  $\fZ_\fc\times_\fc(T\times \t^*)\to T\times \t^*,\
(z,\, (t,\tau))\mto (\vkap_\fc(z)t,\tau)$.
The $\fZ_\fc$-action on $\xx$ gives, via the isomorphism
$\xx\cong\T^\psi\tbm$, see \eqref{psi-can}, a  $\fZ$-action on $\T^\psi\tbm$.

 The  proof of  the following result  is straightforfard.

\begin{lem}\label{ham-act} 
The $\fZ_\fc$-action on $\T^*T=T\times \t^*$,
resp.  $\T^\psi\tbm$, is Hamiltonian with moment map
$\th\ccirc pr_2$, resp. 
$\vth\ccirc\mum$.\qed
\end{lem}

\subsection{}\label{ysec} 
Recall the setting of  diagram \eqref{id-tg} and
let $(\T^*\tb)\reg:=\mu_{\T^*\tb}\inv(\g_r)$.
Using the identifications  of  diagram \eqref{id-tg},
we obtain the following pair of $G$-equivariant  torsors on $\tgr$:
\beq{2tors}
\xymatrix{
Z\ 
\ar[dr]_<>(0.5){^{\text{$\pi^*\fZ$-torsor}}}^<>(0.4){_{\mu_Z}}&&\ (\T^*\tb)\reg\
%\ar@{}[d]|{=} 
\ar[dl]^<>(0.5){^{\text{$T_\tgr$-torsor}}}_<>(0.4){_{\tilde\varrho }}
\\
&\tgr
%\ar@{=}[rr] 
&&% \tgr\ar@{=}[rr] && \tgr
&&
}
\eeq

\begin{prop}\label{imb} There is a canonical   morphism
$\kap_Z: Z\to (\T^*\tb)\reg$, of  $G$-equivariant schemes over $\tgr$, such that
the  following  diagram commutes
\[
\xymatrix{
\pi^*\fZ \times_\tgr \wtx\
\ar[rrr]^<>(0.5){\pi^*\fZ\text{\em -action}}\ar[d]_<>(0.5){\vkap\,\times\,\kap_Z}
&&&\ \wtx\ 
\ar[d]_<>(0.5){\kap_Z}\\
T_\tgr\times_\tgr (\T^*\tb)\reg\
\ar[rrr]^<>(0.5){T_\tgr\text{\em -action}}
&&&\ (\T^*\tb)\reg
}
\]
\end{prop}

\begin{proof} We first define  $\kap_Z$ pointwise.
To this end, we use the map $\kap_{\bar\b,\b}$ from \eqref{bbo-bco}
and let $\kap_Z$ be given by the following assignment:
\beq{kxyz}\kap_Z:\ \wtx\ \to\ (\T^*\tb)_r=\tgr \times_{\bb}\tb,\quad
(\bar{\b}, \b ,x)\ \longmapsto\  \big((\b ,x),\,
(\b, \kap_{\bar{\b},\b }(x\mod\bar{\b})\big).
\eeq

To define this map scheme theoretically, 
we use   the isomorphisms $p_\Om\times p$ and $p_-\times p_\Om$
 from \eqref{wtx} and put $u:=(p_\Om\times p)\ccirc(p_-\times p_\Om)\inv$.
We obtain the following  commutative diagram:
\[
\xymatrix{
&\tbm\times_\bb\Om\ \ar[rr]^<>(0.5){u}_<>(0.5){\cong}&&\ \Om\times_\bb\tb\ \ar@{->>}[rr]^<>(0.5){pr_\tb}&&
\ \tb\ \ar[dr]^<>(0.5){\varrho}&\\
Z\ \ar[ur]^<>(0.5){\wt\pi}\ar@{.>}[rrr]^<>(0.5){\kap_Z}\ar@{=}[dr]&&&\ (\T^*\tb)\reg\ \ar@{=}[r]\ &
\ {\tgr\times_\bb\tb\en\hphantom x}\
\ar[ur]\ar[dr]^<>(0.5){\wt\varrho}\ar@{}[rr]|{\en\ \Box}&& \ \overset{}\bb\\
&\ \oox\times_{\g_r}\tgr\ \ar@{->>}[rrrr]^<>(0.5){pr_\tgr}&&&&\ \tgr\ \ar[ur]&
}
\]
In this diagram, the composite map 
$(pr_\tb\ccirc u\ccirc\wt\pi)\times pr_\tgr:\ Z\to \tb\times \tgr$
 (along the upper and lower halves of the perimeter) 
factors  through a dotted map
$Z\to \tgr\times_\bb\tb$. We let this map be
 the required morphism $\kap_Z$.

We must prove that the diagram in the statement of Proposition \ref{imb}
commutes. All schemes being reduced, it is sufficient to check this pointwise.
Thus, let $(\bar{\b}, \b ,x)\in \wtx$ and  put ${\mathfrak h}:=\b \cap\bar{\b}$.
This is a Cartan subalgebra, by Proposition \ref{genpos}, and one has
a triangular decomposition
$\g=\u \oplus{\mathfrak h}\oplus\bar{\u}$,
where $\u:=\u(\b), \bar\u:=\u(\bar\b)$.
Let $B$, resp. $U$ and $H$, be the group
corresponding to $\b$, resp. $\u$ and ${\mathfrak h}$.
Fix  $g\in G_x$. We know that
$G_x\sset B$, so conjugation by $g$
produces
a triangular decomposition of the form
$\g=\u \oplus  g({\mathfrak h})\oplus  g(\bar{\u})$.

Since $x\in \b\cap (\bco(\bar\b)+\bar\b)$,
one can write 
$x=h+u$, where $h\in{\mathfrak h}$ and $u$ is a nilpotent  element contained in the open
$H$-orbit in $\u\cap\fd(\bar\b)$.
Note that since $x$ is fixed by $g$, we have
$x= g(h)+ g(u)$, where $ g(h)\in  g({\mathfrak h})$ and $ g(u)\in  g(\u \cap\fd(\bar\b))=\u \cap\fd(g(\bar\b))$.
Put $\u'=[\u,\u]$. Thus, $x\mod\bar\b\in \bco(\bar\b)$,
resp. $u\mod\u'\in \bbo(\b)$.  Going through the
construction of the map $\kap_{\bar{\b},\b }$,
one finds that $\kap_{\bar{\b},\b }(x\mod\bar\b)=u\mod \u'$,
resp.
$\kap_{g(\bar{\b}),\b }(x\mod g(\bar{\b}))= g(u)\mod g(\u')= g(u)\mod \u'$.
Now, write $z: y\mto z\star_{_\fZ}y$, resp. $t: y\mto t\star_{_T} y$, for the action of  
$G_x$ on $\mum\inv(x)$ and $\pi\inv(x)$, 
resp.  $T=B /U$ on $\u /\u'$.
%where we have used that  $\vkap_{\b ,x}(g)=g\mod U$, by definition. We compute:
By definition, one has $\vkap_{\b ,x}(g)=g\mod U$.  Writing $\kap_{\b,x}$ for the restriction
of the map $\kap_Z$ to the fiber $\wt\mu\inv(\b,x)$, cf. \eqref{cartZ}, we compute:
\begin{align*}
\kap_{\b,x}(g\star_{_\fZ}\bar\b,\,g\star_{_\fZ}\b,  g(x))&=\  \kap_{\b,x}(g\star_{_\fZ}\bar\b,\b,x)
\ =\ \kap_{g(\bar{\b}),\b }(x\mod g(\bar{\b}))\ =\  g(u)\mod \u' \\
& =\
(g\mod U )\star_{_T} (u\mod \u' )\ =\ \vkap_{\b ,x}(g)\star_{_T}\kap_{\bar{\b},\b }(x).\qedhere
\end{align*}
\end{proof}

Let $T_\tgr\times^{\pi^*\fZ} Z$ be a pushout of the $\pi^*\fZ$-torsor $Z$ via 
the morphism $\vkap$. Thus, $T_\tgr\times^{\pi^*\fZ} Z$ is a $G\times W$-equivariant
$T_\tgr$-torsor.
It follows from Proposition \ref{imb} and the universal property of pushouts, that
the morphism $\kap_Z$ induces a morphism
\beq{T-tors}
\kap:\ T_\tgr\times^{\pi^*\fZ} Z\too (\T^*\tb)_r,\quad
\ \big(t,\ (\bar{\b}, \b ,x)\big)\  \mto\
t\cdot\kap_Z(\bar{\b}, \b ,x),
\eeq
of $G$-equivariant $T$-torsors on $\tgr$. 
Since
any morphism of torsors is an isomorphism, we obtain
\begin{prop}\label{tors-thm} The map $\kap$ is an isomorphism of $G$-equivariant
$T$-torsors on $\tgr$. \qed
\end{prop}

\begin{proof}[Proof of Theorem \ref{thm*}]
We view $T\times \t^*$ and $\xx$  as schemes over $\fc$ via the maps
$T\times \t^*\to \t^*\to\fc$
and $\xx\to \g_r\to \fc$, respectively, cf. Lemma \ref{torsor}.
We find
\begin{align}
 T_\tgr\times_\tgr Z
% =\ (T\times \tgr)\times_\tgr Z\
 % &=\ T\times Z\
 % &\overset{\eqref{oxx_fib}}{\ccong}\
     \ccong(T\times \tgr)\times_\tgr  (\xx\times_\fc\t^*)
  \ =\ (T\times \t^*)\times_\fc\xx\ =\ \T^*T\times_\fc\xx.\label{TZ}
 %   \ =\ \T^*T\times_\fc\T^\psi\tbm.\nonumber
\end{align}

Now, we have  a diagram
$T_\tgr\times_\tgr Z\xrightarrow{g}
T_\tgr\times^{\pi^*\fZ} Z\xrightarrow{h}\g_r$, where
the map $g$ is an $(h\ccirc g)^*\fZ$-torsor by construction.
Using the isomorphism $\kap$ of Proposition \ref{tors-thm}
 we obtain a chain of  $G$-equivariant morphisms
%   \beq{star1}
%     \xymatrix{
%         \T^*T\times_\fc\xx\   \ar[r]^<>(0.5){\  \eqref{TZ}\  }_<>(0.5){\cong} &
%     \  T_\tgr\times_\tgr Z  \ar[r]^<>(0.5){g}& T_\tgr\times^{\pi^*\fZ} Z\  
%     \ar[r]^<>(0.5){\  \kap\inv\  }_<>(0.5){\cong} & \  (\T^*\tb)_r\
%     \ar@{^{(}->}[r]^<>(0.5){j}&
%    \T^*\tb \ar[r]^<>(0.5){{\mu_{\T^*\tb}}} & \g\  
%     \ar[r]^<>(0.5){\vpi}&\  \fc,
%   }
%   \eeq
\beq{star1} \T^*T\times_\fc\xx\  \xrightarrow[\cong]{\eqref{TZ}}\
  T_\tgr\times_\tgr Z \   \xrightarrow{g}\
  T_\tgr\times^{\pi^*\fZ} Z\  \xrightarrow[\cong]{\kap\inv}\
  (\T^*\tb)_r\  \stackrel{j}\into \ \T^*\tb\   \xrightarrow{\mu_{\T^*\tb}}\
  \g\ \xrightarrow{\vth}\ \fc,
  \eeq
where  $j: (\T^*\tb)_r\into \T^*\tb$ is an open imbedding.
Let  $\underline{\mu}: \T^*\tb\xrightarrow{{\mu_{\T^*\tb}}}\g^* \xrightarrow{\vth}\fc$, resp.
$\underline{\kap}:   \T^*T\times_\fc\xx\to(\T^*\tb)_r$, be
the composite map.
  We obtain  the following chain of $G$-equivariant  algebra maps

\beq{tors map}
\C[\fc]\ \xrightarrow{\underline{\mu}^*}\ \C[\T^*\tb]\ \xrightarrow{j^*} \ \C[(\T^*\tb)_r]\
\xrightarrow{\underline{\kap}^*}
\C[\T^*T\times_\fc \xx] \ccong \C[\T^*T]\otimes_{\C[\fc]} \C[\xx].
\eeq
where 
the  isomorphism on the right holds since the varieties $\T^*T, Z$, and $\fc$
are affine.

Since the map $g$ is an $(h\ccirc g)^*\fZ$-torsor,
the map $\underline{\kap}$ is
a  $\underline{\mu}^*\fZ_\fc$-torsor.
Therefore, the algebra map $\underline{\kap}^*$
induces an isomorphism
$\C[(\T^*\tb)_r]\iso \C[\T^*T\times_\fc \xx] ^{\underline{\mu}^*\fZ_\fc}\,=\,
\C[\T^*T\times_\fc \xx] ^{\fZ_\fc}$.

  We claim next that the map $j^*$    is an isomorphism.
   To prove this, we use that  $\tg\sminus \tgr$ has codimension $\geq2$ in $\tg$,
see e.g.
\cite[Proposition 1.9.3]{BR}.
Since $\T^*\tb$ is a $T$-torsor over $\tg$, we conclude that $(\T^*\tb)\sminus (\T^*\tb)_r$ 
has codimension $\geq2$ in $\T^*\tb$, and the claim follows.

Combining everything together  we obtain $G$-equivariant
isomorphisms of $\C[\fc]$-algebras 
\[
  \C[\T^*\tb]\ \xrightarrow[\cong]{(j\ccirc\underline{\kap})^*}\
  \C[\T^*T\times_\fc \xx] ^{\fZ_\fc}\ccong
  \big(\C[\T^*T]\otimes_{\C[\fc]} \C[\xx]\big)^{\fZ_\fc}.
  \]

To complete the proof,  let $Y$ be an arbitrary 
symplectic manifold  equipped
with a Hamiltonian $\fZ_\fc$-action with moment map $\mu_Y$,
and let $\C[Y]^{\C[\fc]}$  be the Poisson centralizer
of $\mu_Y^*(\C[\fc])$ in $\C[Y]$. Then, equation \eqref{df}
and the fact that the group scheme $\fZ_\fc$ is connected
imply that one has
$\C[Y]^{\fZ_c}=\C[Y]^{\C[\fc]}$. Applying this in the  case
$Y=\T^*T\times_\fc \xx$, from the previous isomorphisms we deduce
\beq{final iso}
 \C[\T^*\tb]\ccong
  \C[\T^*T\times_\fc \xx] ^{\C[\fc]}\ccong
  \big(\C[\T^*T]\otimes_{\C[\fc]} \C[\xx]\big)^{\C[\fc]}.
  \eeq
 The theorem follows since $\T^\psi(G/\bU)=\T^\psi\tb_-=\xx$ by definition, see \eqref{psix}.
\end{proof}

\begin{proof}[Proof of Corollary \ref{Wcor}] 
The Weyl group acts on $T_\tgr$ and on $Z$.
We transport the diagonal $W$-action on $T_\tgr\times^{\pi^*\fZ} Z$ via the isomorphism
$\kap$. The resulting   $W$-action on $(\T^*\tb)_r$ induces a
  $W$-action on $\C[\T^*\tb]=\C[(\T^*\tb)_r]$ by algebra automorphisms.

To prove that 
the  above defined  $W$-action on $(\T^*\tb)\aff$ agrees
with  the one constructed in
\cite[Proposition 5.5.1]{GR},
it is sufficient to show that the two actions agree over the locus of regular semisimple elements.
This is easy to check  by comparing the constructions of these actions.
\end{proof}

\begin{cor}\label{kz_fibers} In the setting  of  Lemma  \ref{bk-lem} the following holds:

\vi One has $\Ker\vkap_{\b,x}=U_x$,
resp.  $\im\vkap_{\b,x}$  equals the image of the group $Z(G_h)$ in ~$T=B/[B,B]$.

\vii If $s\in \bbo(\b)$ is such that the fiber $\kap_Z\inv(\b,x,s)$ is nonempty, then
this fiber is a $U_x$-torsor.
\end{cor}

\begin{proof} Part (i) is immediate from  Lemma  \ref{bk-lem}. Part (ii) follows from (i)
using that the map \eqref{T-tors} is an isomorphism.
\end{proof}

\subsection{}\label{pf*} 
Recall the
map $q:\xx\to\tbm$, see \eqref{moment},
and let $pr_\tb: \T^*\tb\to\tb$ be the natural projection.
We also  consider the  map
 $\kz=\wt\pi\times \kap: Z \to \xx\times \T^*\tb$.
By construction, we have $\im\kz\sset\xx\times(\T^*\tb)_r$.
Formula \eqref{kxyz} shows that image of the composite map $(q\times pr_\tb)\ccirc\kz:\
Z \to \xx\times \T^*\tb\to \tbm\times\tb$, equals $\Xi$.
In terms of  \eqref{zz2}, we obtain the following diagram:
\beq{Y}
\xymatrix{
\oox=G\times_{U_\bbf} (\bbe+\bbb)\ \ar[d]^<>(0.5){q} &&
Z=G\times (\bbe+\g_\bbh)
\ \ar[ll]_<>(0.5){\wt\pi}
\ar[rr]^<>(0.5){\kap}\ar[d]^<>(0.5){(q\times pr_\tb)\ccirc\kz}&& \ \T^*(G/U_\bbe)=G\times_{U_\bbe}
  \b_\bbe \
\ar[d]^<>(0.5){pr_\tb }\\
\tbm=G/{U_\bbf} && \ \Xi=G\ \ar@{->>}[ll]_<>(0.5){p_-}
\ar@{->>}[rr]^<>(0.5){p}&&
\ \tb=G/U_\bbe
}
\eeq
In this diagram, the map $\wt\pi$ is induced
by the inclusion $\bbe+\g_\bbh\into \bbe+\bbb$,
resp. the map $\kap$ is induced
by the inclusion $\bbe+\g_\bbh\into \b_\bbe$.  
Further, the maps $\kz$ and $(q\times pr_\tb)\ccirc\kz$ take the following form:
\[
\xymatrix{
Z=G\times (\bbe+\g_\bbh)
\ar@{^{(}->}[rrr]^<>(0.5){\kz:\en (g,x)\,\mto\,(g,x,x)}&&&
(q\times pr_\tb)\inv(\Xi)=G\times (\bbe+\bbb)\times\b_\bbe
\ar@{->>}[rrr]^<>(0.5){q\times pr_\tb:\ (g,x,y)\,\mto\,g}&&&  G.
}
\]

\begin{prop}\label{lag}
The map $\kz$ is a
$G$-equivariant closed imbedding that makes the first projection
$Z\to \oox$  a finite morphism and the second  projection
 $Z\to \T^*\tb$ a birational isomorphism.
Furthermore, the  map $(q\times pr_\tb)\ccirc\kz:\ Z\to \Xi\ $
is a fibration  with  affine-linear fibers.
\end{prop}

\begin{proof}  
We claim  first that   the restriction
of the map  $Z\to \T^*\tb$  to the  regular semisimple locus
is an isomorphism. To see this, consider the following maps
\[
\xymatrix{
Z_{rs}\ \ar[rr]^<>(0.5){z\ \mto\ (1, z)}&&
\ (\pi^*\fZ\times^{\pi^*\fZ} Z)_{rs}\ 
\ar[rr]^<>(0.5){(\vkap|_{\tg_{rs}})\times\Id_Z}&&
\ (T_\tgr\times^{\pi^*\fZ} Z)_{rs}\ 
\ar[r]^<>(0.5){\kap|_{rs}}&
\ (\T^*\tb)_{rs}.
}
\]
The first map above  is the tautological isomorphism.
As we have mentioned in the proof of 
Lemma \ref{imvkap},  the second map above
is  easily seen to be an isomorphism.
 Finally, the third map
is  an isomorphism by Proposition \ref{imb}.
We conclude that the map  $Z\to \T^*\tb$ is a birational isomorphism.
Furthermore, the left cartesian square in \eqref{cartZ} implies that  the first projection
$Z\to \oox$ is a finite morphism.

The scheme $\wt\Xi:=(q\times pr_\tb)\inv(\Xi)$ is closed in $\oox\times \T^*\tb$
since
the orbit $\oom$ is  closed in $\tbm\times\tb $,
by Lemma \ref{closed}. 
Thus, proving that $\kz$ is a closed imbedding
reduces to showing that so is the map $Z\to \wt\Xi$.
We  identify $\Xi=G$ and use the explicit formulas for
the maps $\kz$ and $q\times pr_\tb$ given above.
The fiber of the projection $\wt\Xi\to\Xi$,
resp.  of the map
$(q\times pr_\tb)\ccirc \kz$, over an element $g\in G=\Xi$
may be identified with $(\bbe+\b_\bbf)\times \b_\bbe$, resp.
$\bbe+\g_\bbh$. Then, the restriction, $\kz_1$, of the map $\kz$ to the fiber 
becomes the diagonal imbedding
$\kz_1: \bbe+\g_\bbh\into (\bbe+\b_\bbf)\times \b_\bbe$, which is 
a closed imbedding. It follows that $\kz$  is 
a closed imbedding.

Next, we note that  $(\bbe+\b_\bbf)\times \b_\bbe$ has the natural structure of an affine-linear space,
a fiber of the twisted cotangent bundle on $\tbm\times\tb$.
The image of the map $\kz_1$  is an affine-linear
subspace of that  affine-linear space.
Thus, the map $(q\times pr_\tb)\ccirc\kz:\ Z\to \Xi$ is an affine-linear fibration.
\end{proof}

\begin{rem} \label{lagr-corr}
The symplectic structures on $\T^\psi\tbm$ and $\T^*\tb $
give  ${\T^\psi\tbm\times \T^*\tb}$ a natural symplectic structure.
With that structure,  we can  view $\T^\psi\tbm\times \T^*\tb $  as 
a twisted cotangent bundle on $\tbm\times\tb$.
We use the isomorphism $\xx\cong \T^\psi\tbm$ and identify
$\xx\times \T^*\tb$ with  $\T^\psi\tbm\times \T^*\tb$.
Then, one can show that     $q\times pr_\tb:Z\to \Xi$
is a twisted  conormal bundle on $\Xi$.
In particular, the Miura variety $Z$ is  a smooth Lagrangian 
correspondence between the symplectic manifolds $\T^\psi\tbm$ and $\T^*\tb$.
\end{rem}

\subsection{The case of a simply connected group}\label{sc_sec} 
In this subsection, we explain how to adapt the constructions of previous sections
to a `simply connected' setting where the adjoint group $G$ is replaced by $G\sc$,
a simply connected cover of $G$. 
To this end, let $\om_i,\ i\in I$, be the   fundamental weights of $\g$,
and ${\mathbb{X}}^*$ the weight lattice.
For each $i\in I$, we fix an irreducible  finite dimensional representation $V_i $ of $\g$,
 with highest weight
$\om_i$ (such a representation is defined uniquely up to a noncanonical isomorphism). 
Associated with  every Borel $\b$, there is a canonical $\b$-stable filtration
$V_i ^{\geq\mu,\b},\ \mu \in {\mathbb{X}}^*$,  such that 
$\gr^{\mu,\b} V_i :=V_i ^{\geq\mu,\b}/V_i ^{>\mu,\b}$ is a $\mu$-weight space for the natural action of 
the universal Cartan algebra $\t=\b/[\b,\b]$. 
The highest weight space $V_i ^{\geq\om_i,\b}$ is
the line $V_i ^{\u(\b)}$  formed by the vectors killed by $\u(\b)$. Dually, one has a line
$\gr^{w_0(\om_i),\b} V_i =V_i /\u(\b)V_i $,
where $w_0\in W$ is the longest element.
Let $\ooo^i(\b)\sc=V_i ^{\geq\om_i,\b}\sminus\{0\}$, resp. $\ooo^i_-(\b)\sc=
\gr^{w_0(\om_i),\b} V_i \sminus\{0\}$.
The action of $T\sc$,
the abstact maximal torus of $G\sc$, makes $\ooo(\b)\sc:=\prod_i\ \ooo^i(\b)\sc$, 
resp. $\ooo_-(\b)\sc=\prod_i\ \ooo^i_-(\b)\sc$,
a $T\sc$-torsor.
Similarly to \S\ref{inv}, we define the following $G\sc$-equivariant $T\sc$-torsor on $\bb$:
\[
\wt\bb\sc:=\{(\b, s)\ \mid  \b\in \bb, \en s\in\bbo(\b)\sc\},\resp
\tbm\sc:=\{(\b, s)\ \mid  \b\in \bb, \en s\in\bbo_-(\b)\sc\}.
\]

For  every $i\in I$ and a pair  $\b,\bar{\b}\in \bb$ of Borel subalgebras in opposite position,
the composite map $V_i ^{\u(\b)}\into V_i \onto V_i /\u(\bar\b)V_i $ is
an isomorphism. An inverse of this map induces an isomorphism
\[\bco^i(\bar{\b})\sc= (\gr^{w_0(\om_i),\bar{\b}} V_i )\sminus\{0\}\
\iso\
\bbo^i(\b)\sc=V_i ^{\geq\om_i,\b}\sminus\{0\}.
\]
We obtain an isomorphism
$\kap_{\bar{\b},\b}\sc: \bco(\bar{\b})\sc\iso \bbo(\b)\sc$,
 such that $\kap_{\bar{\b},\b}\sc(ts)=w_0(t)\inv\kap_{\bar{\b},\b}\sc(s)$.
We define
\[
\Xi\sc\ :=\ \{(\bar{\b},\bar{x},\b,x)\in \tbm\sc \times\tb\sc\mid
(\b ,\bar{\b})\in \Om,\ \kap_{\bar{\b},\b }\sc(\bar{x})=x \}.
\]

Let
$p_\Om\sc: \Xi\sc\to \Om,\,(\bar{\b},\bar{x},\b,x)\mto(\bar{\b},\b)$,  resp.
$p_-\sc: \Xi\sc\to \tbm\sc,\,(\bar{\b},\bar{x},\b,x)\mto(\bar{\b},\bar{x})$, and
$p\sc:\Xi\sc\to\tb\sc,\,(\bar{\b},\bar{x},\b,x)\mto(\b,x)$, be the natural projection.
Mimicing Section \ref{obb}, one shows
that  $\Xi\sc$ is a $G\sc_\Delta$-torsor
and one has the following isomorphisms, cf.  \eqref{wtx},
\[
\xymatrix{
\tbm\sc\times_{\bb}\bbm\ &&\ 
\oom\sc\ \ar[ll]_<>(0.5){p_-\sc\times p_\Om}
\ar[rr]^<>(0.5){p_\Om \times p\sc} &&\
\bbm\times_{\bb}\tb\sc .}
\]

Let  $T\sc$ act on $\tbm\sc \times\tb\sc$ via the $T\sc$-action
on the second factor $\tb\sc$ (along the fibers of  $\tb\sc\to\bb$).
The actions of $G\sc_\De$ and $T\sc$  make $\tbm\sc \times\tb\sc$
a $G\sc_\De\times T\sc$-variety.
Let  $\wt\Om\sc$ be the preimage of $\Om$ under the map $\tbm\sc \times\tb\sc\to\bb\times\bb$.
The variety  $\wt\Om\sc$  is $G\sc_\De\times T\sc$-stable
and it contains $\Xi\sc$.
Furthermore, the map $\Xi\sc\times T\sc\to \wt\Om\sc$
induced by the $T\sc$-action 
is  a $G\sc_\Delta\times T\sc$-equivariant isomorphism,
cf. Lemma \ref{closed}(i).

Since  any Borel subgroup $B$ of $G\sc$ contains $Z(G\sc)$,
the center of  $G\sc$,
there is  a
canonical imbedding $Z(G\sc)\into B/[B,B]=T\sc$. Furthermore, 
we have  canonical
isomorphisms  $G\sc/Z(G\sc)\iso G$, resp. $T\sc/Z(G\sc)\iso T$.

Recall the notation of \S\ref{inv}. 
For every $i\in I$ inside the lattice ${\mathbb{X}}^*$, we have an  equation
$\al_i=\sum_{j\in I}\
\langle \al_j^\vee,\al_i\rangle \cdot\vpi_i$, where
 $\al_j^\vee$ denotes the simple coroot associated 
with $j\in I$. We may (and will) choose,   for  some Borel subalgebra $\b$ and
every $i\in I$, an
 isomorphism:
\[\gr^{\al_i,\b}\g\ccong \bo_{j\in I}\
\big(V_i ^{\u(\b)}\big)^{\o \langle \al_j^\vee,\al_i\rangle},
\]
of 1-dimensional $T\sc$-modules (of the same weight).
Combining these  isomorphisms together we obtain  a  $T\sc$-equivariant morphism
\[
\bbo(\b)\sc=\prod_{i\in I}\ \bbo^i(\b)\sc\too  
\bbo(\b)=\prod_{i\in I}\ \bbo^i(\b),
\quad (s_i)_{i\in I}\ \mto\ \Big(\mbox{$\underset{j}\bigotimes\ s_i^{\o \langle \al_j^\vee,\al_i\rangle}$}\Big)_{i\in I}.
\]

The  morphism above descends to an isomorphism
$\bbo(\b)\sc/Z(G\sc)\iso\bbo(\b)$, of $T$-torsors. Furthermore, this isomorphism
for the chosen Borel $\b$ induces  an
isomorphism
$\tb\sc/Z(G\sc)\iso\tb$, of $G$-equivariant $T$-torsors
on $\bb$. 
A similar construction yields an
isomorphism
$\tbm\sc/Z(G\sc)\iso \tbm$, of $G$-equivariant $T$-torsors
on $\bb$.

We  define 
$\xx\sc:= \xx\times_{\tbm}\tbm\sc$, resp.
$Z\sc:=\xx\sc\times_{\g_r}\tgr$.
Let $\mu_\xx\sc$ be the composition $\xx\times_{\tbm}\tbm\sc\to \xx\to\g$.
With these definitions, there is an analogue of 
the diagram from  the proof of Proposition \ref{imb} that provides the  construction
 of a  morphism 
\[\kap_Z\sc: Z\sc\to \tgr\times_{\bb}\tb\sc\cong \T^*(\tb\sc),
\]
of $G\sc$-equivariant schemes on $\tgr$.

Next,
one  has the
 universal centralizer group scheme
$\dis \fZ\sc :=\{(g,x)\in G\sc\times\g_r\mid \Ad g(x)=x\}$.
Using the $G\sc$-equivariant morphism  $\mu_\xx\sc: \xx\sc\to\g_r$,  one gets a canonical
action 
$\fZ\sc\times_{\g_r} \xx\sc \to \xx\sc$.
There is a simply-connected analogue $\vkap\sc: \pi^*\fZ\sc\to T\sc_\tgr$,
of the map from Lemma \ref{imvkap}.
One checks
that the  simply-connected counterpart of the diagram of Proposition \ref{imb} commutes.
This implies an analogue of Proposition \ref{tors-thm}, i.e., 
an  isomorphism
\beq{tors-sc-thm}\kap\sc:\ T\sc_\tgr\times^{\pi^*\fZ\sc} Z\sc\ \iso\ (\T^*(\tb\sc))_r,
\eeq
of $G\sc\times T\sc$-equivariant schemes on $\tgr$.

We let $W$ act on $\xx\sc\times_\fc\t$ via its action on the second factor.
This gives, using the isomorphism $\xx\sc\times_{\g_r}\tgr\cong \xx\sc\times_\fc\t$,
 a  $W$-action on $Z\sc$.
Thus, one can transport the diagonal $W$-action on the variety in the LHS
of \eqref{tors-sc-thm}
to obtain a  $W$-action on $(\T^*(\tb\sc))_r$.

\section{The Miura bimodule}\label{6}

\subsection{} \label{mi-sec}
Given a Lie algebra $\fk$  and a
left, resp. right, $\fk$-module $E$,
we write $E^\fk$ for the space of $\fk$-invariants
and use simplified notation 
 $\fk\back E=E/\fk E$, resp. $E/\fk=E/E\fk$,
for $\fk$-coinvariants.
Similarly, given a second  Lie algebra $\fk'$ and a
$(\UU\fk',\UU\fk)$-bimodule
$E$,  we write  $\fk'\back E/\fk=
E/(E\fk+\fk'E)$.

Below, we view $\ug$, resp. $\ut$, as the algebra of  left invariant, resp. invariant,
differential operators on $G$, resp. $T$.
Fix a Borel subgroup $B$ with unipotent radical $U$. Let 
$\b=\Lie B,\, \u=\Lie U$. Since $\u$ is a Lie ideal of $\b$ and the Lie algebra $\t=\b/\u$ is abelian,
we have $\ut=\UU\b/\u=\u\back\UU\b=(\u\back\UU\b)^\u=
((\u\back\UU\b)^\u)^{\t}$,
where $(-)^\t$ stands for invariants of the adjoint $\t$-action.
We obtain
a chain of algebra imbeddings
\beq{hcdef}
\ut\ =\ ((\u\back\UU\b)^\u)^{\t}\ \into
\ ((\u\back\ug)^\u)^{\t}\into (\u\back\ug)^\u\ \into\  (\u\back\dd(G))^{\u}.
\eeq

We put $\dds:= (\u\back\dd(G))^{\u}$, let
$a:\ut\to\dds$ be the composite of the maps above,
and $\text{proj}: \UU\b\onto \ut=\u\back\UU\b$
the projection.
Thus, one has the following diagrams of algebra maps
\[
\xymatrix{
  \dd(G)\  &\  \UU\b\  \ar[l]\ar[rr]^<>(0.5){_{i_T\ccirc\text{proj}}}
&&\   \dd(T),\quad\text{resp.}\quad\dds\  &\  \ut\  \ar[l]_<>(0.5){a}
\ar[r]^<>(0.5){i_T}&\   \dd(T).
}
\]

We apply the construction of Hamiltonian reduction in the setting of Example \ref{examps}(ii) for
the triple $A_1=\dd(T),\, A_2=\dd(G),\,Z=\uub$, resp.
$A_1=\dd(T),\, A_2=\dds,\,Z=\ut$,
and the maps on the left, resp. right, of the  diagram above.
We obtain the following algebras
\beq{algs}\big(\dd(T)\ \bo_{\uub}\ \dd(G)\big)^{\b}\ccong \big(\dd(T)\ \bo_{\ut}\ (\u\back\dd(G))^{\u}\big)^\t\ =\
(\dd(T)\ \bo_{\ut}\ \dds\big)^\t,
\eeq
where the isomorphism on the left results from the fact that one can perform Hamiltonian reduction
in stages: first with respect to the Lie algebra $\u$ and then
with respect to $\t=\b/\u$. 
Hamiltonian reduction  with respect to  $\u$ does not affect $\dd(T)$, the first tensor factor, since
the map $\b\to\dd(T)$ factors through $\b/\u\to\dd(T)$.
The space $\dd(T)\,\o_{\uub}\,\dd(G)=\dd(T)\, \o_{\ut}\, (\u\back\dd(G))$ has the structure of 
a left $(\dd(T)\, \o_{\uub}\, \dd(G))^{\uub}$-module via an `inner' action, and also of a
right $\dd(T)\opp\o\dd(G)$-module via an `outer' action, see Example ~\ref{examps}(ii).

Fix an opposite Borel subgroup $\bar B$ with Lie algebra $\bar\b$.
Let $\bU=[\bar B,\bar B]$, resp. $\bu=[\bar\b,\bar\b]$.
In \S\ref{ss1.2} we have introduced the Lie subalgebra 
$\up\subset\dd(G)$ and the ring  $\dd^\psi(G/\bU):=(\dd(G)/\up)^\up$.
The two-sided quotient $\u\back\dd(G)/\up=
(\u\back\dd(G))/\up$ $=\u\back(\dd(G)/\up)$ has the natural structure of a $(\dds,\ddp)$-bimodule.

Let the Lie algebra $\up$  act on $\dd(T)\,\o\, \dd(G)$  by 
$x:\,(u_T\o u_G)\mto u_T\o (u_G\cdot x)$.
We
define the Miura bimodule  as follows:
\begin{align}
\BM\ :=\ \big(\dd(T)\,\o_{_{\uub}}\,\dd(G)\big)/\up
&\ =\
\big(\dd(T)\, \o_{_{\ut}}\,(\u\back\dd(G))\big)/\up\label{miu1}\\
 &\ =\
\dd(T)\,\o_{_{\ut}}\,(\u\back\dd(G)/\up).\nonumber
\end{align}
The `inner' action of the algebra $(\dd(T)\, \o_{\uub}\, \dd(G))^{\b}$ on  $\dd(T)\,\o_{\uub}\,\dd(G)$
survives in $\BM$ and it makes $\BM$ a left module over any of the algebras in
\eqref{algs}.
The  `outer' action of $\dd(T)\opp\o\dd(G)$ descends to a right action  of the 
algebra $\dd(T)\opp\o\ddp$, on $\BM$.   The right action commutes with the left action;
furthermore, it
makes $\BM$ a $(\dd(T),\,\ddp)$-bimodule.
Let $1_\BM\in\BM$ be the image of the element $1\o 1\in\dd(T)\otimes \dd(G)$ in $\BM$.

Recall that the Harish-Chandra homomorphism $\hc: Z\g\to \ut$,
where  $Z\g$ denotes the center of the algebra $\ug$,
may be defined as follows.
First, one observes that the image of $Z\g$
under the projection $pr: \ug\to\u\back\ug$ 
is clearly contained in the subspace
$((\u\back\ug)^\u)^{\ad\t}$. On the other hand,
weight considerations yield  an equality
$((\u\back\UU\b)^\u)^{\t}= ((\u\back\UU\g)^\u)^{\t}$
(in fact, it is known that $((\u\back\UU\g)^\u)^{\t}=(\u\back\UU\g)^\u$,
see Lemma \ref{vlem}(iii)  below).
We obtain an algebra isomorphism $\sigma: \ut\iso ((\u\back\UU\g)^\u)^{\t}$,
cf. \eqref{hcdef}.
The Harish-Chandra homomorphism
is then defined as a composition
$Z\g\xrightarrow{\,pr\,} ((\u\back\UU\g)^\u)^{\t}\xrightarrow{\,\sigma\inv\,}\ut$.

The imbedding $\ug\into \dd(G)$ provides an isomorphism of $Z\g$
and the algebra of  bi-invariant differential
operators on $G$. This  isomorphism  descends to an algebra map
$Z\g\to\dds, z\mto \dag z$,
resp. ${Z\g\to\dd^\psi(G/\bU), z\mto z^\psi}$.
The assignment  $z_1\o z_2\mto i_T(\hc(z_1))\o z_2^\psi$ gives an imbedding of
$Z\g\o Z\g$  as a subalgebra of $\dd(T)\o \ddp$.
Thus, we may (and will)  view $\BM$  as 
 a $(Z\g,Z\g)$-bimodule using the  `outer' action of this subalgebra.

 Let $\BM^{Z\g,\out}:=\BM^{Z\g}$, where
 we follow the notation of Example \ref{examps}(ii) and for any  $(Z\g,Z\g)$-bimodule $E$
 write
 $E^{Z\g}:=\{x\in E\mid zx=xz,\  \forall z\in Z\g\}$.

\begin{lem}\label{Z-lem} We have $1_\BM\in \BM^{Z\g,\out}$.
\end{lem}

\begin{proof} 
Fix  $z\in Z\g$.  It is well known (and follows from Lemma \ref{trick}(i) below)
that inside $\dds$
one has
$\dag z=a(\hc(z))$.
Hence, 
inside  $(\dd(T)\o_{\UU\t}\dds)^\t$ we find
\beq{dds2}
i_T(\hc(z))\o1\ =\ 1\o a(\hc(z))\ =\ 1\o \dag z,\qquad\forall z\in Z\g.
\eeq

It is clear that inside
$\u\back\dd(G)/\up$, we have $\dag z 1=1 z^\psi$. 
Hence, 
writing $z_T:=i_T(\hc(z))$,
 inside $\BM=\dd(T)\,\o_{_{\ut}}\,(\u\back\dd(G)/\up)$, we obtain
\[z1_\BM\ =\  z_T(1\o1)1\ =\ 1(z_T\o1)1\ \overset{\eqref{dds2}}{=\!=\!=}\ 
  1(1\o \dag z)1\ =\ 1(1\o1)z^\psi\ =\ 1_\BM z,
  \]
  where we have used the notation
  $b(1\o 1)b'$, resp. $1(b\o b')1$, for the outer, resp. inner, action
on $1_\BM=1\o 1$
of an element $b\o b'$ of $\dd(T)^{op}\o\ddp$, resp. $(\dd(T)\o_{\UU\t}\dds)^\t$.
\end{proof}

\subsection{}
Let $E$ be a $\ug$-module.
For any   Lie subalgebra $\fk\sset \ug$, the $Z\g$-action on $E$
survives in $E^\fk$, resp. $\fk\back E$.
In the case $\fk=\u$, there is also a $\ut$-action $a: x\mto a* x$ on
 $E^\u$, resp. $\u\back E$, induced by the  $\UU\b$-action on $E$.

The following
is a simple consequence of results of Kostant \cite{Ko2}.

\begin{lem}\label{trick} \vi For any $z\in Z\g$ and    $x\in
  E^\u$, resp. $x\in \u\back E$, one has $zx=\hc(z)*x$.
  In particular, there is a well defined map $\ut\o_{Z\g} (\u\back E)\to  \u\back E$,
  given by the $\ut$-action.

\vii If  $E$ is locally nilpotent
as a $\up$-module then the composite map
\[\textsl{action}\ccirc(\Id_\ut\o f_E):\
  \ut\o_{Z\g} E^\up\ \xrightarrow{\Id_\ut\o f_E} \  \ut\o_{Z\g} (\u\back E)\
  \xrightarrow{\textsl{action}}\  \u\back E,\]
 where $f_E$ is the composition
 $E^\up\into E \onto \u\back E$,
is
 an isomorphism of  $\ut$-modules.
\end{lem}
\begin{proof}
  Part (i)  is immediate  from the construction of the
  Harish-Chandra homomorphism.
  Now, let $E$ be as in (ii). 
It was shown in \cite[\S3]{Ko2} that the natural map
${\UU\g/\up\o_{Z\g} E^\up\to E}$ is an isomorphism.
In the case where $E$ is finitely genenerated as a  $\UU\g$-module
a proof of this isomorphism can also be  found  in \cite{GG}, Theorem 6.1.
To prove  the general case, choose a family $E_\al$ of 
 finitely genenerated  $\UU\g$-submodules of $E$, such that $E=\underset{^{\too}}\lim\, E_\al$.
It is clear that we have $E^\up=\underset{^{\too}}\lim\, (E_\al)^\up$.
We obtain 
\[\UU\g/\up\o_{Z\g} E^\up\ =\
\UU\g/\up\o_{Z\g} \big(\underset{^{\too}}\lim\, (E_\al)^\up\big)\ =\
\underset{^{\too}}\lim\, \big(\UU\g/\up\o_{Z\g} (E_\al)^\up\big)\ =\
\underset{^{\too}}\lim\, E_\al\ =\ E,
\]
where the second equality holds  since the functor $\UU\g/\up\o_{Z\g} (-)$ commutes with direct limits.

From the above isomorphism, we deduce
\begin{align*}
 \u\back E\ =\  \u\back(\ug\o_{\ug} E)\ =\ (\u\back\ug)\,\o_{\ug}\, E&=
\u\back\ug\,\o_{\ug} \,\big(\UU\g/\up\,\o_{Z\g}\, E^\up\big)\\
&=
(\u\back\ug/\up)\ \o_{Z\g}\, E^\up=\ut\o_{Z\g}\, E^\up,
\end{align*}
where we have used that the composite $\ut=\u\back\UU\b\into
\u\back\ug\onto\u\back\ug/\up$ is an isomorphism.
\end{proof}

Let   $E:=\dd(G)/\up$, so $\u\back E=\u\back\dd(G)/\up$.
We have natural maps
\[
\dd(T)\,\oinn_{Z\g}\, (\dd(G)/\up)^\up\,\xrightarrow{\Id_{\dd(T)}\o f_E}\
\dd(T)\,\oinn_{Z\g}\, (\u\back\dd(G)/\up)\,\onto\,\dd(T)\,\o_\ut\,(\u\back\dd(G)/\up).
\]

\begin{cor}\label{big} The composition of the above maps yields
an isomorphism $\dd(T)\, \oinn_{Z\g}\,  \dd^\psi(G/\bU)\, \iso\, \BM$, of $(\dd(T),\,\dd^\psi(G/\bU))$-bimodules
with respect to the outer action.
\end{cor}
\begin{proof} Observe  that the action of $\up$ on $E=\dd(G)/\up$ by left multiplication
is locally-nilpotent.
Hence, we compute
\begin{align*}
\BM&\ =\ \dd(T)\o_{\ut}(\u\back\dd(G)/\up)\qquad\text{(by \eqref{miu1})}%\ccong \dd(T)\, \o_{\ut}\, \ut \,\o_{Z\g}\,(\u\back\dd(G)/\up)
\nonumber\\
&\ccong \dd(T)\, \o_{\ut}\, \big(\ut \,\oinn_{Z\g}\,(\dd(G)/\up)^\up\big)\qquad\text{(by Lemma \ref{trick})} \\
&\ccong \dd(T)\ \oinn_{Z\g}\ (\dd(G)/\up)^\up\ =\ \dd(T)\ \oinn_{Z\g}\ \dd^\psi(G/\bU).\qedhere
\end{align*}
\end{proof}

The isomorphism of the corollary restricts to
an isomorphism
\beq{zgb}
\big(\dd(T)\, \oinn_{Z\g}\,  \dd^\psi(G/\bU)\big)^{Z\g, \out}\ccong \BM^{Z\g, \out}.
\eeq

For  $\mu\in\BX $, let $t^\mu\in\C[T]$ denote the character $\mu$ viewed as a regular function on $T$.
The algebra  $\dd(T)$  has a weight decomposition
$\ \dd(T)=\oplus_{\mu\in\BX }\  t^\mu\cdot \ut=\oplus_{\mu\in\BX }\  \ut\cdot t^\mu$,
with respect to the adjoint $\t$-action, equivalently, $T$-action by translations.
Similarly, the adjoint $\t$-action on $\dds$ is semisimple, so  the algebra $\dds$ has a weight grading:
\[\dds=\oplus_{\mu\in\BX }\  \dds^\mu,\qquad \dds^\mu:=\{u\in\dds\mid a(h) \cdot u-u\cdot a(h)=\mu(h) u,\ \forall h\in \t\}.
\]
It is clear that each of the spaces $\dds^\mu$ is stable under the $\ut$-action 
$h: u\mto a(h)\cdot u$, on
$\dds$. So, one has direct sum decompositions
\[(\dd(T)\ \bo_{\ut}\ \dds)^\ut\ =\
\bplus_{\mu\in\BX }\  (\C t^{-\mu}\cdot\ut)\ \o_{\ut}\ \dds^\mu
\ =\
\bplus_{\mu\in\BX }\  (\C t^{-\mu}\o \dds^\mu).
\]
Furthermore, 
it is easy to check that the following map:
\beq{alg-twist}
\phi:\ \dds=\bplus_{\mu\in\BX }\  \dds^\mu\too(\dd(T)\ \o_{\ut}\ \dds)^\ut,
\quad \sum_\mu\ u_\mu\mto \sum_\mu\ (t^{-\mu}\o u_\mu),
\eeq
is a $\BX $-graded algebra isomorphism.

Recall that  the left (inner) action on $\BM$
 of
the algebra $(\dd(T)\ \bo_{\ut}\ \dds)^\ut$ and the right (outer) action of the algebra 
$\dd(T)\opp\o\ddp$  commute.
Therefore, Lemma \ref{Z-lem} implies that for any $u\in \dds$, one has $\phi(u)1_\BM\in \BM^{Z\g,\out}$.
We define a map $\kap_\dd$  as a composition
\beq{kap}
\kap_\dd:\ \xymatrix{
\dds\ \ar[rr]^<>(0.5){u\mto \phi(u)1_\BM}&&\
\BM^{Z\g,\out}\ \ar@{=}[rr]^<>(0.5){\eqref{zgb}}&&\
(\dd(T)\opp\ \oinn_{Z\g}\  \dd^\psi(G/\bU))^{Z\g,\out}.
}
\eeq
Explicitly,  the element $\kap_\dd(u)$ is uniquely determined
from the equation $\phi(u)1_\BM=1_\BM\kap_\dd(u)$.
Using this, for any $u,v\in \dds$, we find
\begin{align*}
1_\BM\kap_\dd(uv)&=\phi(uv)1_\BM=\phi(u)(\phi(v)1_\BM)=\phi(u)(1_\BM\kap_\dd(v))\\
&=
(\phi(u)1_\BM)\kap_\dd(v)=(1_\BM\kap_\dd(u))\kap_\dd(v)=1_\BM(\kap_\dd(u)\kap_\dd(v)).
\end{align*}
We deduce that $\kap_\dd$   is an algebra homomorphism.

A $G$-invariant volume form on $G$, resp. $\tb$, % and a  $T$-invariant volume form on $T$,
provides  an algebra isomorphism $\dd(G)\iso\dd(G)\opp$, resp. $\dd(G/U)\iso\dd(G/U)\opp$.
%, and  $\dd(T)\iso\dd(T)\opp$, to be denoted $u\mto u\opp$.
These isomorphisms
are, in fact, independent of the choices of  invariant volume forms since such a form is unique 
 up to a nonzero constant factor. Using that $\dd(\tb)=\dd(G/U)\cong(\dd(G)/\u)^\u$,
 we obtain a chain of algebra isomorphisms

\beq{opp-iso}
\dd(\tb)\ccong\dd(\tb)\opp\ccong\big((\dd(G)/\u)^\u\big)\opp\ccong
\big((\dd(G)\opp/\u)^\u\big)\opp\ccong(\u\back\dd(G))^\u\ =\ \dds.
\eeq

The following result is a more precise version of Theorem \ref{mthm}.
\begin{thm}\label{kaph-thm} The following composite map   is an algebra isomorphism:
\[
\xymatrix{
\dd(\tb)\ \ar[rr]^<>(0.5){\eqref{opp-iso}}_ <>(0.5){\cong}&& \ \dds\ \ar[r]^<>(0.5){\kap_\dd}
&\ (\dd(T)\opp\ \oinn_{Z\g}\  \dd^\psi(G/\bU))^{Z\g,\out}.
}
\]
\end{thm}

%We put $\hc_\mu=\tau_\mu\ccirc\hc$. Then, for any $u\in \dds^\mu$ and $z\in Z\g$, one has
%$u\cdot \dag z=a(\hc_\mu(z))\cdot u$.

%Given a weight $\mu\in\BX $, let  $(\u\back\dd(G)/\up)^\mu:=\{\bar D\in \u\back\dd(G)/\up\mid
%\hc_\mu(z)\bar D=\bar Dz, \ \forall  z\in Z\g\}$.
%The proof of Lemma \ref{act-lem} shows that the isomorphism of the  theorem breaks up
%into a collection of isomorphisms
%\beq{weight-kap}
%\kap_\dd_\mu:\ \dd(\tb)^\mu\ \iso\ t^\mu\ \boxtimes\ (\u\back\dd(G)/\up)^\mu,\quad
%u\mto u1_\BM,\qquad \forall \mu\in\BX .
%\eeq

%Given $\mu\in\BX $, let $\tau_\mu: \ut\to \ut$  be an algebra automorphism
%induced by the map $h\mto h+\mu(h),\ h\in\t$.
%The image of the  Harish-Chandra homomorphism equals $\ut^W$,
%the algebra of $W$-invariants under a $\rho$-shefted action
%$w: h\mto (\tau_\rho)\inv(w(\tau_\rho(h)))$.

 \subsection{Review of twisted differential operators (TDO)}\label{tors-sec} 
In this subsection we extend the
formalism of TDO developed by Beilinson-Bernstein \cite{BB}  to a slightly more general setting
of torsors of group schemes.  Such an extention will be used in the next subsection to give an intrinsic construction of the Miura bimodule
that does not depend on the choice of a pair $B, \bar B$, of opposite Borel subgroups.

Let ${N}$ be a smooth affine group scheme on
a smooth variety $X$
and $\fn=\Lie {N}$ the corresponding  Lie algebra, a locally free $\oo_X$-module.
Let $p: P\to X$ be an ${N}$-torsor.
The $\oo_P$-module  $p^*\fn=\oo_P\o_{p^\hdot\oo_X}p^\hdot\fn$
comes equipped with the natural structure of a Lie algebroid.
The Lie bracket is
given by the formula
$[f_1\o \gamma_1, f_2\o \gamma_2]=
f_1f_2\o [\gamma_1,\gamma_2]+f_1\gamma_1(f_2)\o \gamma_2-
f_2\gamma_2(f_1)\o \gamma_1$.
There is an exact sequence of locally free sheaves
\beq{ex-sec}0\to p^*\fn  \to \T_P \to p^*\T_X\to 0.
\eeq

Let $M=\op{Aut}_N(P)$ be the group  scheme of (fiberwise) automorphisms of  $P$.
The Lie algebra $\fm=\Lie M$ is a locally free  $\oo_X$-module $(p_*\oo_P\o_{\oo_X} \fn)^N$,
 of sections of  an associated bundle  $P\times_{N}\fn$
 for the adjoint representation of $N$ on $\fn$.
Let $\fk:=(p_*\T_P)^N$.
The commutator of  vector
fields gives   a  Lie bracket  on $\fk$  which is  not $\oo_X$-linear in general.
Applying $(p_*(-))^{N}$ to \eqref{ex-sec} 
we get an exact sequence of sheaves of Lie algebras 
\[
0\to {\fm} \to \fk\xrightarrow{a}\T_X\to 0.
\]

The map  $a$ makes $\fk$ a  Lie
algebroid on $X$. Similarly, the bracket 
\[[k_1\oplus f_1, k_2\oplus f_2]:=[k_1,k_2]\oplus \big(a(k_1)(f_2)-
a(k_2)(f_1)\big),\]
 gives $\fk\oplus \oo_X$ the structure of  a  Lie
algebroid on $X$ with anchor map $(k\oplus f)\mto a(k)$.
There is a canonical isomorphism
$(p_*(F_1\dd_P))^{N}\cong \fk\oplus \oo_X$, of Lie algebroids,
where $F_1\dd_P=\T_P\oplus\oo_P$ is
the sheaf of first order  differential operators.

Let  $\Om^1_{P/X}$ be  the sheaf of relative 1-forms on $P$
and $\fm^*:={\mathcal H\textit{om}}_{\oo_X}(\fm,\oo_X)$.
We have canonical isomorphisms
\beq{m=n}\fm^*\cong\fk^*/\Om^1_X\cong(p_*\Om^1_{P/X})^{N}
\cong(p_*\oo_P\o\fn^*)^N.
\eeq
Since $\Ga(X,\n^*)^N\sset \Ga(X,\,(p_*\oo_P\o\fn^*)^N)$,
any  $N$-invariant  section $\psi\in \Ga(X,\n^*)$
gives, via  \eqref{m=n},  an associated section $\psi_\fm\in\Ga(X,\fm^*)$. 

The Lie algebroid $\fk$ acts on $\fm^*$ via the Lie derivative $\textit{Lie}$.
A section $\phi$ of $\fm^*$
is said to be  $\fk$-{\em
  invariant}  if 
$\textit{Lie}_k(\phi(m))=\phi([k,m])$ for all local sections $k\in\fk,\ m\in \fm$.

Let  $\psi\in \Ga(X,\n^*)$ be an  $N$-invariant  section. 
Then,  for any  local sections $u_1,u_2$ of 
$\fn $,  resp. $\fm$, one has  $\psi([u_1,u_2])=0$, resp. $\psi_\fm([u_1,u_2])=0$.
Therefore,  the  map $p^*\fn \to F_1\dd_P,\ u\mto u-(p^*\psi)(u)$,  
resp. $\fm\to \fk\oplus\oo_X,\ u\mto u-\psi_\fm(u)$,
is a morphism of Lie algebroids.
Let $\fn_P^\psi\sset F_1\dd_{P}$, resp.
$\fm^\psi \sset \fk\oplus\oo_X=(p_*(F_1\dd_P))^{N}$, be the image
of that morphism. It is clear that $p^*(\fm^\psi)=\fn_P^\psi$.

Assume now that the $N$-invariant  section $\psi$ is such 
that the associated section $\psi_\fm\in\Ga(X,\fm^*)$ is $\fk$-invariant.
Then,
 $\fm^\psi$ is a Lie ideal of $(p_*(F_1\dd_P))^{N}$. 
It follows  that  
$\fm^\psi \cdot(p_*\dd_P)^{N}=(p_*\dd_P)^{N}\cdot\fm^\psi $  is a two-sided
ideal of $(p_*\dd_P)^{N}$.

One has a push-out diagram of Lie algebroids
\[
\xymatrix{
0\ar[r]& \fm\ar[d]^<>(0.5){\psi_\fm}\ar[r]& \fk \ar[d]\ar[r] &\T_X\ar[r]\ar@{=}[d]&0\\
0\ar[r]& \oo_X\ar[r]& (\fk\oplus\oo_X)/\fm^\psi\ar[r]&\T_X\ar[r]&0
}
\]
\noindent
The bottom row of the diagram  is a Picard algebroid on $X$.
Let $\dd_X^\psi$  be an associated TDO on $X$, see \cite[\S 2.1.4]{BB}.
There are   canonical algebra isomorphisms
\[\dd_X^\psi\ccong (p_*\dd_P)^N/(p_*\dd_P)^{N}\fm^\psi\ccong
(p_*(\dd_P/\dd_P\fn_P^\psi))^N.
\]
Here, the first  isomorphism follows from the
construction of $\dd_X^\psi$ given  in {\em loc cit}. The second  isomorphism
is obtained by applying the functor $p_*(-)^N$  to
the  isomorphism $p^*\dd_X^\psi \cong 
{\dd_P/\dd_P \fn_P^\psi}$. 
The latter   isomorphism follows from the  canonical isomorphism
$\fn_P^\psi\cong p^*\fm^\psi$. 

The sheaf $p^\hdot\dd_X^\psi $ acts naturally on 
$p^*\dd_X^\psi $ by right multiplication. This gives  $\dd_P/\dd_P \fn_P^\psi$
 the canonical  structure of a $(\dd_P, p^\hdot\dd_X^\psi )$-bimodule. Furthermore, one has
\[
\cend_{\dd_P}\bigl(\dd_P/\dd_P\fn_P^\psi\bigr)
\ =\ (p^\hdot\dd_X^\psi)^{op},\en\text{resp.}\en
\cend_{p^\hdot\dd_X^\psi}\bigl(\dd_P/\dd_P\fn_P^\psi\bigr)\ =\
\dd_P.
\]

For any
 left $\dd_X^\psi $-module $\ce$,
the sheaf
\[
p^*\ce=(\oo_P\bo_{p^\hdot\oo_X} p^\hdot\dd_X^\psi )
\o_{p^\hdot\dd_X^\psi }\ce\cong
(\dd_P/\dd_P \fn_P^\psi)\bo_{p^\hdot\dd_X^\psi }\ce.
\]
has the natural structure of a left
$\dd_P$-module.
For any  left, resp. right,  $\dd_P$-module $\cf$ and $i\geq 0$, the sheaf
$H_i(\fm^\psi, p_* \cf)$, resp. $H^i(\fm^\psi, p_* \cf)$, inherits a
left,
resp. right, action of
$\dd_X^\psi $.

\subsection{}\label{gtors}
Fix a connected linear algebraic group $\bg$ and  a $\bg$-torsor $P\to\pt$.
Let $\fl \bg=\op{Aut}_\bg(P)$ be the group of automorphisms of $P$.
Further, let $X$ be a smooth $\bg$-variety and $p: P\to X$ a smooth, surjective
$\bg$-equivariant morphism. It follows that $\bg$ acts transitively on $X$. Let
$N_x\sset \bg$ denote the stabilizer of $x\in X$ in $\bg$.
The family $N_x,\ x\in X$, forms a smooth {\em group scheme  of stabilizers}
$N$, a subgroup of the constant group scheme $\bg\times X\to X$.
The map $p: P\to X$ is a $\bg$-equivariant  $N$-torsor. 
The group scheme $M=\op{Aut}_N(P)$ of  automorphisms of this $N$-torsor,
can be identified in  a natural way with a constant subgroup scheme
of the group scheme $\fl\bg\times X\to X$, that is, we have
$M=\fl N\times X\to X$, where $\fl N$ is a subgroup of the group $\fl\bg$.
Moreover, $P$ is a $\fl\bg$-torsor and the map $P\to X$ descends to
a  $\bg$-equivariant  isomorphism
$P/\fl N\iso X$.

A choice of base point  $\pt\in P$ gives an isomorphism
$\bg\cong \fl\bg$ and a $\bg$-equivariant, resp. $\fl\bg$-equivariant,  isomorphism
$P\cong\bg$, resp. $P\cong\fl\bg$.
Furthermore, writing $x=p(\pt)$, we get a group isomorphism $\fl N\cong N_x$, and a
 $\bg$-equivariant isomorphism
$X\cong \bg/N_x$.

The goal of the rest of this subsection is to define a localized, i.e. sheaf-theoretic,
 counterpart of the Miura bimodule.
To this end, we introduce some notation.
Let   $ U_{\tb}$,
resp.    $U_\tbm$,   be the  group scheme of stabilizers for
the $G$-action on  $\tb$ and   $\tbm$.
Write $\u_\tb=\Lie U_\tb$, resp. $\u_\tbm=\Lie  U_\tbm$.
According to  Lemma \ref{closed} the natural map
 $p: \Xi\to\tb$, resp.    $p_-: \Xi\to\tbm$,
  is a $U_\tb$-torsor, resp. $U_\tbm$-torsor.
Furthermore, $\Xi$  is a $G$-torsor.

In section \ref{can-psi} we have constructed  a canonical  $G\times T$-invariant  section $\Psi\in 
\Ga(\tbm,\u^*_\tbm)$ such that the value of $\Psi$ at every point
$(\bar\b,s)\in\tbm$ is  a nondegenerate character of $\u(\bar\b)$.
Associated with $\Psi$, there is a canonically defined
sheaf $\u_\tbm^\Psi$  on $\tbm$, and the corresponding
TDO  $\dd_\tbm^\Psi$, cf. \S\ref{tors-sec}. 

Let $\dag\dd_\tb:=(\u_\tb\back p_*\dd_\Xi)^{\u_\tb}$.
This is a TDO on $\tb$ and there are canonical isomorphisms
\beq{dd1}
p^*(\dag\dd_\tb)/p^*_-\u^\Psi_\tbm\ccong 
p^*\u_\tb\back\dd_\Xi/p^*_-\u^\Psi_\tbm\ccong
p^*\u_\tb\back p^*_-\dd_\tbm^\Psi,
\eeq
 of sheaves on $\Xi$. Furthermore, it follows from section 
\ref{tors-sec} that the sheaf in \eqref{dd1} has a natural left action of the algebra
$p^\hdot(\dag\dd_\tb)$, as well as a right action of $p_-^\hdot\dd_\tbm^\Psi$. These two actions commute.

We will need
a modification of  the sheaf  defined in  \eqref{dd1} that incorporates the
maximal torus $T$.
In more detail,   let  $B_\bb$, resp.  $\bd_\tb$,
be  the  group scheme of stabilizers for the $G$-action on $\bb$, resp. 
 $T\times G$-action on $\tb$ by  $(g,t): x\mto gxt\inv$.
Let $B_\tb$ be a pull-back of $B_\bb$ via the projection $\tb\to \bb$.
Thus, we have $U_{\tb}\cong [B_\tb,B_\tb]$ and  $B_\tb/U_\tb\cong T_\tb$ is a constant group scheme.
It follows  that
the composition $\bd_\tb\into T\times G\times\tb\xrightarrow{pr_{G\times\tb}} G\times\tb$ yields an
isomorphism $\bd_\tb\iso B_\tb$.  
Let  $\b_\tb=\Lie B_\tb$, resp. $\bde=\Lie \bd_\tb$ and $\ude:=[\bde,\bde]=0\oplus\u_\tb$.
There are canonical isomorphisms $\bde/\ude\cong
\b_\tb/\u_\tb\cong \t\o \oo_\tb$.

Next, we consider
a special  case of the general setting at the begining of this subsection where  $\bg:=T\times G$ and
$P:=T\times \Xi$.
Let $\ba: T\times \Xi\to\tb$, resp.  $\ba_-: T\times \Xi\to\tb_-$,
be a map defined by the formula $\ba(t,x)= p (x)t\inv$, resp.  $\ba_-(t,x)= p _-(x)$. 
By  Lemma \ref{closed}, we have a $T\times G$-equivariant
isomorphism $\tom \cong  T\times \Xi$; furthermore, $\tom$
is a $T\times G$-torsor.
Therefore,
the map  $\ba$ is a $\bd_\tb$-torsor, resp. $\ba_-$ is
 a  $T_\tbm\times_{_\tbm} U_{\tbm}$-torsor.
The map $\ba_-$ equals the composition $\wt\Om\into \tbm\times \tb\xrightarrow{pr_1}\tbm$.
The sheaf $(\bde\back\ba_*\dd_\tom)^{\bde}$ is a TDO on
$\tb$; moreover, using a  localized version of the map
\eqref{alg-twist}  one  constructs a natural isomorphism
$\ba^\hdot\big((\bde\back\ba_*\dd_\tom)^{\bde}\big)\cong \ba^\hdot(\dag\dd_\tb)$.

Below,
we will abuse notation and given a sheaf  $\cf$ on $\Xi$,
 write $\dd_T\boxtimes \cf$ for
$\oo_{T\times\Xi}\,\o_{_{\oo_T\boxtimes \oo_\Xi}}(\dd_T\boxtimes\cf)$.
Mimicing formula \eqref{dd1}, we define
\[
\mm := \ 
\ba^*\bde\backslash\dd_{\tom}/ \ba^*_-\u_\tbm^\Psi
%\ccong \ba^*\bde\backslash(\dd_T\boxtimes p^*_-\dd_\tbm^\Psi)
\ccong\t\back(\ba^*\ude\back\dd_{T\times\Xi}/ \ba^*_-\u_\tbm^\Psi)\ccong
\t\back\big(\dd_T\boxtimes (p^*\u_\tb\back p^*_-\dd_\tbm^\Psi)\big),
\]
where  $\t$-coinvariants in the third, resp. fourth, term   are taken
with respect to the $\t$-action induced by the  natural map
$\t\to \Ga(\tom, \t\o \oo_\tom) \iso\Ga(\tom, \ba^*\bde/\ba^*\ude)$,
resp. $\t\to \dd(T)\o \Ga(\Xi, p^*\b_\tb/p^*\u_\tb),\ h\mto h\o 1-1\o h$.
It is clear that $\mm$  is  a $G\times T$-equivariant sheaf on $\wt\Om$.
According to section \ref{tors-sec},  this sheaf
comes equipped with a
natural left action of  the algebra $\ba^\hdot\big((\bde\back\ba_*\dd_\tom)^{\bde}\big)
\cong \ba^\hdot(\dag\dd_\tb)$, as well as a right action of
$\dd_T^{op}\boxtimes p^\hdot_-\dd_\tbm^\Psi$.
The left and right actions commute, so $\mm$ acquires the structure of
a $(\ba^\hdot(\dag\dd_\tb),\  \dd_T^{op}\boxtimes p^\hdot_-\dd_\tbm^\Psi)$-bimodule.

We now fix a base point $(\bar B,\bar x,B,\kap_{\bar \b,\b}(x))\in\Xi$. This gives an identification
$\Xi=G$, resp. $\tom={G\times T}$,  $\tb=G/U$, and $\tbm=G/\bU$.
The map $\ba$ reads: $\ba(t,g)=gt\inv U$.
Let $\psi\in\bu^*$ be the value of $\Psi$ at the point $(\bar B,\bar x)$.
We obtain identifications
$p^*\u_\tb=\u\o\oo_G$, resp. $p_-^*\u^\psi_\tbm=\bar\u^\psi\o\oo_G$.
Using these identifications one deduces an  isomorphism
 $\Ga\big(\Xi,\ p^*\u_\tb\back\dd_\Xi/p_-^*\bu_\tbm^\Psi)=\u\back\dd(G)/\up$,
and algebra isomorphisms 
\begin{align}
  \Ga\big(\tbm,\dd^\Psi_\tbm)\  &\cong\ \dd^\psi(G/\bU),\label{gaga}\\
  \Ga(\tom, \ba^\hdot(\dag\dd_\tb))\ &=\  
                                       \Ga(\tb, \dag\dd_\tb)\ =\ \Ga(\tb, (\u_\tb\back\dd_\tb)^{\u_\tb})\
=\ (\u\back\dd(G))^\u\   =\ \dag\dd. \nonumber
\end{align}

Thus,  the space $\Ga(\tom, \mm)$
acquires the structure of a  $\big(\dag\dd,\,\dd(T)^{op}\o \dd^\psi(G/\bU)\big)$-bimodule.
We compute
\begin{align}
\Ga(\tom, \mm)&\cong\Ga\big(\tom, \ \t\back(\dd_T\boxtimes (p^*\u_\tb\back\dd_\Xi/p_-^*\bu_\tbm^\Psi))\big)
\nonumber\\
&=
\t\back\big(\dd(T)\o \Ga(\Xi,\ p^*\u_\tb\back\dd_\Xi/p_-^*\bu_\tbm^\Psi)\big)\nonumber\\
&=\t\back\big(\dd(T)\o (\u\back\dd(G)/\up)\big)=
\dd(T)\o_\ut (\u\back\dd(G)/\up)=\BM.\label{mcan}
%\quad\text{resp.}\quad\BM=\dd(T)\ \o_{\ut}\  (\u\back\dd(G)/\up).
\end{align}
We conclude that the sheaf $\mm$ may be viewed as 
 a localized version of  the Miura bimodule $\BM$.

The definition of  $\mm$ involves neither the choice of a pair  $\b,\bar\b$ of opposite Borels,
nor the choice of a nondegenerate character $\psi\in\bu^*$.
Thus, the object $\Ga(\tom, \mm)$ in the LHS  of  \eqref{mcan}
provides a  completely canonical construction of
the Miura bimodule. Similarly, formulas \eqref{gaga} provide
canonical constructions of the algebras
$\dag\dd$ and $\dd^\psi(G/\bU)$. With these constructions,  the isomorphism
of Theorem \ref{kaph-thm}, hence of our main Theorem \ref{mthm}, becomes a canonical isomorphism of canonically defined algebras.

\section{Proofs of main results}
\label{7}
\subsection{Reformulation in terms of isotypic components}\label{reform}
Throughout this section,  $V$ stands for  an irreducible finite dimensional $G$-representation.
We make $V\o\ug$  an $\ug$-bimodule by letting the left, resp. right, action of $g\in \g$ 
be defined by $g(v\o u)=v\o(gu)$, resp. $(v\o u)g=-(gv)\o u+v\o (ug)$. 
Let $\ad g:\ v\o u\mto g(v\o u)-(v\o u)g=(gv)\o u+ v\o (gu-ug)$, be the corresponding
 adjoint $\g$-action. This action is locally finite.

We fix a pair $\b,\bar\b$ of opposite Borel subalgebras and identify $\t$ with $\b\cap\bar\b$.
A right  $\ug$-module $\ver=\u\back\ug=\ug/\u(\ug)$ is called the universal  Verma module. 
It is clear from definitions that one has
isomorphisms 
\begin{align}
  \u\back(V \o \ug)/\up&\ccong (V \o (\u\back\ug))/\up\,=\, (V \o \ver)/\up,\nonumber\\
  (\u\back(V\o \ug))^\u&\ccong (V\o \ver)^\u.\label{coinv}
\end{align}
where $(-)^\u$ stands for $\u$-invariants of the right, equivalently, adjoint action.
The left action of the algebra $\uub\sset\ug$ on $ V\o \ug$ descends to
a $\ut$-action on $V\o\ver$. The  isomorphism(s) in the first, resp. second, line of \eqref{coinv} respect
the left action of the algebra 
$\ut$ and  the right action of the algebra $(\ug/\up)^\up$, resp.  $(\u\back\ug)^\u$.
Using the map  $Z\g\to (\ug/\up)^\up$, resp. $\ut\to (\u\back\ug)^\u$, cf. \eqref{hcdef}, one may view $(V \o \ver)/\up$ as a $(\ut, Z\g)$-bimodule, resp.
$(V\o \ver)^\u$ as a $(\ut,\ut)$-bimodule.
There is also an  adjoint $\t$-action on  $(V \o \ver)^\u$
induced by the adjoint action of $\g$  on $V \o \ug$.
The $\ad\t$-action is semisimple, so
one has a  weight space decomposition
\beq{weightE}(V \o \ver)^\u=\oplus_{\mu\in
\BX }\ (V \o \ver)^{\u,\mu},
\eeq
where each  $\ad\t$-weight space $(V \o \ver)^{\u,\mu}$ is a
$(\ut,\ut)$-sub-bimodule of $(V \o \ver)^\u$.

Fix $\mu\in\BX $ and let $\tau_\mu: \ut\to\ut$ be an algebra
automorphism defined on generators $t\in \t$ by the assignment $t\mto t-\mu(t)$.
 Let $J_\mu$ be  an ideal of the algebra $\ut\o Z\g$ generated by the elements
$\tau_\mu(\hc(z))\o 1-1\o z$,\ $z\in Z\g$.
We claim that one has an inclusion
\beq{vhc}
(V \o \ver)^{\u,\mu}\sset ((V \o \ver)^\u)^{J_\mu},\qquad\forall\mu\in\BX ,
\eeq
where for any  $(\ut, Z\g)$-bimodule $E$, we write
$E^{J_\mu}:=\{x\in E\mid \tau_\mu(\hc(z))x=xz,\ \forall z\in Z\g\}$.

To prove \eqref{vhc}, let $x\in (V \o \ver)^{\u,\mu}$.
For  $t\in \t$, we compute
$xt=tx -\ad t(x) =tx -\mu(t)\cdot  x =
\tau_\mu(t) x$.
It follows that $xa=\tau_\mu(a)x$  holds for all $a\in \ut$.
We deduce that for any $z\in Z\g$,  one has
$xz=x\hc(z)=\tau_\mu(\hc(z)) x$,
where the first equality holds by Lemma \ref{trick}(i).
This proves  \eqref{vhc}.

Recall that we view  $\ug$  as the algebra of left invariant differential operators on $G$.
Also view $\C[G]$, resp. $\dd(G)$, as a  $G\times G$-module
 where 
the first, resp. second, copy of $G$ acts by left, resp. right, translations.
Thus, we  have $\C[G]=\bplus\ V^*\o V$, resp.
  $\dd(G)=\C[G]\o \ug=\bplus\ V^*\o (V\o \ug)$, where
the sum ranges over the set of isomorphism classes of irreducible $G$-representations.
Hence, using \eqref{coinv} we obtain a decomposition 
\beq{vhh}\u\back\dd(G)/\up\ =\ \bplus\ V^*\o (V \o \ver)/\up,
\quad\text{resp.}\quad
(\u\back\dd(G))^\u\ =\ \bplus\ V^*\o (V \o \ver)^\u,
\eeq
into isotypic components with respect   to the $G$-action by left translations.
From  \eqref{miu1}  and the first isomorphism in \eqref{vhh}, we get
\[
\BM=\dd(T)\, \o_{_\ut}\, (\u\back\dd(G)/\up) \ccong
\C[T]\, \o\, (\u\back\dd(G)/\up) \
=\ \bplus\ V^*\o \big(\C[T]\, \o\, (V \o \ver)/\up\big).
\]
Further, write $\C[T]=\bplus_{\mu\in\BX }\ \C t^{-\mu}$. We find
\begin{align*}(\C[T]\, \o\, (V \o \ver)/\up)^{Z\g,\out}\ &=\ 
\bplus_\mu\ \big(\C t^{-\mu}\o (V \o \ver)/\up\big)^{Z\g,\out}\\
&=\ 
\bplus_\mu\ \C t^{-\mu}\o ((V \o \ver)/\up)^{J_\mu}.
\end{align*}
Thus, we obtain a decomposition
\beq{st1}\BM^{Z\g,\out}\ =\ 
\bplus_{_{V,\mu}}\
V^*\o \C t^{-\mu}\o \big((V \o \ver)/\up\big)^{J_\mu}.
\eeq

Similarly,  from the second  isomorphism in \eqref{vhh}, we deduce
\beq{st2}\dag\dd =(\u\back\dd(G))^\u\ =\ \oplus_\mu\ \dag\dd^\mu\
=\ 
\bplus_{_{V,\mu}}\
V^*\o (V \o \ver)^{\u,\mu}.
\eeq

The map $\kap_\dd$, in \eqref{kap}, commutes with the $G$-action by
left translations and it also respects the $\mu$-decompositions.
Therefore, we conclude
that  this map induces, for each $(V,\mu)$,  a map $\kap_{V,\mu}:\ (V \o \ver)^{\u,\mu}
\to \big((V \o \ver)/\up\big)^{J_\mu}$,  between the corresponding
$(V,\mu)$-components of decompositions  \eqref{st2} and  \eqref{st1},
respectively.
It is easy to check that the map  $\kap_{V,\mu}$ equals the composition of  the following natural 
inclusion and projection:
\beq{ver-map}
\kap_{V,\mu}:\ (V\o\ver)^{\u,\mu}\  \xrightarrow{\eqref{vhc}}\ (V\o\ver)^{J_\mu}\too
\big((V\o\ver)/\up\big)^{J_\mu}.
\eeq

The  $(\ut,Z\g)$-bimodule structure on $V\o\ver$ induces
a  $(\ut,Z\g)$-bimodule structure on each of the three
objects in \eqref{ver-map}.  It is clear from the above discussion that Theorem \ref{kaph-thm}
follows from the theorem below, to be proved in \S\ref{vpf}.

\begin{thm}\label{vthm}
For  any  irreducible $G$-module
$V$ and $\mu\in
{\mathsf Q}$, the  map
\[\kap_{V,\mu}:\ (V\o\ver)^{\u,\mu} \too \big((V\o\ver)/\up\big)^{J_\mu},
\] 
is an isomorphism of $(\ut,Z\g)$-bimodules.
\end{thm}

\subsection{}\label{class-act} In this subsection we discuss a Poisson counterpart of Theorem \ref{vthm}.

Let  $\eh:=\bbe+\g_\bbh$.
It is known, e.g. by   Proposition \ref{munuprop}(ii)-(iv),
that this affine linear space  is contained in $\g_r$.
Let  $\fZ|_\eh$ be  the restriction of the group scheme $\fZ$ to
$\eh$ and $\fz_\eh$  the Lie algebra of  $\fZ|_\eh$.
We  put  $\g_\bbh[\eh]:=\g_\bbh\o\C[\eh]$, resp. 
$\fz[\eh]:=\Ga(\eh, \fz_\eh)$.
Thus, an element of  $\g_\bbh[\eh]$,  resp.  $\fz[\eh]$,  is a polynomial map $\eh\to \g_\bbh$,
  resp. a  polynomial map $\xi: \eh\to \g$ such that
$\xi(x)\in\g_x$ for all $x\in \eh$. 
By Lemma \ref{bk-lem}, for any $x\in\eh$ we have $\g_x\sset \b_\bbe$. Therefore,
there is a well-defined morphism $\vkap_\eh: \fz_\eh\to \g_\bbh\times \eh$ and the corresponding
map $\vkap_\eh:\fz[\eh]\to \g_\bbh[\eh]$ of global sections, such that
$(\vkap_\eh(\xi))(x)=\xi(x)\mod\u_\bbe$ for all $x\in\eh$.
Here and below, we use  identifications $\g_\bbh=\b_\bbe/\u_\bbe=\t$.

The map $s: \eh\to \tgr, \, x\mto (\b_\bbe, x)$, provides a section of the map
$\pi: \pi\inv(\eh) \to \eh,\, (\b,x)\mto x$.
Hence, there is a canonical isomorphism
$\fZ|_\eh \cong s^*(\pi^*\fZ)$.
The morphism $\vkap_\eh$ defined above is induced, via the isomorphism,
by the morphism $\vkap: \pi^*\fZ\to T_\tgr$, see Lemma \ref{imvkap}.

Let $T$, resp.   $\fZ_\eh$, act on $T\times \eh$ by
$t_1: (t,r)\mto (t_1t,r)$, resp. $z: (t,r)\mto (\vkap_\eh(z)t,r)$.
By  \eqref{Y}, \eqref{zz2} and \eqref{TZ}, one has $G\times \fZ_\eh$-equivariant
isomorphisms
\[\T^*T\times_\fc\oox\ccong T\times Z\ccong T\times (G\times \eh) \ccong
(T\times \eh )\times G.
\]
Here $G$
acts  in the LHS, resp. RHS, through its action on $\oox$, resp. action
on $G$ by left translations. The group scheme $\fZ_\eh$ acts on
$\T^*T\times_\fc\oox$ diagonally and acts on $(T\times \eh )\times G$
through its action on $T\times \eh$.
We deduce an isomorphism
\[\C[\T^*T\times_\fc \oox]\ccong\C[(T\times \eh )\times G]=\C[G]\o \C[T\times \eh ],\]
of $G$-representations and also of  $\fz[\eh]$-modules.
One has
 the   decomposition $\C[G]=\bplus\ V^*\o V$ 
 and the weight decomposition $\C[T]=\oplus_\mu\, \C t^\mu$.
 Put  $V[\eh]=V\o\C[\eh]$.
Separating $G$-isotypic components in the above isomorphism, for every $V$, we obtain
\beq{tzg}
(V\o \C[\T^*T\times_\fc \oox])^G \ccong V\o\C[T]\o\C[\eh]
\ccong\big(\bplus_{\mux}\ \C t^{\mu}\big)\o V[\eh ].
\eeq

Write  $\xi\vv$ for the natural pointwise action of an element $\xi\in \fz[\eh]$ on a polynomial map
$\vv:\eh\to V$ given by 
$(\xi\vv)(x):=\xi(x)(\vv(x))$. Then, the $\fz[\eh]$-action  on the LHS of \eqref{tzg}
goes, via the above isomorphism, to a
$\fz[\eh]$-action on the RHS given by the formula
\[
\xi(t^\mu\o \vv)= t^\mu \o \big((\mu\ccirc\vkap_{\eh})(\xi)\cdot \vv +\xi\vv\big),
\quad\forall \mu\in {\mathsf{Q}},\,\vv\in V[\eh],
\]
where $\mu\in\BX$ is viewed as
 a linear function on $\g_\bbh=\t$, so
$\mu\ccirc\vkap_{\eh}$ is the composition
$\fz[\eh]\xrightarrow{\vkap_\eh}\g_\bbh[\eh]\xrightarrow{\mu} \C[\eh]$.

Let $z\in \C[\fc]$. 
We may  view the differential $d(\vth^*z)$ of
the function $\vth^*z\in\C[\g]^G$,  as a polynomial map $\g\to \g$.
The map being $G$-equivariant, it follows 
that $d(\vth^*z)(x)\in\g_x$, for all $x\in \g$. Therefore, restricting this map to $\eh$,  one gets
an element of $\fz[\eh]$ to be denoted $\zf$. In terms of
the isomorphism of Lemma \ref{fz-lem}, one can identify
 $\zf$ with  the image of the section $dz$ of $\T^*\fc$ under the
composition $\Ga(\fc, \T^*\fc)\iso\Ga(\fc,\fz_\fc)\to \Ga(\t, \vth_\eh^*\fz_\fc)\cong\fz[\eh]$,
where $\vth_\eh$ denotes the composite $\eh\into\g\to\g/\!/G=\fc$.

Let ${\tilde z}$ be a pull-back of $z\in \C[\fc]$ via the
composite map $\T^*T\times_\fc\oox\to\fc$, cf. \eqref{star1}.
The Poisson bracket with  ${\tilde z}$
gives an operator $\{{\tilde z},-\}$  on $\C[\T^*T\times_\fc \oox]$.
The function ${\tilde z}$ being $G$-invariant,
the  operator  $\{{\tilde z},-\}$ commutes with the $G$-action,
hence restricts to an operator on each isotypic component
${(V\o \C[\T^*T\times_\fc \oox])^G}$. Equation \eqref{df} implies that
the latter operator is given, in terms of decomposition \eqref{tzg}, by the formula
\[
\{{\tilde z},\,  t^\mu\o \vv\}= \zf(t^\mu\o \vv)=t^\mu \o \big((\mu\ccirc\vkap_{\eh})(\zf)\cdot \vv +\zf\vv\big).
\]
We define
\beq{f-act}
V[\eh]^{\{\mu\}}:=
\big\{\vv\in V[\eh]\mid \xi(\vv)=\mu\ccirc\vkap_{\eh}(\xi)\cdot \vv,\ \forall \xi\in \fz[\eh]\big\}.
\eeq

From \eqref{df} and \eqref{tzg}-\eqref{f-act}, using that the Poisson action of
$\C[\fc]$ commutes with the $G$-action,
we find that  the Poisson centralizer of $\C[\fc]$ in the LHS of \eqref{tzg}
has the following decomposition
\begin{align}
\big((V\o \C[\T^*T\times_\fc \oox])^G\big)^{\C[\fc]} &\cong
\big(V\o \C[\T^*T\times_\fc \oox]^{\fz[\eh]}\big)^G \label{TZG}\\
&\cong\big(V\o\C[T]\o\C[\eh]\big)^{\fz[\eh]}
\cong 
\bplus_{\mux}\ \C t^{-\mu}\o V[\eh]^{\{\mu\}}.\nonumber
\end{align}

We are going to use Theorem \ref{thm*} and compare
\eqref{TZG} with an analogous decomposition for $\C[\T^*\tb]$.
To this end,
let $B=B_\bbe$, resp. $U=U_\bbe$, and  identify $\tb=G/U$.
The group $B$ acts on $\b$ via the adjoint action and it also acts on
every  $G$-representation $V$. We let $B$ act  on
$V[\b]:=V\o \C[\b]$ diagonally. Let $V[\b]^{U,\mu}\sset V[\b]$ denote the $\mu$-weight
space of the residual action of $T=B/U$ on $V[\b]^U$.
The group $G$, resp. $ T$, acts on $G/U$ by left, resp. right,  translations.
Using   $G\times T$-equivariant  isomorphisms
 $\T^*(G/U)\cong G\times_U\u^\perp\cong G\times_U\b$ and
 the Peter-Weyl  decomposition we obtain, for every $V$, an isomorphism
\beq{tbg}
\big(V\o \C[\T^*\tb]\big)^G\ccong \big(V\o \C[\b]\big)^U\ =\ V[\b]^U
\ =\
\bplus_{\mux} (V[\b])^{U,\mu}.
\eeq

The isomorphism 
$\C[\YY]\cong 
\C[\T ^* T\times_{\fc}\xx]^{\C[\fc]}$ of  Theorem \ref{thm*} is
 $G$-equivariant. Hence it  induces  for every $V$, 
an isomorphism, cf.   \eqref{TZG} and ~\eqref{tbg},
\[\kap_{V}:\  \big(V\o \C[\T^*\tb]\big)^G=V[\b]^U\  \xrightarrow{\ \cong\ }
\  \big(V\o \C[\T^*T\times_\fc \oox]^{\C[\fc]}\big)^G.
\]

View $f\in V[\b]^{U,\mu}$ as  a $V$-valued polynomial function  on $\b$
and let  $f|_{\eh}\in V[\eh]$ be its  restriction
 to $\eh \subset\b$. Then, from  commutativity of the diagram of
Proposition \ref{imb} one easily derives that $f|_{\eh }\in
V[\eh]^{\{\mu\}}$ and, moreover, one has
$\kap_{V}(f)=t^{-\mu}\o f|_{\eh }$. It follows that the map $\kap_{V}$  respects the $\mu$-decompositions
in \eqref{TZG} and  \eqref{tbg}, respectively. Thus, this map
breaks up into a direct sum of maps
\[\kap_{V,\mu}:\  V[\b]^{U,\mu}\too
V[\eh ]^{\{\mu\}},\quad f\mto  f|_{\eh },\qquad\mu\in\BX .
\]

The map $\kap_V$ being an isomorphism, we obtain
the following result that may be viewed as a Poisson counterpart of
Theorem \ref{kaph-thm}, cf.  isomorphisms \eqref{class} below.
\begin{prop}\label{vthm-cl} For every irreducible $G$-module $V$ and $\mu\in\BX $, the map
$\kap_{V,\mu}$ is an isomorphism.\qed
\end{prop}

\subsection{A deformation construction}\label{vpf-1}

We are going to deduce Theorem \ref{vthm} from Proposition \ref{vthm-cl}
by a deformation argument. To this end, we need to recall
a well-known construction  of deformation 
to the normal bundle.

Let $X$ be a smooth connected affine algebraic variety equipped with a
smooth (i.e. flat with smooth fibers) morphism
$X\to \C$.
Let $I\subset \C[{X}\times_\C{X}]$
 be the ideal of the 
 diagonal ${X_\De}\subset{X}\times_\C{X}$.
  Write $\hb$ for the coordinate on $\C$ and define a
  $\chb$-subalgebra  $A\subset\C[\hb,\hb\inv]\o_\chb\C[{X}\times_\C{X}]$
  by  $A:=\sum_{n\geq0}\ \hb^{-n} I^n$.
  By definition, one has
  an algebra imbedding
  $\C[{X}]\o_\chb\C[{X}]=\C[{X}\times_\C{X}]\into A$.
  The algebra $A$, viewed as a subalgebra of
  $\C[\hb,\hb\inv]\o_\chb\C[{X}]\o_\chb\C[{X}]$,
is generated
by  its subalgebra $\C[{X}]\o_\chb\C[{X}]$ and the elements
$\frac{z\o1-1\o z}{\hb},\ z\in \C[X]$.
If  $z_1,\ldots,z_r$,
 is a set of generators of
 $\C[{X}]$, then the  $\chb$-algebra $A$ is
generated  by the elements
$\frac{z_i\o 1- 1\o z_i}{\hb}$, and  either $z_i\o1$, or  $1\o z_i$,  for $i=1,\ldots,r$.
Thus,   $A$   is a  finitely generated algebra
and we put $\nf_X:=\Spec(A)$.
The algebra imbedding 
$\C[{X}\times_\C{X}]\into A=\C[\nf_X]$ induces  a canonical  
 morphism
$p : \nf_X\to {X}\times_\C{X}$, which is an isomorphism over $\C\sminus \{0\}$.
It is known,
cf. eg. \cite[5.11]{Gi3},  that  the scheme $\nf_X$ is smooth and the composite
$\nf_X\to {X}\times_\C{X}\to\C$ is a smooth morphism.

The diagonal imbedding $\De: X\to  X_\De\subset X\times_\C  X$ has a canonical
lift to  a closed imbedding $\eps :  X\into \nf_X$ such that $p \ccirc\eps=\De$.
The defining ideal of the image of $\eps $ is an ideal $\J\subset \nf_X$ generated
by the elements $\frac{z\o1-1\o z}{\hb},\ z\in \C[X]$, where we identify
elements of $\C[{X}]\o_\chb\C[{X}]$ with their images under $p ^*$.
By construction, one has a $\chb$-algerbra isomorphism  $\C[\nf_X]/\J\cong \C[X]$.

Let $X_0$, resp. $\nf_{X,0}$, be the fiber of $X\to\C$, resp.  $\nf_X\to \C$,
over $0\in \C$. 
The variety
$\nf_{X,0}$ is known to be canonically isomorphic to   $\T(X_0)$,
the total space of the tangent bundle
on $X_0$.  The variety $p\inv_\nf(X_\De)$
is a union of two irreducible components,
$\eps (X)$ and $\T(X_0)$. These components meet transversely at the zero section
$X_0\subset \T(X_0)$ of the tangent bundle.
It follows that the zero section, viewed as a subvariety of $\nf_X$,
is defined  by the
ideal  $\J+ \hb\,\C[\nf_X]=\J+ \hb\,\C[X\times_{\BA^1} X]$.

Given a $\chb$-module $M$, we write $M\hbo:=M/\hb M$.
Let $M$ be
a $\C[X\times_\C X]$-module
such that the $\C[X\times_\C  X]\hbo$-module $M\hbo$  is annihilated
by  the ideal of the diagonal $X_0\subset X_0\times X_0$.
Thus,  for any $z\in \C[X]$ and $m\in M$, one has
$(z\o1-1\o z)m=\hb m'$, for some $m'\in \C[X\times_\C  X]$.
If $M$ is  $\hb$-torsion free, then the element $m'$  is  uniquely determined
by $z$ and $m$.
We conclude that for any  $\hb$-torsion free
$\C[X\times_\C X]$-module $M$ such that $M\hbo$  is  annihilated
by  the ideal of the diagonal of $X_0\times X_0$,
the $\C[X\times_\C X]$-action on $M$ has a canonical
extention to a $\C[\nf_X]$-action.
Furthermore, the latter action induces a $\C[\T(X_0)]$-action on $M\hbo$.
\smallskip

Below we abuse terminology and 
  given $\chb$-algebras $A_1$ and $A_2$, refer to
  left $A_1\o_{\chb}A_2^{op}$-modules  as  `$(A_1,A_2)$-bimodules'.
Recall that a bimodule $M$ over a commutative algebra
$\mathcal{A}$ is said to be symmetric
if $am=ma$ for all $a\in \mathcal{A}, m\in M$.

Let  $\uh\fk$ denote the asymptotic enveloping
algebra of a Lie algebra $\fk$.
By definition,  $\uh\fk$  is a  $\chb$-algebra generated by 
the vector space $\fk$, with
relations $xy-yx = \hbar [x,y]$ for $x,y \in \fk$. We have
$(\uh\fk)\hbo=\sym\fk$.  
For any $x\in\fk$, the map $\fk\to\fk,\, y\mto [x,y]$, extends uniquely
to a $\chb$-linear derivation $\ad x$ of the algebra $\uh\fk$.
This gives  an `adjoint' action of $\fk$ on $\uh\fk$.
The $\ad\fk$-action is locally finite.

The assignment $\hb \mapsto 1 \o \hb$, $x\mto
1 \otimes x -x \otimes \hb$,
has a unique extension to an algebra homomorphism
$\uh\fk\to{\mathcal U}\fk^{op}\o\uh\fk$. 
Via this homomorphism, for any left $\fk$-module $V$ and a
right $\uh\fk$-module $E$,
 the vector space $V\o E$ acquires
the structure of a right $\uh\fk$-module. 
We define a left $\uh\fk$-action on $V\o \uh\fk$  by $x(v\o u)=v\o xu$.
This makes  $V\o\uh\fk$ an $(\uh\fk,\uh\fk)$-bimodule.
There is also  an `adjoint' action of $\fk$ on 
$V\o\uh\fk$  defined by $\ad x(v\o u):=- x(v) \o u + v\o \ad x(u)$.
The above actions  are related by the formula
$x(v\o u)-(v\o u)x=\hb\cdot\ad x(v\o u)$.
It follows that for any $a\in\uh\g$ and $m\in V\o\uh\fk$, one has
$am-ma\in \hb\cdot(V\o\uh\fk)$, hence
$(V\o\uh\fk)\hbo$ is a symmetric $(\sym\g,\sym\g)$-bimodule.
Furthermore, the following formula
\beq{poissm} \{u\hbo,\,m\hbo\}:= \big(\mbox{$\frac{1}{\hb}$}(um-mu)\big)\hbo,
  \qquad
  u\in\uh\g,\, m\in V\o\uh\fk,
\eeq
gives $(V\o\uh\fk)\hbo$ the structure of a Poisson module
over the  algebra $\sym\g$ equipped
with  the  Kirillov-Kostant Poisson structure.

Various  constructions of previous subsections have asymptotic analogues.
In particular, one has the universal asymptotic  Verma module 
$\ver_\hb:=\u\back\uh\g$.
Let $V$ be a finite dimensional $G$-representation.
 By the above, the vector space $V\o \ver_\hb$ has
the structure of a right $\uh\g$-module. 
The left action of the subalgebra $\uh\b\subset\uh\g$
on  $V\o \uh\g$ descends to  a left $\uh\t$-action on $V\o \ver_\hb$.
 This makes  $V\o\ver_\hb$ an $(\uh\t,\uh\g)$-bimodule. 
The adjoint  $\b$-action  on 
$V\o\uh\g$ descends to  an $\ad\b$-action on
$V\o\verh$. 
We write
$(V\o \verh)^\u$ for the space of $\u$-invariants of the right, equivalently adjoint, action.
Similarly to the discussion between formulas \eqref{coinv} and \eqref{weightE},
the space  $(V\o \verh)^\u$ has
the natural structure of a  $(\uh\t,\uh\t)$-bimodule.
There is also an adjoint $\t$-action on $(V\o \verh)^\u$ which
 is related to the bimodule structure by the formula
 $t(v\o m)-(v\o m)t=\hb\ad t$, for all
$t\in\t,\,v\o m\in (V\o \verh)^\u$.
The $\ad\t$-action is  semisimple. Thus, one has a weight decomposition
$(V\o \verh)^\u=\bplus_{\mu\in\BX }\ (V\o \verh)^{\u,\mu}$
as a direct sum of $(\uh\t,\uh\t)$-subbimodules, cf. \eqref{weightE}.

Observe that  $(V\o\ver_\hb)\hbo$ is symmetric
as a $(\sym\b,\sym\b)$-bimodule since $(V\o \uh\g)\hbo$ is  symmetric as a
$(\sym\g,\sym\g)$-bimodule,
cf. discussion above formula \eqref{poissm}.
Furthermore, the action of $\sym\b$ factors through an action of $\sym\t$.
The identification  $\g/\u=\b^*$, see \eqref{killing},
yields an identification $\sym(\g/\u)=\C[\b]$. We deduce
\beq{class1}(V\o\ver_\hb)\hbo\ccong V\o\sym(\g/\u)\ =\ V\o \C[\b].
\eeq
The projection $\sym\b\to\sym\t$ corresponds, via the identifications
$\C[\t]=\C[\t^*]$,
 to a  pull-back map $pr_\t^*: \C[\t]\to  \C[\b]^U$ 
 induced by the projection $pr_\t: \b\to \b/\u=\t$.
 Therefore, the action of $\sym\t$ on $(V\o\ver_\hb)\hbo$
 corresponds, via  isomorphism \eqref{class1}, to a $\C[\t]$-action 
on $V\o \C[\b]$ given
by $f: (v\o g)\mto v\o ((\pr_\t^*f)\cdot g)$.
We  remark that the map $pr_\t^*: \C[\t]\to  \C[\b]^U$
is an algebra isomorphism.
This follows from the fact that the fiber of the map  $pr_\t$
over any regular element of $\t$ is a single $U$-orbit,
cf.  \cite[Lemma 1.1.14]{CG}, hence
every  $U$-invariant polynomial on $\b$ must be
a pull-back of  a polynomial on $\t$.

\begin{lem}\label{vlem} \vi
  The left  $\uh\t$-action makes
 $(V\o\verh)^{\u}$
 a  free  $\uh\t$-module of finite rank.

 \vii  The map 
$(V\o\verh)^{\u}\hbo\to (V\o \C[\b])^{\u}$ induced by \eqref{class1}
is an isomorphism of $\C[\t]$-modules. 

\viii The following maps induced by the  imbedding $\t=\b/\u\into \g/\u$ 
 are algebra isomorphisms:
 \[\ut\  \mto\   (\u\back\ug)^\u\  =\  \End_{\ug}\ver,\en\  \text{resp.}\en \
   \uh\t\  \mto\   (\u\back\uh\g)^\u\  =\  \End_{\uh\g}\verh,\]
 where the equalities come from \eqref{qhr}.
\end{lem}
\begin{proof} Parts (i) and (ii)  of the lemma are
  essentially contained in \cite{GR}.
  In more detail, for any Lie algebra $\fk$
  we view (as we may) the algebra $\uh\fk$ as the Rees algebra
  associated with the PBW filtration on $\UU\fk$.
  This makes $\uh\fk$ a non-negatively  graded algebra with
  homogeneous components $(\uh\fk)^{(i)}$, where
  $(\uh\fk)^{(0)}=\C$ and $(\uh\fk)^{(1)}=\fk\oplus\C\hb$.
  This applies, in particular,  to $\fk=0$, resp. $\fk=\g$ and $\fk=\t$.
  Let
  $V\o\verh=\oplus_{i\geq0}\,(V\o\verh)^{(i)}$  be a grading defined by
  the formula 
  $(V\o\verh)^{(i)}=(V\o 1)(\uh\g)^{(i)}$.
  One also has a grading $\C[\t]=\oplus_{i\geq0}\,\C^i[\t]$, resp.
  $V\o \C[\b]=\oplus_{i\geq0}\,V\o\C^i[\b]$,
  where $\C^i[-]$  denotes the space of degree $i$ homogeneous polynomials.
  Then,   $(V\o\verh)^\u$, resp. $(V\o \C[\b])^\u$,  is a graded
  $\uh\t$-submodule of $V\o\verh$,
  resp.   $\C[\t]$-submodule of $V\o\C[\b]$.
    Now, Lemma ~3.5.2 of  \cite{GR}
  says that the map in the statement of part (ii)
   is an isomorphism
  of graded $\C[\t]$-modules, proving (ii).

To prove (i), let $\dd_\hb(G)$ be
   the Rees algebra associated with the filtration on $\dd(G)$ by order of the
  diffirential operator. Thus, $\dd_\hb(G)$ is a graded $\chb$-algebra
  and we put $\dag\dd_\hb:=(\u\back\dd_\hb(G))^\u$,
    an asymptotic analogue of the
    algebra $\dag\dd=(\u\back\dd(G))^\u$.
    There is a graded $\chb$-algebra imbedding $\uh\t\into \dag\dd_\hb$, resp.
    and  an  isotypic decomposition 
    $\dag\dd_\hb= \oplus_V\, V^* \o (V \o \verh)^\u$
    as direct sum of graded   $\uh\t$-modules,
    an asymptotic analogue of the imbedding $a$,   cf. 
Section \ref{mi-sec},  resp.  
decomposition \eqref{vhh}.
       It was shown in
  \cite[Proposition 3.2]{GR}  that  $\dag\dd_\hb$ is  free
  as a graded $\uh\t$-module.
      It follows that  each  isotypic component
    $(V\o\verh)^{\u}$ is  projective  as a left $\uh\t$-module.
    This projective module is necessarily free since it is
    a nonnegatively graded module,
    proving ~(i).

Statement  (iii) is well known but we provide a proof for completeness.
To this end,
  we apply (the proofs of) parts (i) and (ii) 
in the special case where $V$ is the trivial representation.
We deduce that $(\u\back\uh\g)^\u$ is a
free graded $\uh\t$-module and
the map $((\u\back\uh\g)^\u)\hbo\to\C[\b]^U$ is an isomorphism
of graded $\C[\t]$-modules.
Therefore, to prove that the second map in (iii) is an isomorphism it
suffices,
by the graded Nakayama lemma, to prove that
the composition
$\C[\t]=\uh\t\hbo\to ((\u\back\uh\g)^\u)\hbo\to\C[\b]^U$
 is an isomorphism. This composition is
the pull-back map $pr_\t^*: \C[\t]\to  \C[\b]^U$,
which is an isomorphism by the remark before Lemma \ref{vlem}.
Finally, we observe that the first isomorphism of part (iii)
follows from the second by  the specialization at $\hb=1$.
\end{proof}

Let $\zh\g$ denote the center of $\uh\g$.
There is an asymptotic counterpart of the Harish-Chandra homomorphism
defined as a composition
$\hc_\hb: \zh\g\to \End_{\uh\g}\verh =(\u\back\uh\g)^\u\iso \uh\t$,
where the first map comes from the $\zh\g$-action on $\verh$ and the isomorphism on the
right
is an inverse of the  isomorphism of Lemma \ref{vlem}(iii).
One shows, mimicing standard arguments,
that the map $\hc_\hb$ yields an algebra isomorphism
$\hc_\hb: \zh\g\iso (\uh\t)^W=\C[\t^*\times\BA^1]^W$.
Here, $W$-invariants  are taken with respect to the
`dot-action' of $W$ defined on
generators $t\in\t$ by
$w: t\mto w(t)-\hb\rho(w(t))+\hb\rho(t)$,
where $\rho$ is the half-sum of positive roots.
The isomorphism $\hc_\hb$ specializes at $\hb=0$ to the Chevalley isomorphism
$(\sym\g)^G\cong(\sym\t)^W$.

Let $f_i,\, i\in  I$, be   simple root vectors for the Lie algebra $\bu:=\u_\bbf$
and $\psi\in\bu^*$ a nondegenerate character.
Let $\uph$ be a $\chb$-submodule of $\uh\g$ generated  by the vector space
$[\bu,\bu]\sset\g$ and the elements $f_i-\psi(f_i),\, i\in  I$.
For any $x,y\in\uph$, we have $xy-yx\in\uph$.
The adjoint action of $\bu$ on $V\o\uh\g$ descends
to a locally finite $\bu$-action on $(V\o\uh\g)/\uph$.
We write $((V\o\uh\g)/\uph)^\bu$ for the corresponding space of 
$\bu$-invariants.

We put $\up:=\{x-\psi(x),\ x\in\bu\}$. We
view this vector space as a subspace of $\sym\g$  and let
$(\!(\u+\up)\!)$ denote an ideal of the algebra
$\sym\g$ generated by the subspace $\u+\up\subset\sym\g$.
We have natural identifications
$\sym\g/(\!(\u+\up)\!)=\C[(\bbe+\b_\bbf)\cap\b_\bbe]=\C[\eh]$.
The  quotient map $\sym(\g/\u)\to \sym\g/(\!(\u+\up)\!)$ corresponds, via 
the identifications, 
to the restriction map $\C[\b]\to \C[\eh]$ induced by the
imbedding $\eh\into\b$. Using that inside
$(\uh\g)\hbo=\sym\g$ one has $\uph\hbo=\up$, we
deduce  natural isomorphisms, cf. \eqref{coinv}, 
\beq{class}
\big((V\o\ver_\hb)/\uph\big)\hbo=
\big(\u\back(V\o \uh\g)/\uph\big)\hbo=V\o \sym\big(\g/(\!(\u+\up)\!)\big)
=V[\eh].
\eeq

Let $\g=\oplus_{m\in\Z}\ \g(m)$, resp. $V= \oplus_{m\in\Z}\ V(m)$, be the weight
decomposition with respect to the adjoint, resp. natural, action of the element
$\bbh\in\g$. Define a  $\Z$-grading  on
$\uh\g$, to be referred to as the {\em Kazhdan grading},
by assigning 
the elements of  $\g(m)\sset \g,\ m\in\Z$,  degree $2+m$ and  assigning $\hb$  degree $2$.
With this grading, the algebra $\uh\g$ may be identified with the Rees algebra 
of the algebra $\ug$ equipped with the Kazhdan
filtration, cf. \cite{Ko2}.
From now on, we will use the Kazhdan grading as our default grading.

The Lie algebra $\u:=\u_\bbe$ is a graded subalgebra of $\g$,
so
the grading on $\uh\g$ induces one on $\verh$.
We equip $V\o\verh$ with
the standard grading on a tensor product.
The Lie algebra  $\uph$ is a {\em graded} $\chb$-submodule  of $\uh\g$
since the elements $f_i$ have degree zero. 
Hence, the grading on $V\o\verh$ induces a quotient grading on $(V\o\ver_\hb)/\uph$,
resp. $\u\back(V\o \uh\g)/\uph$.

We also define a grading on  $(\sym\t)[\hb]$ by placing 
the vector space $\t$  in degree $2$.
Using  the  natural identifications
$\uh\t=\C[\t^*\times\BA^1]=(\sym\t)[\hb]$,
this gives a grading on $\uh\t$, resp. $\C[\t^*\times\BA^1]$.
We equip  $\zh\g$, resp.  $(\sym\t)^W[\hb]$,
 with the 
gradings induced from the one on $\uh\g$, resp. $(\sym\t)^W[\hb]$, by restriction.
The  above defined gradings  induce gradings on various other objects,
e.g. $\uh\t\o_{\chb}\zh\t$, in particular, on
 all
objects which appear in Lemma \ref{vlem}. It is immediate to check
that all maps in
the statement of lemma, as well as the  asymptotic Harish-Chandra homomorphism,
respect the gradings.

\begin{lem}\label{vlem1} 
  \vi The  natural map $\zh\g\to \End_{\uh\g}(\uh\g/\uph)
=((\uh\g)/\uph)^\bu$ is a graded
  $\chb$-algebra isomorphism.

\vii  The left  $\uh\t$-action makes
$(V\o\verh)/\uph$, 
a  free graded $\uh\t$-module of rank $\dim V$.

\viii There are natural isomorphisms of graded $(\uh\t,\zh\g)$-bimodules
\[\uh\t\o_{\zh\g}((V\o\uh\g)/\uph)^\bu\ \iso\ \u\back(V\o\uh\g)/\uph\ccong
(V\o \verh)/\uph.
\]
\end{lem}
\begin{proof} 
Part (i) is essentially
due to Kostant. In more detail, Kostant considers the Kazhdan filtration on
$\ug$, and equips $Z\g$, resp. $\ug/\up$, with induced filtrations.
The natural map $Z\g\to (\ug/\up)^\bu$ respects the filtrations
and the map $\zh\g\to ((\uh\g)/\uph)^\bu$, resp.
$\zh\g\hbo\to ((\uh\g)/\uph)^\bu\hbo$, may be identified with the induced map of
Rees, resp. associated graded, algebras.
The map of associated graded algebras is an isomorphism, by one of the main results of \cite{Ko2}.
The Kazhdan filtration on $\ug/\up$ being bounded  below,
it follows that the map of Rees algebras, that is, the map
in (i),  is also an isomorphism.

To prove (ii), let $\UU(\uph)$ be a $\chb$-subalgebra of $\uh\g$ generated
by $\uph$. Then,  $\verh$ is free over $\uh\t\o_{\C[\hb]}\UU(\uph)$,
by \cite[Theorem 4.6]{Ko2}. Following Kostant, one concludes
that  $(V\o\verh)/\uph$ is  free over $\uh\t$.
Finally, the first isomorphism in (iii) is an asymptotic analogue of Lemma \ref{trick} applied
in the case of $E:=(V\o\uh\g)/\uph$. The second   isomorphism
is an asymptotic analogue
of \eqref{coinv}.
\end{proof}
 
\subsection{Proof of Theorem \ref{vthm}}\label{vpf}
Let $\cz_\hb:=\Spec(\zh\g)$ and let $\cz_\hb\to\C$
be the map induced by the imbedding $\chb\into \zh\g$.
Below, we  use simplified notation $\nf$
for the variety obtained by  the deformation construction of Section \ref{vpf-1} applied
to $\cz_\hb$
viewed as a scheme over $\C$.
From Lemma \ref{fz-lem} and the 
identification
$(\cz_\hb)_0=\fc$ provided by the
specialization of the asymptotic Harish-Chandra isomorphism
at $\hb=0$, we deduce
\beq{nf fz}
\C[\nf]\hbo\, =\, \C[(\cz_\hb)_0]\, \cong\, \C[\T\fc]\, \cong\,  \Ga(\fc, \sym\T^*_\fc)\, \cong
\, \Ga(\fc, \sym\fz_\fc).
\eeq

% Recall the closed imbedding
% $\eps_{\nf_X}: \cz_\hb\into \nf$ that lifts the diagonal imbedding
% $\cz_\hb\into \cz_\hb\times_\C\cz_\hb$, see Section \ref{vpf-1},
% and let $p:\nf\to  \cz_\hb\times_\C\cz_\hb$ denote 
 % the canonical projection.
% Let
% $\zh\g\o_{\chb}\zh\g\xrightarrow{p^*}\C[\nf]\xrightarrow{\eps_{\nf_X}^*} \zh\g$
% be the corresponding algebra homomorphisms.

We view $\uh\t$ as a $\zh\g$-module via the
imbedding
$\hc_\hb: \zh\g\into \uh\t$
and view $\zh\g\o_{\chb}\zh\g$, resp. $\C[\nf]$, as a $\zh\g$-module via the
action of $\zh\g\o1$. Thus, using the identification
$\uh\t=\C[\t^*\times\BA^1]$,
  we have
  $\uh\t\o_{\zh\g}\C[\nf]=\C[(\t^*\times \BA^1)\times_{\cz_\hb}\nf]$.
  The gradings on $\zh\g$ and $\uh\g$ induce natural gradings on 
  $\C[\nf]$, resp. $\uh\t\o_{\zh\g}\C[\nf]$.
  Recall the canonical map $p: \nf\to \cz_\hb\o_{\BA^1}\cz_\hb$.
  The following composition
  \[(\t^*\times \BA^1)\times_{\cz_\hb}\nf\, \xrightarrow{\ \Id\times p\ }\,
  (\t^*\times \BA^1)\times_{\cz_\hb}(\cz_\hb\o_{\BA^1}\cz_\hb)\,=\,
  (\t^*\times \BA^1)\times_ {\BA^1}\cz_\hb.
  \]
induces  a  graded algebra imbedding
  $\uh\t\o_{\chb}\zh\g\into\uh\t\o_{\zh\g}\C[\nf]$.

For a $(\uh\t\o_\chb\zh\g)$-module $M$ one can view
$M\hbo$ as a  $(\C[\fc]\o\C[\fc])$-module
via the imbedding
$\th^*\o \Id: \C[\fc]\o\C[\fc]\into \C[\t^*]\o\C[\fc]$.
From  the fourth paragraph of Section \ref{vpf-1} we deduce  that if $M$
is $\hb$-torsion free and $M\hbo$ is annihilated by the elements
$z\o 1-1\o z,\, z\in \C[\fc]$,
then the action  of $\uh\t\o_{\chb}\zh\g$ on
$M$ has a unique extension to a $(\uh\t\o_{\zh\g}\C[\nf])$-action.
Furthermore, the latter action gives an action of
the abelian Lie algebra $\Ga(\t^*, \th^*\fz_\fc)$ on $M\hbo$,
by \eqref{nf fz}.

% \beq{map3}
% \uh\t\o_{\chb}\zh\g=\uh\t%\underset{\zh\g}
% \o_{\zh\g} (\zh\g\o_{\chb}
%   \zh\g) \xrightarrow{\Id\o p^*}
%   \uh\t\o_{\zh\g}\C[\nf]
 %  \xrightarrow{\,  \Id\o \eps_{\nf_X}^*\, }
 %  \uh\t\o_{\zh\g}\zh\g=\uh\t.
 %  \eeq

 % \[
% \xymatrix{
 %  \uh\t\underset{\chb}\o\zh\g=\uh\t\underset{\zh\g}\o (\zh\g\underset{\chb}\o
%   \zh\g)\ \ar[rr]^<>(0.5){\Id_{\uh\t}\o p^*}&&
%   \uh\t\underset{\zh\g}\o\C[\nf]\ar[rr]^<>(0.5){\Id_{\uh\t}\o \eps_{\nf_X}^*}&&
 %  \uh\t\underset{\zh\g}\o\zh\g=\uh\t.
% }
% \]

Fix $\mu\in \BX$ and let $\tau_{\hb,\mu}: \uh\t\to\uh\t$
be a graded algebra automorphism 
defined by the assignment $\t\ni t\mto t-\hb\mu(t)$.
Let  $\J_\mu$ be  an ideal of  the algebra $\uh\t\o_{\zh\g}\C[\nf]$ 
 generated, as a $(\uh\t\o 1)$-module, by the elements
 $\frac{\tau_{\hb,\mu}(\hc_\hb(z))\o1 -1\o z}{\hb},\, z\in\zh\g$.
 From the isomorphism $\C[\nf]/\J\iso \C[\fc\times\BA^1]$, see Section
 \ref{vpf-1},  one obtains an isomorphism $(\uh\t\o_{\zh\g}\C[\nf])/\J_\mu\iso \uh\t$,
 such that the composition $\uh\t\o_\chb\zh\g\into(\uh\t\o_{\zh\g}\C[\nf])/\J_\mu\,\to\, \uh\t$
 sends $t\o z$ to $t\cdot\tau_\mu(\hc_\hb(z))$.

 For a graded $\uh\t\o_{\zh\g}\C[\nf]$-module $M$ we put
$M^{\J_\mu}=\{m\in M\mid xm=0\ \forall x\in\J_\mu\}$. 
 The ideal $\J_\mu$ is homogeneous, so $M^{\J_\mu}$  is a graded submodule of $M$.

 Recall that the space
 $(V\o\verh)^\u$, resp. $(V\o\verh)/\uph$,
 has the structure of a (graded) $(\uh\t,\uh\t)$-bimodule,
 resp. $(\uh\t,\zh\g)$-bimodule.
 We will view  these bimodues as
 as 
 $(\uh\t\o_\chb\zh\g)$-modules.
 The specialization of  each of the two $(\uh\t\o_\chb\zh\g)$-modules
 at $\hb=0$ is annihilated by  the elements
 $z\o 1-1\o z,\, z\in \C[\fc]$, since $(V\o\uh\g)\hbo$
 is symmetric as a $(\sym\g,\sym\g)$-bimodule.
 We conclude that $(V\o\verh)^\u$,
  resp.  $(V\o\verh)/\uph$, 
has the canonical structure of a graded $\uh\t\o_{\zh\g}\C[\nf]$-module.
 
\begin{prop}\label{vprop}
For any  finite dimensional  $G$-module  $V$  and
  $\mu\in{\mathsf Q}$
one has an inclusion
$(V\o\verh)^{\u,\mu} \sset ((V\o\verh)^\u)^{\J_\mu}$.
Moreover, the composition
\beq{comh}
\kap_{V,\mu,\hb}:\ (V\o\verh)^{\u,\mu}\, \into\, ((V\o\verh)^\u)^{\J_\mu}\,\to \,\big((V\o\verh)/\uph\big)^{\J_\mu}
\eeq
is an isomorphism of graded $\uh\t\o_{\zh\g}\C[\nf]$-modules.
\end{prop}

Observe that the canonical algebra isomorphism
$\C[\nf]|_{\hb=1}\cong(\zh\g\o_\chb\zh\g)|_{\hb=1}=Z\g\o Z\g$ yields isomorphisms
$(\uh\t\o_{\zh\g}\C[\nf])|_{\hb=1}\cong \ut\o_{Z\g}(Z\g\o Z\g)=\ut\o Z\g$.
Thus,  Theorem \ref{vthm} follows from  the above
proposition  by specializing at $\hb=1$.

\begin{proof}[Proof of Proposition  \ref{vprop}]
  The proof of the first statement of the proposition is
  similar to the proof of \eqref{vhc}. Specifically,
  let $v\o m\in (V\o\verh)^{\u,\mu}$ and $t\in\t$.
  Using the equation
  $(v\o m)t=t(v\o m)-\hb\ad t(v\o m)$
  and mimicing computations in   the proof of \eqref{vhc},
  one finds that for all $z\in\zh\g$, we have
  $(v\o m)z=\tau_{\hb,\mu}(\hc_\hb(z))(v\o m)$.
Since $(V\o\verh)^{\u,\mu}$ is  a torsion free $\chb$-module,
we deduce that  this module is annihilated by the elements 
$\frac{\tau_{\hb,\mu}(\hc_\hb(z))\o1-1\o z}{\hb},\, z\in\zh\g$.
These elements generate the ideal $\J_\mu$,
and the first statement of the proposition follows.

To prove the second statement, recall that one has
graded algebra isomorphisms
$\zh\g\cong \C[\fc][\hb]\cong \C[z_1,\ldots,z_r, \hb]$,
where  $r=\dim\t$ and  $z_i,\ i=1,\ldots,r$,
are homogeneous elements of positive degree.
We identify $z_i$ with an element of $\zh\g$ and consider a diagram
\beq{comp}
0\ \to\ (V\o\verh)^{\u,\mu}\ \xrightarrow{\ \kap_{V,\mu,\hb}\ } \ (V\o\verh)/\uph\
\xrightarrow {\ \varsigma\ } \  \bigoplus_{1\leq i\leq r}\ \Big((V\o \verh)/\uph\Big)(\deg z_i-2),
\eeq
where $(\deg z_i-2)$ denotes a grading shift by $\deg z_i-2$, and the map $\varsigma$ is given
by the assignment
\[(v\o m)\ \mto\  \mbox{$\frac{\tau_{\hb,\mu}(\hc_\hb(z_1))\o1-1\o z_1}{\hb}$}(v\o m)\ \oplus\ldots\oplus\
\mbox{$\frac{\tau_{\hb,\mu}(\hc_\hb(z_r))\o1-1\o z_r}{\hb}$}(v\o m).\]

We use the notation of Section \ref{class-act}.
The map $\kap_{V,\mu,\hb}|_{\hb=0}$ may be  identified, via 
Lemma \ref{vlem}(ii) and   isomorphism \eqref{class},
with the composition
$V[\b]^{U,\mu}\into V[\b]\to V[\eh]$, where the
second map is a restriction map  induced
by the imbedding $\eh=\bbe+\g_\bbh\into\b$.
To describe the map $\varsigma\hbo$, recall that thanks to  Lemma \ref{fz-lem},
associated with $f\in\C[\fc]$,
there is a $G$-equivariant polynomial map $d(\vth^*f): \g_r\to \g$.
Let $\xi_f\in \Ga(\eh, \vth^*_\eh\fz_\fc)$ be the restriction of this map to $\eh$.
In the special case $f=z_i\hbo\in\C[\nf]\hbo= \C[\fc]$, we put
$\xi_i:=\xi_f$.

One has the following
algebra  homomorphisms
\[\C[\nf]\to \C[\nf]\hbo\,\to\,\Ga(\fc, \sym\fz_\fc)
  \xrightarrow{\ z\mto 1\o z\ }
  \C[\t^*]\o_{\C[\fc]}\Ga(\fc, \sym\fz_\fc)=\Ga(\eh, \sym(\vth^*_\eh\fz_\fc)),
\]
where the second map is an isomorphism that comes from \eqref{nf  fz} and
$\vth_\eh$ denotes the composite $\eh\into \g_r\onto\fc$.
Let $\be$ be  the composition   of the above homomorphisms.
Going through definitions, one finds
that $\be$
sends $\frac{z_i\o1-1\o z_i}{\hb}$ to $\xi_i\in
\Ga(\eh, \vth^*_\eh\fz_\fc)\subset
\Ga(\eh, \sym(\vth^*_\eh\fz_\fc))$.
It follows that the composition
\begin{align*}\uh\t\o_{\zh\t}\C[\nf]\xrightarrow{\ \ }
  (\uh\t \o_{\zh\t}\C[\nf])\hbo\ \xrightarrow{\Id\times \be}
                                   \ 
                                   \C[\t^*]\o_{\C[\fc]}\Ga(\eh, \sym(\vth^*_\eh\fz_\fc))
  \xrightarrow{\vkap_\eh}\
  \Ga(\eh, \sym(\g_\bbh) 
\end{align*}
sends the element
$\frac{\tau_{\hb,\mu}(\hc_\hb(z_i))\o1-1\o z_i}{\hb}$
to $\vkap_\eh(\xi_i)-\mu(\vkap_\eh(\xi_i))$.
The symbol $\Id$ above stands for the identification
$\uh\t\hbo=\C[\t^*\times\BA^1]=\C[\t^*]$, and
the map $\vkap_\eh$ was defined in  Section \ref{class-act}.
Combining this with formula \eqref{poissm}, it is not difficult to show that
the map $(V\o\verh)/\uph\to (V\o\verh)/\uph$ 
given by the assignment $(v\o m) \mto
\frac{\tau_{\hb,\mu}(\hc_\hb(z_i))\o1-1\o z_i}{\hb}(v\o m)$, 
 specializes at $\hb=0$ to the map
\[V[\eh]\to V[\eh],\quad \vv\mto \xi_i(\vv)-\mu(\vkap_\eh(\xi_i))\cdot\vv.\]

Thus, from the definition of the map $\varsigma$ in \eqref{comp}, resp.
vector space $V[\eh]^{\{\mu\}}$ in \eqref{f-act}, we deduce
\[\ker(\varsigma|_{\hb=0})\ \overset{(1)}=\ V[\eh]^{\{\mu\}}\ \overset{(2)}=\
  \im(\kap_{V,\mu})\ \overset{(3)}=\ 
\im(\kap_{V,\mu,\hb}|_{\hb=0}),
\]
where
(1), resp. (3),  follows from the description of the map $\varsigma|_{\hb=0}$,
resp. $\kap_{V,\mu,\hb}\hbo$, given  above and (2)   holds by Proposition \ref{vthm-cl}.

To complete the proof we observe that
the maps $\kap_{V,\mu,\hb}$ and $\varsigma$ in \eqref{comp} are morphisms of finite
rank free graded $\uh\t$-modules,
by  Lemma \ref{vlem} and Lemma \ref{vlem1}.
The first statement of the proposition implies that  $\im(\kap_{V,\mu,\hb})\subseteq\Ker(\varsigma)$.
Furthermore, we have shown that
 the specialization of  
 diagram  \eqref{comp} at $\hb=0$
gives an exact sequence.
Thus, the  second statement of the proposition  is a consequence of
a standard general semicontinuity result stated in the following lemma.
\end{proof}

\begin{lem}\label{k-seq} Let\ $0\to E'\xrightarrow{a} E \xrightarrow{b} E''$\ be a sequence of morphisms of free
$\Z$-graded $\uh\t$-modules of finite rank such  that $b\ccirc a=0$. If
the induced sequence\ $0\to E'\hbo\xrightarrow{\bar a} E\hbo\xrightarrow{\bar b} E''\hbo$\
is exact, then the original  sequence is also exact.
\end{lem}
\begin{proof}[Proof of Lemma]   Let $p: E\to E\hbo$ denote the projection. From definitions, we find
  $p\inv(\im(\bar a))=\im(a) +\hb E$, resp. $p\inv(\Ker(\bar b))=b\inv(\hb E'')$.
  Thus, we have an equality
  \[\im(a) +\hb E=b\inv(\hb E''),\]
  since $\im(\bar a)=\Ker(\bar b)$ by the assumptions of the lemma.

    Assume by contradiction
    that $\im(a)\subsetneq \Ker(b)$.
    All the  modules involved are  graded and  the gradings
  on these modules  are bounded below, since  the algebra $\uh\t$ is
  graded by nonnegative integers and the modules are  finitely generated.
      Let $e$ be a homogeneous element
  of $E$ of minimal degree such that $e\in \Ker(b)\setminus\im(a)$. Since $e\in b\inv(b(e))=b\inv(0)
  \subseteq b\inv(\hb E'')$,  the displayed
  equation above implies that there are homogeneous elements
  $e'\in E'$ and $x\in E$ such that $e=a(e')+\hb x$.
  We compute $\hb\, b(x)=b(\hb x)=b(e-a(e'))=b(e)-b(a(e'))=0-0=0$.
  The module $E''$ being free, hence $\hb$-torsion free,
  we deduce that $x\in \Ker(b)$. Since $\deg x=\deg e-\deg\hb<\deg e$, 
  we must have  $x\in \im(a)$, by our choice of the element $e$.
    Writing $x=a(y)$, we obtain
    $e=a(e')+\hb x=a(e')+\hb\,  a(y)=a(e'+\hb y)$.
    Thus, $e\in\im(a)$, a contradiction.  
  \end{proof}
\subsection{Proof of Theorem \ref{grthm}}\label{sec6}
Below, we are going to apply  results of  \cite{BF} and \cite{GR}.
  The notation used in these papers is slightly
  different from ours
  due to a different normalization of the Harish-Chandra homomorphism.
  Specifically,  the Harish-Chandra homomorphism
  used in {\em  loc cit}
  is defined as a composition $\hc'_\hb:=\tau_\rho\ccirc\hc_\hb$.
  The
   resulting  imbedding $\hc'_\hb:\zh\g\into\uh\t$
 yields an isomorphism $\zh\g\cong
 (\uh\t)^W=\C[\t^*\times\BA^1]^W=\C[\fc][\hb]$
 (recall that $\rho$ denotes  the half-sum of positive roots and
 $W$-invariants are taken with respect to  
 the dot-action
  of $W$). 
We view  $\uh\t$  as a $\zh\t$-algebra
via the imbedding ~$\hc'_\hb$ and  write $\nt$ for the corresponding algebra
$\uh\t\o_{\zh\g}\C[\nf]$.

We use the setting of Section \ref{satake-intro}.
Thus, we have the groups $\bt=\Gm \times T^\vee \subset \Gm\ltimes\GO$,
and there are
canonical isomorphisms
$H^\hdot_\bt(\pt)=\C[\t^*\times\BA^1]=\uh\t$, resp.
$H^\hdot_{\Gm\ltimes\GO}(\pt)=
H^\hdot_{\Gm\times  G^\vee}(\pt)=\C[\fc][\hb]
=\zh\g=\C[\cz_\hb]$.
This yields 
 natural graded algebra maps
 \begin{align*}
   \C[\cz_\hb\times_{\BA^1}\cz_\hb]=\C[\fc]\o\C[\fc]\o\chb\ &\ \iso H^\hdot_{G^\vee}(\pt)\o H^\hdot_{G^\vee}(\pt)\o H^\hdot_\Gm(\pt)\\
  &\ \to  H^\hdot_{\Gm}\big(\GO\backslash G^\vee(\KK)/\GO\big)\ \iso\ 
    H^\hdot_{{\mathbb G}_m\times G^\vee}(\Gr).
    \end{align*}

It was shown in \cite[Theorem 1]{BF} that the composite of  the above maps
has a unique extension to a graded algebra isomorphism $\C[\nf]\iso H^\hdot_{{\mathbb G}_m\times\Gv}(\Gr)$.
We deduce $\uh\t$-algebra isomorphisms
\[%\beq{HGr}
\nt=\uh\t\o_{\zh\g}\C[\nf]\ccong 
\C[\t^*\times\BA^1]
\o_{\C[\fc\times\BA^1]}H^\hdot_{{\mathbb G}_m\times G^\vee}(\Gr)
\ccong H^\hdot_\bt(\Gr).
\]
Below, we identify the algebras $H^\hdot_\bt(\Gr)$ and $\nt$.

Associated with $\la\in\BX$,  there is a graded algebra homomorphism
$\eta_\la: \nt\to\uh\t$,
a $\rho$-shifted version of the
homomorphism
$\uh\t\o_{\zh\g}\C[\nf]\to\uh\t$
considered in Section \ref{vpf},
such that the composition $\uh\t\o_\chb\zh\g \into \nt
\xrightarrow{\eta_\la}\uh\t$
sends $t\o z$ to $t\cdot \hc'_\hb(z)$.
The  kernel of $\eta_\la$ is 
an ideal $\jj_\la$ of $\nt$
 generated by the elements $\frac{\tau_{\hb,\la}(\hc'_\hb(z))-z}{\hb},
 \, z\in\zh\g$. This ideal is a $\rho$-shifted counterpart of the
 ideal $\J_\la$ generated by the elements
 $\frac{\tau_{\hb,\la}(\hc_\hb(z))-z}{\hb}$, see Section \ref{vpf}.

 Recall the imbedding $i_\la: \{\pt_\la\} \into \Gr$,
 where $\pt_\la$ is  the  $\bt$-fixed point associated with  $\la$.
It follows from  \cite[\S 3.2]{BF}  that  the restriction map $i_\la^*:  H^\hdot_\bt(\Gr)\to
H^\hdot_\bt(\pt_\la)$ corresponds,
via the  isomorphisms above,
to the   homomorphism $\eta_\la$.
Thus, the notation $\jj_\la$
agrees with  the notation  $\jj_\la=\Ker(i^*_\la)$
used in  Section  \ref{satake-intro}. 
We conclude that there are
canonical graded algebra isomorphisms
\beq{Hmod}
H^\hdot_\bt(\pt_\la)\ccong  H^\hdot_\bt(\Gr)/\jj_\la\ccong
\nt/\jj_\la\ccong \uh\t.
\eeq

For any object  $\cf$  of  the $\bt$-equivariant
constructible derived category of $\Gr$,
the cohomology  $H_\bt^\hdot(\Gr,\cf)$
has the natural structure of a graded $H_\bt^\hdot(\Gr)$-module.
The $H^\hdot_\bt(\Gr)$-action on
$H_\bt^\hdot((i_\la)_!i_\la^!\cf)$
 factors through the quotient 
 $i_\la^*: H^\hdot_\bt(\Gr)\to H^\hdot_\bt(\pt_\la)$,
  since $\supp((i_\la)_!i_\la^!\cf)\sset\{\pt_\la\}$.
Using  \eqref{Hmod}, this translates
into the statement that the image of the
canonical $\nt $-module map $(i_\la)_!: H_\bt^\hdot(i_\la^!\cf) \to
H^\hdot_\bt(\Gr,\cf)$ induced by the adjunction $(i_\la)_!i_\la^!\cf\to \cf$, 
 is annihilated by the ideal $\jj_\la\subset\nt $.

 Now let $\cf$ be an object of the Satake category $\sat$.
 It is known that the group $H_\bt^\hdot(i_\la^!\cf)$, resp.
$H^\hdot_\bt(\Gr,\cf)$, viewed as a graded module over the
subalgebra
$\uh\t\sset \nt$, is a 
free graded module of finite rank;
furthermore, the localization theorem in equivariant cohomology
implies that the map $(i_\la)_!: H_\bt^\hdot(i_\la^!\cf) \to
H^\hdot_\bt(\Gr,\cf)$
is injective, cf. eg. \cite[\S 6.1-6.2]{GR}.

Combining all the above, we see that
proving Theorem \ref{grthm} is equivalent to
showing that the injective map $(i_\la)_!: H_\bt^\hdot(i_\la^!\cf)
\into (H^\hdot_\bt(\Gr,\cf))^{\jj_\la}$
is a bijection.
Thus, writing $P(M)=\sum_{k\in \Z} \dim M_k \cdot t^k$
for the Poincar\'e series of a $\Z$-graded vector space $M=\oplus_{k\in \Z}\, M_k$,
we are reduced to proving the following equality
of  Poincar\'e series:
\beq{peq}
P\big(H_\bt^\hdot(i_\la^!\cf)\big)=P\big((H^\hdot_\bt(\Gr,\cf))^{\jj_\la}\big).
\eeq

To prove this equality, we interpret each side in terms of 
geometric Satake. To this end, for each $\mu\in \BX$
we define a  $(\uh\t,\uh\g)$-bimodule $\verh(\mu)$,
a $\mu$-twisted analogue  of  $\verh$, as follows.
The bimodule $\verh(\mu)$ has
the same underlying vector space as $\verh$
and the same  right  $\uh\g$-action. The left
action of $t\in\t\subset\uh\t$ on  $\verh(\mu)$ is defined in terms of the
 left action of $t$ on  $\verh$ by the formula
$t: m\mto  (t-\hb\mu(t))m$.  Put $\la:=\mu-\rho$.
 Then, the  isomorphism $\kap_{V,\mu,\hb}$
 of Proposition \ref{vprop}  translates into an isomorphism
 $(V\o \verh(\rho))^{\u,\la}\iso \big((V\o\verh(\rho)\big)^{\jj_\la}$
 of graded $\nt$-modules.

Let $\nt  \grmod$ be the abelian category of $\Z$-graded $\nt $-modules,
$\rep(G)$ the  category of finite dimensional $G$-representations,
and $\BS: \rep(G) \iso \sat$  the geometric Satake equivalence.
According to \cite[Theorem 6]{BF}, there is an isomorphism 
\beq{BFu} ((-)\o\verh(\rho))/\uph\ccong H_\bt^\hdot\big(\Gr,\BS(-)\big),
\eeq
of  functors 
$\rep(G)\to \nt  \grmod$.
Hence, for $\cf=\BS(V)$ and any $\la\in\BX$,  there is an isomorphism
$\big((V\o\verh(\rho))/\uph\big)^{\jj_\la}=\big(H_\bt^\hdot(\Gr,\cf)\big)^{\jj_\la}$.
On the other hand, it  follows from
\cite[Theorem 2.2.4]{GR}, 
 cf. also \cite[Remark 2.2.6]{GR},  that
the graded $\uh\t$-module $H_\bt^\hdot(i_\la^!\cf)$ is isomorphic to
$(V\o \verh(\rho+\la))^{\b}$, where $\b$ is the Borel that contains $\u$.
It is clear that  we have $(V\o \verh(\rho+\la))^{\b}= (V\o \verh(\rho))^{\u,\la}$.
Thus, combining the above isomorphisms of graded
modules we obtain the following chain of equalities of
the corresponding Poincar\'e series:
%\[
 % P\big(H_\bt^\hdot(i_\la^!\cf)\big)\stackrel{\text{\cite{GR}}}\eqq
 % \xymatrix{P\big((V\o \verh(\rho))^{\u,\la}\big)
%\ar@{=}[r]^<>(0.5){\ \text{Prop. \ref{vprop}}\ }&
%P\big(((V\o\verh(\rho))/\uph)^{\jj_{\la}}\big)
%}
%\stackrel{\text{\cite{BF}}}\eqq
%P\big(H_\bt^\hdot(\Gr,\cf)^{\jj_\la}\big).
%\]
\begin{align*}
    P\big(H_\bt^\hdot(i_\la^!\cf)\big)
    &\xymatrix{
\ar@{=}[rr]^<>(0.5){\text{\cite{GR}}}&&
                                        P\big((V\o \verh(\rho))^{\u,\la}\big)}
  \\
  &\xymatrix{\ar@{=}[rr]^<>(0.5){\text{Prop. \ref{vprop}}}
    &&
    P\big((V\o\verh(\rho))/\uph)^{\jj_{\la}}\big)
    \ar@{=}[rr]^<>(0.5){\text{\cite{BF}}}&&P\big(H_\bt^\hdot(\Gr,\cf)^{\jj_\la}\big).
    }
\end{align*}

This proves \eqref{peq}, and Theorem \ref{grthm} follows.
\qed

\bibliographystyle{plain}

\begin{thebibliography}{DKS}
\bibitem[Be]{Be} A. Beauville, {\em Symplectic singularities.} Invent. Math. {\textbf {139}} (2000), 541-549.
\bibitem[BB]{BB} A. Beilinson, J. Bernstein, {\em A proof of Jantzen conjectures.}
  I. M. Gelfand Seminar, 1-50, Adv. Soviet Math., {\textbf {16}}, Part 1, Amer. Math. Soc., Providence, RI, 1993. 
\bibitem[BK]{BK} A. Beilinson, D. Kazhdan,
{\em Flat projective connections}.
Unpublished manuscript, 1991; available at
\vskip -2mm
\begin{verbatim}http://www.math.sunysb.edu/~kirillov/manuscripts.html
\end{verbatim}
\bibitem[BGG]{BGG} J. Bernstein, I. Gelfand, S. Gelfand, {\em
Differential operators on the base affine space and a study of $\g$-modules.}
 Lie groups and their representations (Proc. Summer School, Bolyai J\'anos Math. Soc., Budapest, 1971),
 pp. 21-64. Halsted, New York, 1975. 
\bibitem[BBP]{BBP}
  R.~Bezrukavnikov, A.~Braverman, L.~Positselskii, \emph{Gluing of abelian categories and differential operators on the basic affine space}.
  J. Inst. Math. Jussieu {\textbf 1} (2002), 543--557. 
\bibitem[BF]{BF} R.~Bezrukavnikov,
 M.~Finkelberg, \emph{Equivariant Satake category and
Kostant--Whittaker reduction}.   Mosc. Math. J. {\textbf 8} (2008),  39-72.


\bibitem[BP]{BP} R.~Bezrukavnikov, A. Polishchuk, {\em
Gluing of perverse sheaves on the basic affine space.}
arXiv:math/9811155.

% I. Mirkovi\'c,
%{\em Equivariant homology and $K$-theory of affine Grassmannians and Toda
%  lattices.}
%  Compos. Math.  141  (2005),   746-768.


 
\bibitem[BR]{BR}
R.~Bezrukavnikov, S.~Riche, \emph{Affine braid group actions on Springer
resolutions}.    Ann. Sci. \'Ec. Norm. Sup\'er. (4) {\textbf {45}} (2012),  535-599.

% \bibitem[BrF]{BrF} A. Braverman, M. Finkelberg, 
% {\em  Dynamical Weyl groups and equivariant cohomology of transversal slices on affine Grassmannians.}
%  Math. Res. Lett. {\textbf {18}} (2011),  505-512. 
 \bibitem[CG]{CG} N. Chriss, V. Ginzburg, {\it
 Representation theory and complex geometry.} 
Birkh\"auser Boston, 1997.

\bibitem[Di]{Di} J. Dixmier,
{\em Alg\`ebres enveloppantes.}  Cahiers Scientifiques, 
Fasc. XXXVII. Gauthier-Villars \'Editeur, Paris-Brussels-Montreal, Que., 1974.

\bibitem[DG]{DG}  R.
Donagi, D. Gaitsgory, {\em 
The gerbe of Higgs bundles.}
Transform. Groups {\textbf  7} (2002), 109-153. 
\bibitem[DKS]{DKS}
  A. Dancer, F. Kirwan,  A. Swann, {\em Implosion for hyperkähler manifolds}.  
  Compos. Math. 149 (2013),
1592-~1630.
\bibitem[Fu]{Fu}
B. Fu, {\em Symplectic resolutions for nilpotent orbits.} Invent. Math. {\textbf  {151}}
(2003), 167-186.
\bibitem[GG]{GG} W.-L. Gan, V. Ginzburg, {\em
Quantization of Slodowy slices.} Int. Math. Res. Not. 2002, no. 5, 243-255. 

\bibitem[Gi1]{Gi1} V. Ginzburg, \emph{Perverse sheaves and $\C^*$-actions}.  
J. Amer. Math. Soc. {\textbf 4} (1991), 483--490.
\bibitem[Gi2]{Gi2} V. Ginzburg, 
 {\em   Nil Hecke algebras and Whittaker $D$-modules.}
 Lie Groups, Geometry, and Representation Theory,
 137-184,  Progr. Math. {\textbf  {327}}, 
Birkh\"auser, 2018. 

\bibitem[Gi3]{Gi3} V. Ginzburg, \emph{Characteristic varieties
    and vanishing cycles.} Invent. Math. {\textbf  {84}} (1986),  327–402. 

  \bibitem[GR]{GR} V. Ginzburg, S.~Riche, {\em Differential operators on $G/U$ and the affine Grassmannian.}
 J. Inst. Math. Jussieu {\textbf {14}} (2015),  ~493-575. 

%\bibitem[Hi]{Hi} V. Hinich, {\em On the singularities of nilpotent orbits.} 
%Israel J. Math. {\textbf  {73}} (1999), 297-308.

\bibitem[J]{J} B. Jia, 
{\em The Geometry of the affine closure of $T^*(SL_n/U)$.}
arXiv:2112.08649, 2021.

\bibitem[Kal]{Kal} D. Kaledin, {\em Symplectic singularities from the Poisson point of view.}
 J. Reine Angew. Math. {\textbf{600}} (2006), 135-156.
\bibitem[Ka]{Ka} D. Kazhdan, {\em ``Forms'' of the principal series for $GL_n$.}
Functional analysis on the eve of the 21st century,     
 vol 1,   
 153-171, Progr. Math. {\textbf {131}}, Birkh\"auser Boston, Boston, MA, 1995.
\bibitem[KL]{KL}
D. Kazhdan,  G.~Laumon, \emph{Gluing of perverse sheaves and discrete series representation}.   
J. Geom. Phys. {\textbf  5} (1988), ~{63-120.} 

\bibitem[Ko1]{Ko1} B. Kostant, {\em Lie group representations on
polynomial rings.}
 Amer. J. Math.  {\textbf {85}}  (1963) 327-404. 
% \bibitem[Ko2]{Ko2} \bysame, {\em
% The solution to a generalized Toda lattice and representation theory.}
%   Adv. in Math.  {\textbf {34}}  (1979), 195-338. 
\bibitem[Ko2]{Ko2}  B. Kostant,  {\em On Whittaker vectors and
representation theory}.   Invent. Math. {\textbf {48}} (1978), 101-184.
\bibitem[LS]{LS} T. Levasseur, J. T. Stafford, {\em
 Differential operators and cohomology groups on the basic affine space.}
Studies in Lie theory, 377-403, Progr. Math. {\textbf {243}}, Birkh\"auser Boston, Boston, MA, 2006.
% \bibitem[MV]{MV} 
%   I. Mirkovi\'c, K. Vilonen,
%   {\em Geometric Langlands duality and representations of algebraic groups over commutative rings.}
%   Ann. of Math. (2) {\textbf {166}} (2007),  95–143. 

\bibitem[Na]{Na} H. Nakajima, {\em Instantons on ALE spaces, quiver varieties, and Kac-Moody algebras.} Duke Math. J. 76 (1994),  365-416.

\bibitem[Ngo]{Ngo} B.C. Ngo, {\em
 Le lemme fondamental pour les alg\`ebres de Lie.} 
Publ. Math. Inst. Hautes \'Etudes Sci. {\textbf {111}} (2010), 1-169.
%   \bibitem[Pa]{Pa} D.  Panyushev,
%   {\em Rationality of singularities and the Gorenstein property for nilpotent
 %   orbits.} Funct. Anal. Appl. {\textbf  {25}} (1991) 225-~226.


\bibitem[SS]{SS} T.A. Springer, R. Steinberg, {\em
Conjugacy classes.} 1970 Seminar on Algebraic Groups and Related Finite Groups.
Lecture Notes in Mathematics, {\textbf  {131}},  pp. 167–266.


\bibitem[W]{W} X. Wang, {\em
    A new Weyl group action related to the quasi-classical Gelfand-Graev action.}
 Selecta Math.  27 (2021),  38.
\end{thebibliography}

\end{document}